\documentclass[11pt,a4paper]{article}
\usepackage[autolanguage]{numprint}
\usepackage[english]{babel}
\usepackage{amsmath}
\usepackage{preview}
\usepackage{amsfonts}
\usepackage{amssymb}
\usepackage{amsthm}
\usepackage{mathrsfs}
\usepackage{dsfont}
\usepackage{paralist}
\usepackage{natbib}
\usepackage{bm}
\usepackage{color}
\usepackage[colorlinks=true,linkcolor=blue,citecolor=blue,pdfborder={0 0 0}]{hyperref}
\usepackage{graphicx}
\usepackage{tabularx}
\usepackage[left=3.5cm,right=3.5cm,top=2cm,bottom=2cm,includeheadfoot]{geometry}
\usepackage{comment}
\usepackage{tikz}
\usetikzlibrary{decorations.pathreplacing}
\usepackage{resizegather}
\usepackage{grffile}
\usepackage{graphicx}
\usepackage{caption}
\usepackage{subcaption}
\usepackage{placeins}
\DeclareFontFamily{OT1}{pzc}{}
\DeclareFontShape{OT1}{pzc}{m}{it}{<-> s * [1.10] pzcmi7t}{}
\DeclareMathAlphabet{\mathpzc}{OT1}{pzc}{m}{it}

\usepackage{xspace}

\usepackage{enumitem}
\SetLabelAlign{LeftAlignWithIndent}{\hspace*{2.0ex}\makebox[1.5em][l]{#1}}

\newcount\Comments  
\newcommand{\kibitz}[2]{\ifnum\Comments=1\textcolor{#1}{#2}\fi}

\newcommand\Item[1][]{%
  \ifx\relax#1\relax  \item \else \item[#1] \fi
  \abovedisplayskip=0pt\abovedisplayshortskip=0pt~\vspace*{-\baselineskip}}
\newcommand{\GG}[1]{}

\newtheoremstyle{normal}
{2ex}               
{3ex}               
{}                  
{}                  
{\bfseries} 
{}                  
{2pt}   
{\thmname{#1}\thmnumber{ #2.} \thmnote{(#3)}}

\newtheoremstyle{italic}
{2ex}
{3ex}
{\itshape}
{}
{\bfseries} 
{}
{2pt}
{\thmname{#1}\thmnumber{ #2.} \thmnote{(#3)}}

\theoremstyle{normal}
\newtheorem{definition}{Definition}[section]
\newtheorem{remark}[definition]{Remark}
\newtheorem{example}[definition]{Example}
\newtheorem{condition}[definition]{Condition}

\theoremstyle{italic}
\newtheorem{theorem}[definition]{Theorem}
\newtheorem{lemma}[definition]{Lemma}
\newtheorem{proposition}[definition]{Proposition}
\newtheorem{corollary}[definition]{Corollary}

\newcommand\R{\mathbb{R}}

\DeclareMathOperator*{\esssup}{ess\,sup}

\allowdisplaybreaks[3]

\begin{document}

\title{Laws of large numbers for Hayashi-Yoshida-type functionals}

\author{Ole Martin and Mathias Vetter\thanks{Christian-Albrechts-Universit\"at zu Kiel, Mathematisches Seminar, Ludewig-Meyn-Str.\ 4, 24118 Kiel, Germany.
{E-mail:} martin@math.uni-kiel.de/vetter@math.uni-kiel.de} \bigskip \\
{Christian-Albrechts-Universit\"at zu Kiel}
}

\maketitle

\begin{abstract}
In high-frequency statistics and econometrics sums of functionals of increments of stochastic processes are commonly used and statistical inference is based on the asymptotic behaviour of these sums as the mesh of the observation times tends to zero. Inspired by the famous Hayashi-Yoshida estimator for the quadratic covariation process based on two asynchronously observed stochastic processes we investigate similar sums based on increments of two asynchronously observed stochastic processes for general functionals. We find that our results differ from corresponding results in the setting of equidistant and synchronous observations which has been well studied in the literature. Further we observe that in the setting of asynchronous observations the asymptotic behaviour is not only determined by the nature of the functional but also depends crucially on the asymptotics of the observation scheme.  
\end{abstract}


\smallskip



\section{Introduction}
\def\theequation{1.\arabic{equation}}
\setcounter{equation}{0}

In the field of high-frequency statistics inference on properties of a multidimensional stochastic process $(X_t)_{t \geq 0}$ in continuous time is based on discrete observations $X_{t_{i,n}}$, and asymptotics are studied as the mesh of the observation times $t_{i,n}$, $i \in \mathbb{N}_0$, tends to zero. An important class of statistics in this field are sums of some functional $f$ evaluated at the increments of the stochastic process $X$ over the observation intervals, so 
\begin{align}\label{intro1}
\sum_{i:t_{i,n}\leq T} f(X_{t_{i,n}}-X_{t_{i-1,n}}).
\end{align}  
Such sums occur frequently, both as statistics whose asymptotic behaviour is used for inference and as estimators for asymptotic variances in central limit theorems. For example, if $X$ is a semimartingale and if $f(x)=x_k x_m$ holds it is well known that \eqref{intro1} converges to the quadratic covariation $[X^{(k)},X^{(m)}]_T$ as the mesh of the observation times tends to zero, where $X^{(l)}$ denotes the $l$-th component of $X$; compare e.g.\ Theorem 23 in \cite{Pro04}. For this reason the understanding of the asymptotic behaviour of such sums is of great interest to statisticians working in high-frequency statistics. 

The asymptotics of \eqref{intro1} in the setting of synchronous observation times are well understood. An overview of relevant results can be found in \cite{JacPro12}. However, in applications where a multidimensional stochastic process is investigated different components $X^{(l)}$ of $X$ are usually observed at different times $t_{i,n}^{(l)}$. As methods designed for synchronous data are widely available, people often artificially synchronize the data to be able to use existing results; compare e.g.\ the concept of refresh times in \cite{Barnetal11} and \cite{ASFaXi10}. Not only does this lead to a reduction of the number of used data points and therefore to less efficient methods, but it might also lead to inconsistent estimation as has been shown empirically in \cite{epp79} and has been demonstrated mathematically in \cite{HayYos05}. Therefore it is advantegeous to work with the asynchronous observations directly for which analogous expressions to \eqref{intro1} are needed which solely rely on the increments 
$$X^{(l)}_{t_{i,n}^{(l)}}-X^{(l)}_{t_{i-1,n}^{(l)}}, \quad l=1,\ldots,d.$$ 
In the following we restrict ourselves to the case $d=2$ and observe a bivariate It\^o semimartingales $X$. In order to work directly with asynchronous observations the expression
\begin{align}\label{intro2}
\sum_{i,j:t_{i,n}^{(1)} \vee t_{j,n}^{(2)} \leq T} f(X^{(1)}_{t_{i,n}^{(1)}}-X^{(1)}_{t_{i-1,n}^{(1)}},X^{(2)}_{t_{j,n}^{(2)}}-X^{(2)}_{t_{j-1,n}^{(2)}})\mathds{1}_{\{(t_{i-1,n}^{(1)},t_{i,n}^{(1)}] \cap  (t_{j-1,n}^{(2)},t_{j,n}^{(2)}] \neq\emptyset\}}
\end{align}
was introduced in \cite{HayYos05} for the specific choice $f(x_1,x_2)=x_1x_2$, and they have shown in the continuous case that it converges in probability to $[X^{(1)},X^{(2)}]_T$ as the mesh of the observation times tends to zero. In the case of additional jumps the same result is shown in \cite{BibVet15}, and corresponding central limit theorems are given in \cite{HayYos08} and \cite{BibVet15} as well. In \cite{MarVet18a, MarVet18b} statistical tests for the presence of common jumps in two asynchronously observed stochastic processes are developed based on \eqref{intro2} for $f(x_1,x_2)=(x_1)^2(x_2)^2$. The purpose of this paper is to study the asymptotics of sums of the form \eqref{intro2} for general functions $f:\mathbb{R}^2 \rightarrow \mathbb{R}$ and general asynchronous observation schemes. We allow the process $X$ to be an arbitrary It\^o semimartingale, which is a very general model for stochastic processes, and the observation times $t_{i,n}^{(l)}$, $i \in \mathbb{N}_0$, to be modeled by arbitrary increasing sequences of stopping times for $l=1,2$. 

Our results are surprisingly different compared to the convergence of \eqref{intro1} in the setting of synchronous observation times. Of course, there are some similarities, namely that the limiting variables only depend on the pure jump part of $X$ (with the exception of the quadratic covariation case $f(x_1,x_2)=x_1x_2$) and that they are of the form $\sum_{s \leq T} f(\Delta X_s^{(1)},\Delta X_s^{(2)})$. Also, we need that $f(x_1,x_2)$ vanishes sufficiently fast as $(x_1,x_2)\rightarrow 0$. However, the situation in the asynchronous setting is much more restrictive in both aspects: In general, we obtain convergence of \eqref{intro2} only for functions $f$ where the limit consists solely of common jumps, and $f$ not only has to vanish much faster in a neighbourhood of zero than in the synchronous setting, but it also needs to vanish around the axes of $\R^2$. The conditions on $f$ for \eqref{intro2} to converge can be weakened if we apply further restrictions to the observation scheme which proves that the convergence of \eqref{intro2} in the asynchronous setting does not only depend on the function $f$ but also on the asymptotic behaviour of the observation scheme. These so-called non-normalized functionals \eqref{intro2} may be used to derive inference about the structure of the common jumps of $X^{(1)}$ and $X^{(2)}$. 

In the second part of this paper we derive inference on the common structure of the continuous parts of $X^{(1)}$ and $X^{(2)}$, for which we investigate normalized functionals of the form
\begin{align}\label{intro3}
n^{p/2-1}\sum_{i,j:t_{i,n}^{(1)} \vee t_{j,n}^{(2)} \leq T} f(X^{(1)}_{t_{i,n}^{(1)}}-X^{(1)}_{t_{i-1,n}^{(1)}},X^{(2)}_{t_{j,n}^{(2)}}-X^{(2)}_{t_{j-1,n}^{(2)}})\mathds{1}_{\{(t_{i-1,n}^{(1)},t_{i,n}^{(1)}] \cap  (t_{j-1,n}^{(2)},t_{j,n}^{(2)}] \neq\emptyset\}},
\end{align}
$p \geq 0$. Here $n^{-1}$ represents the rate by which the length of the \glqq average\grqq \space observation interval decreases as $n$ tends to infinity. Contrary to the jump part the continuous martingale part of $X$ is not scale-invariant but the increments scale with the square root of the length of the observation interval. Therefore we normalize the sum with the factor $n^{p/2-1}$, where $p$ depends on the function $f$, in order to obtain convergence to a limit which depends on the continuous martingale part. As in the setting of synchronous observation times we achieve convergence for positively homogeneous functions $f$. However we are only able to state the explicit form of the limit for specific positively homogeneous functions $f$ as it is in general not possible to disentagle the contribution of the continuous martingale part and the contribution of the asymptotics of the observation scheme in the limit.

The remainder of this paper is organized as follows: In Section \ref{sec:framework} we introduce the mathematical model for the process $X$ and the observation schemes and discuss necessary structural conditions. In Section \ref{sec:nonnormFunc} we investigate under which conditions on $f$ and the observation scheme we obtain convergence of the non-normalized functionals \eqref{intro2}. In Section \ref{sec:normFunc} we similarly examine the asymptotics of the normalized functionals \eqref{intro3}. Throughout the paper we draw comparisons to the corresponding results in the setting of synchronous observation times and try to point out where additional challenges arise due to the asynchronicity of the observation scheme and how to deal with them. In Section \ref{sec:outlook} we discuss directions of future research. All proofs are gathered in Section \ref{sec:proofs}.

\section{Framework and motivation}\label{sec:framework}
\def\theequation{2.\arabic{equation}}
\setcounter{equation}{0}

Throughout this paper we consider the following model for the process and the observation times: Let $X =(X^{(1)},X^{(2)})$ be a two-dimensional It\^o semimartingale on $(\Omega,\mathcal{F},\mathbb{P})$ of the form
\begin{multline}
\label{ItoSemimart}
X_t = X_0 + \int \limits_0^t b_s ds + \int \limits_0^t \sigma_s dW_s + \int \limits_0^t \int \limits_{\R^2} \delta(s,z)\mathds{1}_{\{\|\delta(s,z)\|\leq 1\}} (\mu - \nu)(ds,dz) \\
+ \int \limits_0^t \int \limits_{\R^2}  \delta(s,z) \mathds{1}_{\{\|\delta(s,z)\|> 1\}} \mu(ds,dz),
\end{multline}
where $W$ is a two-dimensional standard Brownian motion, $\mu$ is a Poisson random measure on $\mathbb{R}^+ \times\mathbb{R}^2$, whose predictable compensator satisfies \mbox{$\nu(ds,dz)=ds \otimes \lambda(dz)$} for some $\sigma$-finite measure $\lambda$ on $\mathbb{R}^2$ endowed with the Borelian $\sigma$-algebra. $b$ is a two-dimensional adapted process, $\sigma$ is a $2 \times 2$ adapted process of the form
\begin{align*} 
\sigma_s= \begin{pmatrix}
\sigma_s^{(1)} &0
\\ \rho_s \sigma_s^{(2)} &\sqrt{1-\rho_s^2} \sigma_s^{(2)}
\end{pmatrix}
\end{align*}
for non-negative adapted processes $\sigma_s^{(1)}$, $\sigma_s^{(2)}$ and an adapted process $\rho_s$ with values in the interval $[-1,1]$. $\delta$ is a two-dimensional predictable function on $\Omega \times \mathbb{R}^+ \times \mathbb{R}^2$. We write $\Delta X_s=X_s-X_{s-}$ with $X_{s-}=\lim_{t \nearrow s} X_t$ for a possible jump of $X$ in $s$. By $\| \cdot \|$ we will always denote the Euclidean norm.

The processes $X^{(l)}$, $l=1,2$, are observed at times $t_{i,n}^{(l)}$, $l=1,2$, and we denote the observation scheme by
\[
\pi_n =\big\{\big(t_{i,n}^{(1)} \big)_{i \in \mathbb{N}_0},\big(t_{i,n}^{(2)} \big)_{i \in \mathbb{N}_0}  \big\}, \quad n \in \mathbb{N},
\]
where $\big(t_{i,n}^{(l)}\big)_{i \in \mathbb{N}_0},~ l=1,2,$ are increasing sequences of stopping times with $t_{0,n}^{(l)}=0$. 
\begin{figure}[b]\label{fig:obs_sch}
\centering
\hspace{-0.6cm}
\begin{tikzpicture}
\draw (0,1.75) -- (11,1.75)
      (0,-0.75) -- (11,-0.75)
      (0,1.5) -- (0,2)
		(1.9,1.5) -- (1.9,2)
		(5,1.5) -- (5,2)
		(7.3,1.5) -- (7.3,2)
		
		(10.8,1.5) -- (10.8,2)
		(0,-0.5) -- (0,-1)
		(2.3,-0.5) -- (2.3,-1)
		(5.7,-0.5) -- (5.7,-1)
		(8,-0.5) -- (8,-1)
		
		(10.3,-0.5) -- (10.3,-1);
\draw[dashed] (11,1.75) -- (12,1.75)
      (11,-0.75) -- (12,-0.75);
\draw[very thick] (9.5,-1.4) -- (9.5,0.25)
      (9.5,0.8) -- (9.5,2.4);
\draw	(0,0.5) node{$t_{0,n}^{(1)}=t_{0,n}^{(2)}=0$}
		(1.9,1) node{$t_{1,n}^{(1)}$}
		(5,1) node{$t_{2,n}^{(1)}$}
		(7.3,1) node{$t_{3,n}^{(1)}$}
		(11,1) node{$t_{4,n}^{(1)}$}
		(9.5,0.5) node{\textbf{$T$}}
		(2.3,0) node{$t_{1,n}^{(2)}$}
		(5.7,0) node{$t_{2,n}^{(2)}$}
		(8,0) node{$t_{3,n}^{(2)}$}
		(10.3,0) node{$t_{4,n}^{(2)}$};
\draw   (0,1.75) node[left,xshift=-0pt]{$X^{(1)}$}
(0,-0.75) node[left,xshift=-0pt]{$X^{(2)}$};
\draw[decorate,decoration={brace,amplitude=12pt}]
	(0,2)--(1.9,2) node[midway, above,yshift=10pt,]{$\leq|\pi_n|_T$};\draw[decorate,decoration={brace,amplitude=12pt}]
	(1.9,2)--(5,2) node[midway, above,yshift=10pt,]{$\leq|\pi_n|_T$};\draw[decorate,decoration={brace,amplitude=12pt}]
	(5,2)--(7.3,2) node[midway, above,yshift=10pt,]{$\leq|\pi_n|_T$};\draw[decorate,decoration={brace,amplitude=12pt}]
	(7.3,2)--(9.5,2) node[midway, above,yshift=10pt,]{$\leq|\pi_n|_T$};	
\draw[decorate,decoration={brace,amplitude=12pt}]
	(9.5,-1)--(8,-1) node[midway, below,yshift=-10pt,]{$\leq|\pi_n|_T$};
	\draw[decorate,decoration={brace,amplitude=12pt}]
	(8,-1)--(5.7,-1) node[midway, below,yshift=-10pt,]{$\leq|\pi_n|_T$};
	\draw[decorate,decoration={brace,amplitude=12pt}]
	(5.7,-1)--(2.3,-1) node[midway, below,yshift=-10pt,]{$\leq|\pi_n|_T$};
	\draw[decorate,decoration={brace,amplitude=12pt}]
	(2.3,-1)--(0,-1) node[midway, below,yshift=-10pt,]{$\leq|\pi_n|_T$};
\end{tikzpicture}
\caption{A realization of the observation scheme $\pi_n$ restricted to $[0,T]$.}
\end{figure}
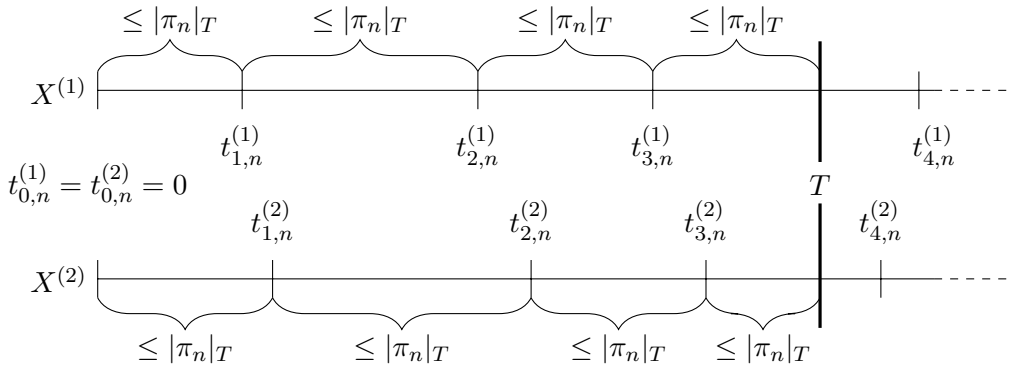
By
\[
|\pi_n|_T=\sup \big\{t_{i,n}^{(l)} \wedge T -t_{i-1,n}^{(l)} \wedge T \big|i\geq 1,~ l=1,2   \big\}
\]
we denote the mesh of the observation times up to $T$. All our test statistics are based on the increments
\[
\Delta_{i,n}^{(l)} X= X_{t_{i,n}^{(l)}}-X_{t_{i-1,n}^{(l)}}, \quad i \geq 1, \quad l=1,2,
\]
and we denote by $\mathcal{I}_{i,n}^{(l)}=\big(t_{i-1,n}^{(l)},t_{i,n}^{(l)}\big],~l=1,2,$ the corresponding observation intervals. Further we denote by $|A|$ the Lebesgue measure of a set $A \subset [0,\infty)$.

\begin{definition}\label{def_exo}
Let $\mathcal{S}=\sigma(\pi_n, n \in \mathbb{N})$ denote the $\sigma$-algebra generated by the observation scheme and $\mathcal{X}=\sigma(X,b,\sigma,\delta,W,\mu)$ denote the $\sigma$-algebra generated by the process $X$ and its components. We call an observation scheme $(\pi_n)_{n \in \mathbb{N}}$ \textit{exogenous} if the observation scheme and the process $X$ are independent, i.e.\ if $\mathcal{S}$ and $\mathcal{X}$ are independent. \qed
\end{definition}

Throughout this paper we impose the following structural assumptions on the It\^o semimartingale \eqref{ItoSemimart} and the observation scheme.

\begin{condition}\label{cond_consistency}
The process $b_t$ is locally bounded and the processes $\sigma_t^{(1)},\sigma_t^{(2)},\rho_t$ are c\`adl\`ag. Furthermore, it holds $\|\delta(\omega,t,z)\| \leq \Gamma_t(\omega) \gamma(z)$ for some locally bounded process $\Gamma_t$ and some deterministic bounded function $\gamma$ which satisfies $\int ( \gamma(z)^2 \wedge 1)\lambda(dz) < \infty.$ The sequence of observation schemes $(\pi_n)_n$ fulfills $|\pi_n|_T \overset{\mathbb{P}}{\longrightarrow} 0$. \qed
\end{condition}

The assumptions made in Condition \ref{cond_consistency} on the components of the process $X$ are not very restrictive and appear in similar form elsewhere in the literature; compare e.g.\ assumption (H) in \cite{JacPro12} or \cite{JacTod09}. These assumptions are fulfilled in most applications e.g.\ in the field of mathematical finance. The assumption $|\pi_n|_T \overset{\mathbb{P}}{\longrightarrow} 0$ on the observation scheme is a minimal requirement to be able to infer properties of the paths $t \mapsto X_t(\omega)$, $t \in [0,T]$, in the limit as $n \rightarrow \infty$. The investigation of properties of the process $X$ based on observations whose mesh decreases to zero characterizes the field of high-frequency statistics.

\section{Non-normalized functionals}\label{sec:nonnormFunc}
\def\theequation{3.\arabic{equation}}
\setcounter{equation}{0}

First note that when considering functionals of the form \eqref{intro1} in the setting of asynchronous observation times it is not straightforward anymore for which pairs of increments $(\Delta_{i,n}^{(1)}X^{(1)},\Delta_{j,n}^{(2)}X^{(2)})$ the evaluation of the function $f$ should be included in the sum. The idea utilized by \cite{HayYos05} is to include $f(\Delta_{i,n}^{(1)}X^{(1)},\Delta_{j,n}^{(2)}X^{(2)})$ if and only if the observation intervals $\mathcal{I}_{i,n}^{(1)}$, $\mathcal{I}_{j,n}^{(2)}$ overlap. In this case a consistent estimator for the quadratic covariation of $[X^{(1)},X^{(2)}]_T$ is obtained by using the function $f(x_1,x_2)=x_1 x_2$ also in the setting of asynchronous observations, i.e.\ they showed
\begin{align}\label{hayyos_est}
\sum_{i,j:t_{i,n}^{(1)}\vee t_{j,n}^{(2)} \leq T} \Delta_{i,n}^{(1)}X^{(1)} \Delta_{j,n}^{(2)}X^{(2)} \mathds{1}_{\{\mathcal{I}_{i,n}^{(1)} \cap \mathcal{I}_{j,n}^{(2)} \neq \emptyset\}} \overset{\mathbb{P}}{\longrightarrow} [X^{(1)},X^{(2)}]_T
\end{align}
in the case of a continuous It\^o semimartingale $X$. The extension to processes including jumps was later given in \cite{BibVet15}. The structure of the sum is illustrated in Figure \ref{hayyos_overlap}.

\begin{figure}[b]
\centering
\hspace{-0.6cm}
\begin{tikzpicture}
\draw   (0.5,1) -- (11,1)
		(0.5,0.7) -- (0.5,1.3)
		(2.1,0.7) -- (2.1,1.3)
		(5.5,0.7) -- (5.5,1.3)
		(7,0.7) -- (7,1.3)
		(8.9,0.7) -- (8.9,1.3)
		(11,0.7) -- (11,1.3)
		(0.5,-1) -- (10.5,-1)
		(0.5,-0.7) -- (0.5,-1.3)
		(3.8,-0.7) -- (3.8,-1.3)
		(5.5,-0.7) -- (5.5,-1.3)
		(10.5,-0.7) -- (10.5,-1.3);
\draw[dashed,blue] (1.3,1) -- (2.15,-1) -- (3.8,1) -- (4.65,-1) 
		(7.2,-1) -- (6.25,1)
		(7.2,-1) -- (7.95,1)
		(7.2,-1) -- (9.95,1);
\draw[dashed] (11,1)--(12.5,1)
		(10.5,-1) -- (12.5,-1);

\draw[decorate,decoration={brace,amplitude=12pt}]
	(0.5,1.35)--(2.1,1.35) node[midway, above,yshift=10pt,]{$| \mathcal{I}_{1,n}^{(1)}|$};
\draw[decorate,decoration={brace,amplitude=12pt}]
	(2.1,1.35)--(5.5,1.35) node[midway, above,yshift=10pt,]{$| \mathcal{I}_{2,n}^{(1)}|$};
\draw[decorate,decoration={brace,amplitude=12pt}]
	(5.5,1.35)--(7,1.35) node[midway, above,yshift=10pt,]{$| \mathcal{I}_{3,n}^{(1)}|$};
\draw[decorate,decoration={brace,amplitude=12pt}]
	(7,1.35)--(8.9,1.35) node[midway, above,yshift=10pt,]{$| \mathcal{I}_{4,n}^{(1)}|$};
\draw[decorate,decoration={brace,amplitude=12pt}]
	(8.9,1.35)--(11,1.35) node[midway, above,yshift=10pt,]{$| \mathcal{I}_{5,n}^{(1)}|$};
\draw[decorate,decoration={brace,amplitude=12pt}]
	(10.5,-1.35)--(5.5,-1.35) node[midway, below,yshift=-10pt,]{$| \mathcal{I}_{3,n}^{(2)}|$};
\draw[decorate,decoration={brace,amplitude=12pt}]
	(5.5,-1.35)--(3.8,-1.35) node[midway, below,yshift=-10pt,]{$| \mathcal{I}_{2,n}^{(2)}|$};
\draw[decorate,decoration={brace,amplitude=12pt}]
	(3.8,-1.35)--(0.5,-1.35) node[midway, below,yshift=-10pt,]{$| \mathcal{I}_{1,n}^{(2)}|$};
	
\draw[decorate,decoration={brace,amplitude=8pt}]
	(2.1,0.65)--(0.5,0.65) node[midway, below,yshift=-10pt,]{\tiny{$\big| \mathcal{I}_{1,n}^{(1)}\cap \mathcal{I}_{1,n}^{(2)}\big|$}};	
	\draw[decorate,decoration={brace,amplitude=8pt}]
	(3.8,0.65)--(2.1,0.65) node[midway, below,yshift=-10pt,]{\tiny{$\big| \mathcal{I}_{2,n}^{(1)}\cap \mathcal{I}_{1,n}^{(2)}\big|$}};
	\draw[decorate,decoration={brace,amplitude=8pt}]
	(5.5,0.65)--(3.8,0.65) node[midway, below,yshift=-10pt,]{\tiny{$\big| \mathcal{I}_{2,n}^{(1)}\cap \mathcal{I}_{2,n}^{(2)}\big|$}};
	\draw[decorate,decoration={brace,amplitude=8pt}]
	(7,0.65)--(5.5,0.65) node[midway, below,yshift=-10pt,]{\tiny{$\big| \mathcal{I}_{3,n}^{(1)}\cap \mathcal{I}_{3,n}^{(2)}\big|$}};
	\draw[decorate,decoration={brace,amplitude=8pt}]
	(8.9,0.65)--(7,0.65) node[midway, below,yshift=-10pt,]{\tiny{$\big| \mathcal{I}_{4,n}^{(1)}\cap \mathcal{I}_{3,n}^{(2)}\big|$}};
	\draw[decorate,decoration={brace,amplitude=8pt}]
	(10.5,0.65)--(8.9,0.65) node[midway, below,yshift=-10pt,]{\tiny{$\big| \mathcal{I}_{5,n}^{(1)}\cap \mathcal{I}_{3,n}^{(2)}\big|$}};
	
\draw (0.5,1) node[left,xshift=-2pt]{$X^{(1)}$}
(0.5,-1) node[left,xshift=-2pt]{$X^{(2)}$};
\end{tikzpicture}
\caption{All products of increments of $X^{(1)}$ and $X^{(2)}$ over intersecting intervals enter the estimation of $[X^{(1)},X^{(2)}]_T$.}\label{hayyos_overlap}
\end{figure}
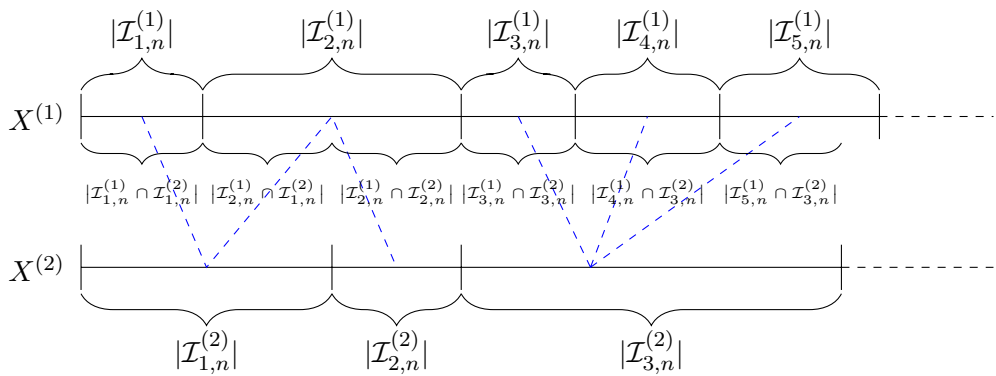

In the style of this famous \textit{Hayashi-Yoshida estimator for the quadratic covariation} we define 
\[
V(f,\pi_n)_T=\sum_{i,j:t_{i,n}^{(1)} \vee t_{j,n}^{(2)}\leq T} f(\Delta_{i,n}^{(1)} X^{(1)},\Delta_{j,n}^{(2)} X^{(2)} )\mathds{1}_{\{\mathcal{I}_{i,n}^{(1)} \cap \mathcal{I}_{j,n}^{(2)}\neq \emptyset \}}
\]
for functions $f:\R^2 \rightarrow \R$. We will see that these functionals converge to similar limits as the functionals \eqref{intro1} in the setting of synchronous observation times for a large class of functions $f$, and not only for $f(x_1,x_2)=x_1x_2$ as in the case of the Hayashi-Yoshida estimator. 

Further we define
\[
V^{(l)}(g,\pi_n)_T=\sum_{i:t_{i,n}^{(l)}\leq T} g(\Delta_{i,n}^{(l)} X^{(l)}), \quad l=1,2,
\]
for functions $g:\R \rightarrow \R$. We will also state an asymptotic result for $V^{(l)}(g,\pi_n)_T$ to compare the results in the setting of asynchronously observed bivariate processes to those in simpler settings.

To describe the limits of the functionals $V(f,\pi_n)_T$ and $V^{(l)}(g,\pi_n)_T$ we denote
\begin{align*}
&B(f)_T=\sum_{s \leq T} f(\Delta X^{(1)}_s,\Delta X^{(2)}_s),
\\&B^*(f)_T=\sum_{s \leq T} f(\Delta X^{(1)}_s,\Delta X^{(2)}_s)\mathds{1}_{\{\Delta X^{(1)}_s  \Delta X^{(2)}_s \neq 0\}},
\\&B^{(l)}(g)_T=\sum_{s \leq T} g(\Delta X^{(l)}_s),~~~l=1,2,
\end{align*}
for functions $f:\R^2 \rightarrow \R$ and $g:\R \rightarrow \R$ for which the sums are well-defined.\\

Using this notation we are now able to state the results. In the setting of synchronous observation times 
\begin{align}\label{synch_times}
t_{i,n}=t_{i,n}^{(1)}=t_{i,n}^{(2)},\quad  i \in \mathbb{N}_0,~n\in \mathbb{N},
\end{align}
the functional $V(f,\pi_n)_T$ coincides with the classical statistic \eqref{intro1} and hence the convergence of $V(f,\pi_n)_T$ in this situation follows from Theorem 3.3.1 of \cite{JacPro12} which we state below.

\begin{theorem}\label{general_nonnormFunc_synch}
Suppose that the observation scheme is synchronous and that Condition \ref{cond_consistency} holds. Then we have 
\begin{align}\label{synch_cons}
V(f,\pi_n)_T \overset{\mathbb{P}}{\longrightarrow} B(f)_T
\end{align}
for all continuous functions $f:\R^2 \rightarrow \R$ with $f(x)=o(\|x\|^2)$ as $x \rightarrow 0$.
\end{theorem}

Actually the statement in Theorem 3.3.1 of \cite{JacPro12} holds for general $d$-dimensional It\^o semimartingales and all functions ${f:\mathbb{R}^d \rightarrow \R}$ with ${f(x)=o(\|x\|^2)}$ for any $d \in \mathbb{N}$. The case $d=1$ then yields the convergence for the functionals $V^{(l)}(g,\pi_n)_T$ stated in the following corollary. 

\begin{corollary}\label{nonnormFunc_onedim}
Under Condition \ref{cond_consistency} we have 
\begin{align}\label{onedim_cons}
V^{(l)}(g,\pi_n)_T \overset{\mathbb{P}}{\longrightarrow} B^{(l)}(g)_T
\end{align}
for all continuous functions $g:\R \rightarrow \R$ with $g(x)=o(x^2)$ as $x \rightarrow 0$.
\end{corollary}

The following theorem states the most general result which can be obtained if the convergence \eqref{cons_22} is supposed to hold for arbitrary It\^o semimartingales and any asynchronous observation scheme.

\begin{theorem}\label{general_22}
Under Condition \ref{cond_consistency} we have
\begin{align}\label{cons_22}
V(f,\pi_n)_T \overset{\mathbb{P}}{\longrightarrow} B^*(f)_T
\end{align}
for all continuous functions $f: \mathbb{R}^2 \rightarrow \mathbb{R}$ with $f(x,y)=O(x^2y^2)$ as $|xy|\rightarrow 0$. 
\end{theorem}

As for the convergence in Theorem \ref{general_nonnormFunc_synch} in the setting of synchronous observation times we need that $f(x,y)$ vanishes as $(x,y) \rightarrow 0$ also in the setting of asynchronous observation times. However in the asynchronous setting we further need ${f(x_k,y_k)\rightarrow 0}$ also for sequences $(x_k,y_k)_{k \in \mathbb{N}}$ which do not converge to zero, but which fulfill ${|x_k y_k| \rightarrow 0}$. Hence the condition on $f$ needed to obtain convergence of $V(f,\pi_n)_T$ in the asynchronous setting is stronger compared to the corresponding condition in the synchronous setting. Further we observe that in the asynchronous setting the limit only consists of common jumps of $X^{(1)}$ and $X^{(2)}$ i.e.\ jumps with $\Delta X^{(1)}_s \neq 0 \neq \Delta X^{(2)}_s$. The following example illustrates the need for the stronger condition as well as why we only consider functions $f$ which yield a limit that consists only of common jumps.

\begin{example}\label{ex_min_req1}
Consider the function $f_{3,0}(x,y)=x^3$, which fulfills $f_{3,0}(x,y)\rightarrow 0$ as $(x,y)\rightarrow 0$ but not as $|xy|\rightarrow 0$, and the observation scheme given by $t_{i,n}^{(1)}=i/n$ and $t_{i,n}^{(2)}=i/(2n)$. Then
\begin{align*}
V(f_{3,0},\pi_n)_T=2 \sum_{i:t_{i,n}^{(1)}\leq T} (\Delta_{i,n}^{(1)} X^{(1)})^3 \overset{\mathbb{P}}{\longrightarrow} 2 B(f_{3,0})_T
\end{align*}
where the convergence is due to Corollary \ref{nonnormFunc_onedim}. However for the standard synchronous observation scheme $t_{i,n}^{(1)}=t_{i,n}^{(2)}=i/n$ we have $V(f_{3,0},\pi_n)_T \overset{\mathbb{P}}{\longrightarrow} B(f_{3,0})_T$ also due to Corollary \ref{nonnormFunc_onedim}. Hence the limit here depends on the observation scheme. If we further consider the observation scheme with $t_{i,n}^{(1)}=i/n$ and
\begin{align*}
t_{i,n}^{(2)}=\begin{cases}
i/n , &n \text{~even},
\\ i/(2n),& n  \text{ odd},
\end{cases}
\end{align*}
then $V(f_{3,0},\pi_n)_T$ does not converge at all unless $B(f_{3,0})_T=0$, as one subsequence converges to $B(f_{3,0})_T$ and the other one to $2B(f_{3,0})_T$. Hence there cannot exist a convergence result for $V(f_{3,0},\pi_n)_T$ which holds for any It\^o semimartingale $X$ and any sequence of observation schemes $\pi_n$, $ n \in \mathbb{N}$. 

If we consider instead a function $f(x,y)$ that vanishes as $|xy| \rightarrow 0$ such a behaviour cannot occur because idiosyncratic jumps do not contribute in the limit as e.g.\ for $\Delta X^{(1)}_s \neq 0$ and $\Delta X^{(2)}_s =0$ we have 
\begin{align*}
\sup_{(i,j):s \in \mathcal{I}_{i,n}^{(1)}, \mathcal{I}_{i,n}^{(1)} \cap \mathcal{I}_{j,n}^{(2)} \neq \emptyset} |\Delta_{i,n}^{(1)}X^{(1)} \Delta_{j,n}^{(2)} X^{(2)}|\overset{\mathbb{P}}{\longrightarrow} 0.
\end{align*}
If on the other hand there is a common jump at time $s$ there is only one summand $f(\Delta_{i,n}^{(1)}X^{(1)},\Delta_{j,n}^{(2)}X^{(2)})$ such that $s \in \mathcal{I}_{i,n}^{(1)}$ and $s \in \mathcal{I}_{j,n}^{(2)}$. Hence common jumps only enter the limit once.

This example shows that the assumption that $f(x,y)$ vanishes as $|xy|\rightarrow 0$ is needed to filter out the contribution of idiosyncratic jumps. These jumps may enter $V(f,\pi_n)_T$ multiple times, where the multiplicity by which they occur may depend on $n$ and $\omega$ and therefore may prevent $V(f,\pi_n)_T$ from converging. \qed
\end{example}

Let us now consider the order by which the function $f(x,y)$ has to decrease as $(x,y) \rightarrow 0$ or, respectively, $|xy| \rightarrow 0$. We observe that in the asynchronous setting the function $f$ has to decrease quadratically in both $x$ and $y$ while in the synchronous setting it only has to decrease quadratically in $(x,y)$.  Adding this condition to the requirement that $f(x_k y_k)$ has to vanish for any sequence with $|x_k y_k| \rightarrow 0$ further diminishes the class of functions $f$ for which $V(f,\pi_n)_T$ converges in the asynchronous setting compared to the synchronous one. The need for this stronger condition on $f$ is due to the fact that the lengths of the observation intervals of $X^{(1)}$ and $X^{(2)}$ may decrease with different rates in the asynchronous setting which is illustrated in the following example.

\begin{example}\label{ex_min_req2}
Let $X^{(1)}_t=\mathds{1}_{\{t \geq U\}}$ for $U \sim \mathcal{U}[0,1]$ and $X^{(2)}$ be a standard Brownian motion independent of $U$. The observation schemes are given by $t_{i,n}^{(1)}=i/n$ and $t_{i,n}^{(2)}=i/n^{1+\gamma}$ with $\gamma>0$. Then for $f(x,y)=|x|^{p_1} |y|^{p_2}$ as illustrated in Figure \ref{obs_freq} we have
\begin{align*}
V(f,\pi_n)_1&= \sum_{i=\lfloor n^{1+\gamma}(\lceil  n U \rceil-1) /n\rfloor/n^{1+\gamma}+1}^{\lceil n^{1+\gamma}\lceil  n U \rceil /n\rceil/n^{1+\gamma}} \big|X^{(2)}_{i/n^{1+\gamma}}-X^{(2)}_{(i-1)/n^{1+\gamma}}  \big|^{p_2}
\\& \geq \sum_{i=1}^{\lceil n^\gamma \rceil} \big|n^{-(1+\gamma)/2}Z_i^n \big|^{p_2}
\\& = n^{-(1+\gamma){p_2}/2+\gamma}
\Big( n^{-\gamma}\sum_{i=1}^{\lceil n^\gamma \rceil} \big|Z_i^n \big|^{p_2} \Big)
\end{align*}
where the $Z_i^n:=n^{(1+\gamma)/2}\Delta_{\lfloor n^{1+\gamma}(\lceil  n U \rceil-1) /n\rfloor/n^{1+\gamma} +i/n^{1+\gamma},n}^{(2)}X^{(2)}$, $i=1,\ldots \lceil n^\gamma \rceil$, are i.i.d.\ standard normal random variables for each $n \in \mathbb{N}$. Hence $V(f,\pi_n)_1$ diverges for $p_2 <2$ if 
\begin{align*}
\gamma > \frac{p_2}{2-{p_2}}
\end{align*}
because the expression in parantheses converges in probability to $\mathbb{E}[|Z|^{p_2}]$, $Z \sim \mathcal{N}(0,1)$, by the law of large numbers. Here we are able to find a suitably large $\gamma$ explicitly because the ${p_2}$-variations of a Brownian motion are infinite for ${p_2}<2$. But we also have $B^*(f)_1=0$ in this setting because $X^{(2)}$ is continuous. Hence \eqref{cons_22} cannot hold for $f(x,y)=|x|^{p_1} |y|^{p_2}$, any It\^o semimartingale of the form \eqref{ItoSemimart} and any observation scheme which fulfills Condition \ref{cond_consistency} if $p_1 \wedge p_2 < 2$. \qed

\begin{figure}[bt]
\centering
\hspace{-0.6cm}
\begin{tikzpicture}
\draw[dashed] (0.5,1) -- (2.5,1)
			(10.5,1) -- (12,1)
			(0.5,-1) -- (1.5,-1)
			(11,-1) -- (12,-1)
			(7,-1.2) -- (7,1.2);
\draw (3,-1.3) -- (3,-0.7)
(4.5,-1.3) -- (4.5,-0.7)
(6,-1.3) -- (6,-0.7)
(7.5,-1.3) -- (7.5,-0.7)
(9,-1.3) -- (9,-0.7)
(10.5,-1.3) -- (10.5,-0.7)
(3.7,1.3) -- (3.7,0.7)
	  (9.3,1.3) -- (9.3,0.7);
\draw (2.5,1) -- (10.5,1)
	  (1.5,-1) -- (11,-1);
\draw (3.7,1.7) node{$\frac{\lceil nU \rceil -1}{n}$}
(9.3,1.7) node{$\frac{\lceil nU \rceil }{n}$}
(7,1.5) node{$U$}
		(3,-1.7) node{$\frac{\lfloor n^{1+\gamma}(\lceil nU \rceil -1)/n\rfloor}{n^{1+\gamma}}$}
		(10.5,-1.7) node{$\frac{\lceil n^{1+\gamma}\lceil nU \rceil /n\rceil}{n^{1+\gamma}}$};	
\draw (0.5,1) node[left,xshift=-2pt]{$X^{(1)}$}
(0.5,-1) node[left,xshift=-2pt]{$X^{(2)}$};
\draw[decorate,decoration={brace,amplitude=12pt}]
	(6,-1.35)--(4.5,-1.35) node[midway, below,yshift=-12pt,]{$1/n^{1+\gamma}$};
	\draw[decorate,decoration={brace,amplitude=12pt}]
	(9.3,0.65)--(3.7,0.65) node[midway, below,yshift=-12pt,]{$1/n$};
\end{tikzpicture}
\caption{Observation times of $X^{(1)}$ and $X^{(2)}$ around the jump time $U$.}\label{obs_freq}
\end{figure}
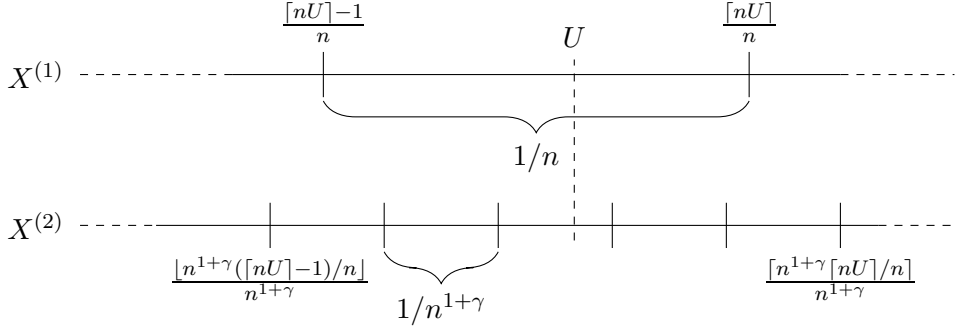
\end{example}

Example \ref{ex_min_req2} shows that the convergence \eqref{cons_22} fails for functions $f(x,y)=|x|^{p_1} |y|^{p_2}$ with $p_1,p_2 <2$ in combination with observation schemes where the observation frequency for one process increases much faster as $n \rightarrow\infty$ than the observation frequency of the other process. If we consider only observation schemes where such a behaviour is prohibited, we can also obtain the convergence in \eqref{cons_22} for functions  $f(x,y)=|x|^{p_1} |y|^{p_2}$ with $p_1,p_2 <2$.

First, we state a result in the case of exogenous observation times introduced in Definition \ref{def_exo}, i.e.\ random observation times that do not depend on the process $X$ or its components.

\begin{theorem}\label{general_1plusBeta}
Assume that Condition \ref{cond_consistency} holds and that the observation scheme is exogenous. Further let $p_1, p_2>0$ with $p_1+p_2 \geq 2$. If we have
\begin{align}\label{cons_cond_1plusBeta}
\sum_{i,j:t_{i,n}^{(1)} \vee t_{j,n}^{(2)}\leq T} 
\big|\mathcal{I}_{i,n}^{(1)} \big|^{\frac{p_1}{2}\wedge 1} \big|\mathcal{I}_{j,n}^{(2)} \big|^{\frac{p_2}{2}\wedge 1}
\mathds{1}_{\{\mathcal{I}_{i,n}^{(1)} \cap \mathcal{I}_{j,n}^{(2)}\neq \emptyset \}}=O_\mathbb{P}(1)
\end{align}
as $n \rightarrow \infty$ it holds
\begin{align}\label{cons_1plusBeta}
V(f,\pi_n)_T \overset{\mathbb{P}}{\longrightarrow} B^*(f)_T
\end{align}
for all continuous functions $f:\R^2 \rightarrow \R$ with $f(x,y)=o(|x|^{p_1}|y|^{p_2})$ as $|xy|\rightarrow 0$. 
\end{theorem}

In the boundary case $p_1+p_2=2$ in Theorem \ref{general_1plusBeta} we achieve convergence for all functions $f$ that are for $|xy|\rightarrow 0$ dominated by the function $|x|^{p_1}|y|^{p_2}$ which is of order $p_1+p_2=2$. Hence Theorem \ref{general_1plusBeta} allows to achieve the convergence in \eqref{cons_1plusBeta} for functions $f$ which are dominated by functions of the same order as the dominating function $\|(x,y)\|^2$ in the synchronous case in Theorem \ref{general_nonnormFunc_synch}. However, the requirement that $f(x,y)$ vanishes as $|xy|\rightarrow 0$ cannot be relaxed because this assumption is as illustrated in Example \ref{ex_min_req1} fundamentally necessary due to the asynchronous nature of the observation scheme.\\

As in the synchronous setting we cannot have the general convergence in \eqref{cons_1plusBeta} for functions $f$ which do not fulfill $f(x,y)=O(\|(x,y)\|^2)$ as $(x,y) \rightarrow 0$, because in this case $B^*(f)$ might not be well defined. An indication for this fact is also given by the observation that condition \eqref{cons_cond_1plusBeta} can never be fulfilled if $p_1+p_2<2$ and $|\pi_n|_T \rightarrow 0$ as shown in the following remark.

\begin{remark}\label{remark_order2}
Suppose that we have $p_1+p_2<2$ and $|\pi_n|_T \rightarrow 0$. In this situation we obtain the following estimate for the left hand side of \eqref{cons_cond_1plusBeta}, using $p_l/2 \wedge 1=p_l/2$, $l=1,2$, and the inequality $\sum_{i=1}^N {a_i}^p \geq (\sum_{i=1}^N {a_i})^p$, which holds for all $N \in \mathbb{N}$, $a_i \geq 0$, $p \in [0,1)$:
\begin{align*}
&\sum_{i,j:t_{i,n}^{(1)} \vee t_{j,n}^{(2)}\leq T} 
|\mathcal{I}_{i,n}^{(1)} |^{\frac{p_1}{2}\wedge 1} |\mathcal{I}_{j,n}^{(2)} |^{\frac{p_2}{2}\wedge 1} \mathds{1}_{\{\mathcal{I}_{i,n}^{(1)} \cap \mathcal{I}_{j,n}^{(2)}\neq \emptyset \}}
\\&~~~~~\geq \sum_{i:t_{i,n}^{(1)} \leq T} 
|\mathcal{I}_{i,n}^{(1)}|^{\frac{p_1}{2}} \Big(\sum_{j:t_{j,n}^{(2)}\leq T} |\mathcal{I}_{j,n}^{(2)} |\mathds{1}_{\{\mathcal{I}_{i,n}^{(1)} \cap \mathcal{I}_{j,n}^{(2)}\neq \emptyset \}}\Big)^{\frac{p_2}{2}}
\\&~~~~~\geq \sum_{i:t_{i,n}^{(1)} \leq T} 
|\mathcal{I}_{i,n}^{(1)}|^{\frac{p_1}{2}} |\mathcal{I}_{i,n}^{(1)}|^{\frac{p_2}{2}}-O((|\pi_n|_T)^{\frac{p_1+p_2}{2}})
\\&~~~~~\geq (|\pi_n|_T)^{\frac{p_1+p_2}{2}-1}T-O((|\pi_n|_T)^{\frac{p_1+p_2}{2}}).
\end{align*}
Here the expression in the last line converges in probability to infinity due to ${p_1+p_2<2}$ and $|\pi_n|_T \overset{\mathbb{P}}{\longrightarrow} 0$.

Suppose $p_1 \wedge p_2\geq 2$ holds. Then we have
\begin{align*}
&\sum_{i,j:t_{i,n}^{(1)} \vee t_{j,n}^{(2)}\leq T} 
|\mathcal{I}_{i,n}^{(1)} |^{\frac{p_1}{2}\wedge 1} |\mathcal{I}_{j,n}^{(2)} |^{\frac{p_2}{2}\wedge 1} \mathds{1}_{\{\mathcal{I}_{i,n}^{(1)} \cap \mathcal{I}_{j,n}^{(2)}\neq \emptyset \}}
\\&~~~~~~~~~~~~~~~~\leq K\sum_{i,j:t_{i,n}^{(1)} \vee t_{j,n}^{(2)}\leq T} 
|\mathcal{I}_{i,n}^{(1)} | |\mathcal{I}_{j,n}^{(2)} | \mathds{1}_{\{\mathcal{I}_{i,n}^{(1)} \cap \mathcal{I}_{j,n}^{(2)}\neq \emptyset \}}
\leq 3 K |\pi_n|_T T.
\end{align*}
Hence Condition \eqref{cons_cond_1plusBeta} is by Condition \ref{cond_consistency} always fulfilled in the setting of Theorem \ref{general_22}. \qed
\end{remark}

\begin{example}\label{ex_min_req2b}
Let $p \in [1,2)$ and consider the deterministic sampling scheme given by $t_{i,n}^{(1)}=i/n$ and $t_{i,n}^{(2)}=i/n^{1+\gamma}$ from Example \ref{ex_min_req2} with $\gamma=\gamma(p)=\frac{2p-2}{2-p}$. In this case it holds
\begin{align*}
&\sum_{i,j:t_{i,n}^{(1)} \vee t_{j,n}^{(2)}\leq T} 
\big(\big|\mathcal{I}_{i,n}^{(1)} \big| \big|\mathcal{I}_{j,n}^{(2)} \big|\big)^{\frac{p}{2}}
\mathds{1}_{\{\mathcal{I}_{i,n}^{(1)} \cap \mathcal{I}_{j,n}^{(2)}\neq \emptyset \}} = Tn^{1+\gamma(p)} \Big(\frac{1}{n}\frac{1}{n^{1+\gamma(p)}}\Big)^{\frac{p}{2}}(1+o(1))
\\&~~~~~~~=Tn^{1+\gamma(p)-(2+\gamma(p))\frac{p}{2}}(1+o(1))=O(1).
\end{align*}
Hence if we want \eqref{cons_1plusBeta} to hold for all functions $f$ with $f(x,y)=o(|xy|^p)$ we can allow for observation schemes where the observation frequencies differ by a factor of up to $n^{\gamma(p)}$ where $\gamma(p)$ increases in $p$. For $p=1$ we have $\gamma(1)=0$ and for $p \rightarrow 2$ we have $\gamma(p)\rightarrow \infty$. \qed
\end{example}

In general we observe that if the $o(|x|^{p_1}|y|^{p_2})$-restriction on $f$ is less restrictive, then the restriction \eqref{cons_cond_1plusBeta} on the observation scheme has to be more restrictive, and vice versa. Here the abstract criterion \eqref{cons_cond_1plusBeta} characterizing the allowed classes of observation schemes can be related, as illustrated in Example \ref {ex_min_req2b}, to the asymptotics of the ratio of the observation frequencies of the two processes. Hence if the observation frequency of one process increases much faster than the observation frequency of the other process we obtain the convergence in \eqref{cons_1plusBeta} only for a small class of functions $f$.

\begin{example}
Condition \eqref{cons_cond_1plusBeta} is fulfilled for $p_1=p_2=1$ in the case where the observation times $\{t_{i,n}^{(l)}:i \in \mathbb{N}\}$, $l=1,2$, are given by the jump times of two independent time-homogeneous Poisson processes with intensities $n\lambda_1,n \lambda_2$. Indeed the arguments used in Lemma 6.5 of \cite{MarVet18b} yield 
\begin{align*}
\sum_{i,j:t_{i,n}^{(1)} \vee t_{j,n}^{(2)}\leq t} 
\big|\mathcal{I}_{i,n}^{(1)} \big|^{\frac{1}{2}} \big|\mathcal{I}_{j,n}^{(2)} \big|^{\frac{1}{2}}
\mathds{1}_{\{\mathcal{I}_{i,n}^{(1)} \cap \mathcal{I}_{j,n}^{(2)}\neq \emptyset \}} \overset{\mathbb{P}}{\longrightarrow} c t,~~ t \geq 0 ,
\end{align*}
for some positive real number $c>0$. Further note that if \eqref{cons_cond_1plusBeta} is fulfilled for $p_1,p_2\geq 1$ it is clearly also fulfilled for $p_1'\geq p_1,p_2'\geq p_2$. \qed
\end{example}

Next, we give a result that may also be applied in a setting with endogenous observation times.

\begin{theorem}\label{general_1plusBeta_end1}
Assume that Condition \ref{cond_consistency} holds. If for all $\varepsilon >0$ there exists some $N_\varepsilon \in \mathbb{N}$ with
\begin{align}\label{cons_cond_1plusBeta_end1}
&\limsup_{n \rightarrow \infty} \mathbb{P}\Big(\sup_{i:t_{i,n}^{(l)} \leq T} \sum_{j \in \mathbb{N}} 
\mathds{1}_{\{\mathcal{I}_{i,n}^{(l)} \cap \mathcal{I}_{j,n}^{(3-l)}\neq \emptyset \}}> N_\varepsilon,~l=1,2 \Big)< \varepsilon,
\end{align}
then it holds
\begin{align}\label{cons_1plusBeta_end1}
V(f,\pi_n)_T \overset{\mathbb{P}}{\longrightarrow} B^*(f)_T
\end{align}
for all continuous functions $f:\R^2 \rightarrow \R$ such that $f(x,y)=o(|x|^{p_1}|y|^{p_2})$ as $|xy|\rightarrow 0$ for some $p_1,p_2 \geq 0$ with $p_1+p_2=2$. 
\end{theorem}

Here \eqref{cons_cond_1plusBeta_end1} ensures that as $n$ tends to infinity the maximal number of observations of the process $X^{(3-l)}$ during one observation interval of $X^{(l)}$ is bounded. This yields that the ratio of the observation frequencies of the two processes is also bounded as $n \rightarrow \infty$ and cannot tend to infinity as in Example \ref{ex_min_req2}.

\begin{example}
Consider the case where $X^{(1)}$ and $X^{(2)}$ are observed alternately. In this case we have $$\sum_{j \in \mathbb{N}} 
\mathds{1}_{\{\mathcal{I}_{i,n}^{(l)} \cap \mathcal{I}_{j,n}^{(3-l)}\}} \leq 2$$
for $l=1,2$ and all $i,n$. 

Note that, although this goes along with a data reduction, the statistician may always use only a subset of all available observations and hence is able to turn the real observation scheme for example into an obervation scheme where the processes are observed alternately. One way to achieve this is the following: Start with the first observation time $t_{i,n}^{(1)}$ of $X^{(1)}$ and set $\tilde{t}_{1,n}^{(1)}=t_{1,n}^{(1)}$, then take the smallest observation time of $X^{(2)}$ larger than $\tilde{t}_{i,n}^{(1)}$ and set $\tilde{t}_{1,n}^{(2)}=\inf\{t_{i,n}^{(2)}|t_{i,n}^{(2)}>\tilde{t}_{1,n}^{(1)}\}$. Further set $\tilde{t}_{2,n}^{(1)}=\inf\{t_{i,n}^{(1)}|t_{i,n}^{(1)}>\tilde{t}_{1,n}^{(2)}\}$ and define recursively the new observation scheme $\tilde{\pi}_n$ by continuing this procedure in the natural way. \qed 
\end{example}

\section{Normalized functionals}\label{sec:normFunc}
\def\theequation{4.\arabic{equation}}
\setcounter{equation}{0}

In Section \ref{sec:nonnormFunc} we have seen that the functional $V(f,\pi_n)_T$ converges to a limit which depends only on the jump part of $X$ for functions $f$ that decay sufficiently fast in a neighbourhood of zero. This is necessary because we need that for such functions $f$ the contribution of the continuous part in $V(f,\pi_n)_T$ becomes asymptotically negligible. The jump part has the property that the magnitude of its increments remains constant as $|\pi_n|_T \rightarrow 0$ while for the continuous martingale part $$C_t =\int_0^t \sigma_s dW_s,~~t \geq 0,$$ the magnitude of the \glqq \textit{normalized}\grqq\ increment $|\mathcal{I}_{i,n}^{(l)}|^{-1/2}\Delta_{i,n}^{(l)} C^{(l)}$ remains constant. Hence if we would like to learn something about the continuous part of $X$ it is reasonable to look at functionals of the normalized increments $|\mathcal{I}_{i,n}^{(l)}|^{-1/2}\Delta_{i,n}^{(l)} X^{(l)}$.

As an illustration for the upcoming results consider the toy example
\begin{align}\label{example_2dim_BM}
X_t^{toy}=\sigma W_t
\end{align}
where the volatility matrix
\begin{align*}
\sigma= \begin{pmatrix}
\sigma^{(1)} &0
\\ \rho \sigma^{(2)} &\sqrt{1-\rho^2} \sigma^{(2)}
\end{pmatrix}
\end{align*}
is constant in time with $\sigma^{(1)},\sigma^{(2)} > 0$ and $\rho \in [-1,1]$. Suppose the observation scheme is exogenous as in Definition \ref{def_exo} and synchronous. Then we have with the notation from \eqref{synch_times}
\begin{align}\label{local_norm_1d}
 \sum_{i:t_{i,n} \leq T} |\mathcal{I}_{i,n}|f(|\mathcal{I}_{i,n}|^{-1/2} \Delta_{i,n} X^{toy}) =  \sum_{i:t_{i,n} \leq T} |\mathcal{I}_{i,n}|f(\sigma Z_i^n) \overset{\mathbb{P}}{\longrightarrow} T \mathbb{E}[f(\sigma Z)]
\end{align}
because of $|\pi_n|_T \overset{\mathbb{P}}{\longrightarrow} 0$, where $Z$ and $Z_i^n=|\mathcal{I}_{i,n}|^{-1/2}\Delta_{i,n} W$, $i \in \mathbb{N}_0$, are i.i.d.\ two-dimensional standard normal random variables for each $n \in \mathbb{N}$. Functionals of this form are discussed in Section 14.2 of \cite{JacPro12}. Two straightforward generalizations of this approach to the setting of asynchronous observation times lead to functionals of the form 
\begin{gather}\label{local_norm_2d}
 ~~\sum_{i,j:t_{i,n}^{(1)} \vee t_{j,n}^{(2)} \leq T} (|\mathcal{I}_{i,n}^{(1)}||\mathcal{I}_{j,n}^{(2)}|)^{1/2}f(|\mathcal{I}_{i,n}^{(1)}|^{-1/2} \Delta_{i,n}^{(1)} X^{(1)},|\mathcal{I}_{j,n}^{(2)}|^{-1/2}\Delta_{j,n}^{(2)} X^{(2)})\mathds{1}_{\{\mathcal{I}_{i,n}^{(1)} \cap \mathcal{I}_{j,n}^{(2)}\neq \emptyset\}}~~~~~~~  \raisetag{32pt}
\end{gather}
and\begin{gather}\label{local_norm_2d_2}
 ~~\sum_{i,j:t_{i,n}^{(1)} \vee t_{j,n}^{(2)} \leq T} |\mathcal{I}_{i,n}^{(1)} \cap \mathcal{I}_{j,n}^{(2)}| f(|\mathcal{I}_{i,n}^{(1)}|^{-1/2} \Delta_{i,n}^{(1)} X^{(1)},|\mathcal{I}_{j,n}^{(2)}|^{-1/2}\Delta_{j,n}^{(2)} X^{(2)})\mathds{1}_{\{\mathcal{I}_{i,n}^{(1)} \cap \mathcal{I}_{j,n}^{(2)}\neq \emptyset\}}.~~~~~~~  \raisetag{32pt}
\end{gather}
Here the main difference compared to the functional from \eqref{local_norm_1d} in the synchronous setting is that the law of $(|\mathcal{I}_{i,n}^{(1)}|^{-1/2} \Delta_{i,n}^{(1)} X^{(1)},|\mathcal{I}_{j,n}^{(2)}|^{-1/2}\Delta_{j,n}^{(2)} X^{(2)})$ is in general not independent of the observation scheme $\pi_n$. This property is due to the fact that e.g.\ the correlation of $|\mathcal{I}_{i,n}^{(1)}|^{-1/2} \Delta_{i,n}^{(1)} X^{{toy},(1)}$ and $|\mathcal{I}_{j,n}^{(2)}|^{-1/2}\Delta_{j,n}^{(2)} X^{{toy},(2)}$ equals
\begin{align}\label{overlap_depend}
\rho \frac{|\mathcal{I}_{i,n}^{(1)} \cap \mathcal{I}_{j,n}^{(2)}|}{|\mathcal{I}_{i,n}^{(1)}|^{1/2}|\mathcal{I}_{j,n}^{(2)}|^{1/2}}
\end{align}
as the increments of $X^{{toy},(1)}$ and $X^{{toy},(2)}$ are correlated only over the overlapping part ${\mathcal{I}_{i,n}^{(1)} \cap \mathcal{I}_{j,n}^{(2)}}$. This difference is also the main reason why it is more difficult to derive convergence results for normalized functionals as we will see later on. Further, regarding \eqref{local_norm_2d} the quantity
\begin{align*}
\sum_{i,j:t_{i,n}^{(1)} \vee t_{j,n}^{(2)} \leq T} (|\mathcal{I}_{i,n}^{(1)}||\mathcal{I}_{j,n}^{(2)}|)^{1/2}\mathds{1}_{\{\mathcal{I}_{i,n}^{(1)} \cap \mathcal{I}_{j,n}^{(2)}\neq \emptyset\}}
\end{align*}
might diverge for $|\pi_n|_T \rightarrow 0$ as has been shown in Example \ref{ex_min_req2b}. 

Due to these observations we will pick another approach where we use a global normalization instead of locally normalizing each $\Delta_{i,n}^{(l)}X^{(l)}$ with $|\mathcal{I}_{i,n}^{(l)}|^{1/2}$. Precisely, we have to assume that the \glqq average\grqq\ observation frequency increases with rate $n$. We then look at functionals of the form
\begin{align}\label{Vbar_mot}
\sum_{i,j:t_{i,n}^{(1)} \vee t_{j,n}^{(2)} \leq T} n^{-1}f(n^{1/2} \Delta_{i,n}^{(1)} X^{(1)},n^{1/2}\Delta_{j,n}^{(2)} X^{(2)})\mathds{1}_{\{\mathcal{I}_{i,n}^{(1)} \cap \mathcal{I}_{j,n}^{(2)}\neq \emptyset\}}. 
\end{align} 
Such functionals also appear to occur more naturally in applications; compare \cite{MarVet18a, MarVet18b}. 

The most common functions for which the functionals \eqref{Vbar_mot} are studied are the power functions ${g_p(x)=x^p}$, $\overline{g}_p=|x|^p$ and $f_{(p_1,p_2)}=x_1^{p_1}x_2^{p_2}$, ${\overline{f}_{(p_1,p_2)}=|x_1|^{p_1}|x_2|^{p_2}}$ where $p,p_1,p_2 \geq 0$. Those functions are members of the following more general classes of functions; compare Section 3.4.1 in \cite{JacPro12}.
\begin{definition}\label{def_pos_hom}
A function $f:\mathbb{R}^d \rightarrow \mathbb{R}$ is \textit{called positively homogeneous of degree $p\geq 0$}, if $f(\lambda x)=\lambda^p f(x)$ for all $x \in \mathbb{R}^d$ and $\lambda\geq 0$. Further $f$ is \textit{called positively homogeneous with degree $p_i \geq 0$ in the $i$-th argument} if the function $ x \mapsto f(x_1, \ldots, x_{i-1},x,x_{i+1}, \ldots ,x_d)$ is positively homogeneous of degree $p_i$ for any choice of $(x_1, \ldots, x_{i-1},x_{i+1}, \ldots ,x_d)\in \mathbb{R}^{d-1}$. \qed
\end{definition}

If the function $f$ is positively homogeneous with degree $p_1$ in the first argument and with degree $p_2$ in the second argument, \eqref{Vbar_mot} becomes
\begin{align*}
n^{(p_1+p_2)/2-1}\sum_{i,j:t_{i,n}^{(1)} \vee t_{j,n}^{(2)} \leq T} f( \Delta_{i,n}^{(1)} X^{(1)},\Delta_{j,n}^{(2)} X^{(2)})\mathds{1}_{\{\mathcal{I}_{i,n}^{(1)} \cap \mathcal{I}_{j,n}^{(2)}\neq \emptyset\}}. 
\end{align*} 
As we are going to derive results only for such functions we denote by
\begin{align}\label{def_norm_func_2d}
\overline{V}(p,f,\pi_n)_T=n^{p/2-1}\sum_{i,j:t_{i,n}^{(1)}\vee t_{j,n}^{(2)} \leq T} f(\Delta_{i,n}^{(1)}X^{(1)},\Delta_{j,n}^{(2)}X^{(2)}) \mathds{1}_{\{\mathcal{I}_{i,n}^{(1)} \cap \mathcal{I}_{j,n}^{(2)} \neq \emptyset\}},
\end{align}
$f:\R^2 \rightarrow \R$, the functional whose asymptotics we are going to study in this section. Further we set
\begin{align}\label{def_norm_func_1d}
\overline{V}^{(l)}(p,g,\pi_n)_T=n^{p/2-1}\sum_{i:t_{i,n}^{(l)}\leq T} g(\Delta_{i,n}^{(l)}X^{(l)}),~l=1,2,
\end{align}
for functions $g:\R \rightarrow \R$. As in Section \ref{sec:nonnormFunc} we will also derive an asymptotic result for $\overline{V}^{(l)}(p,g,\pi_n)_T$ to compare the results in the setting of asynchronously observed bivariate processes to those in simpler settings. 

To describe the limits of the normalized functionals \eqref{def_norm_func_2d} and \eqref{def_norm_func_1d} in the upcoming results we need to introduce some notation. Denote by
\begin{align*}
m_\Sigma(h)=\mathbb{E}[h(Z)],~Z\sim N(0,\Sigma),
\end{align*}
the expectation of a function $h:\mathbb{R}^d \rightarrow \mathbb{R}$ evaluated at a $d$--dimensional centered normal distributed random variable with covariance matrix $\Sigma$. Further we define the expressions
\begin{equation}\label{defintion_Gp_Hp}
\begin{aligned}
&G^{(l),n}_p(t)=n^{p/2-1}\sum_{i:t_{i,n}^{(l)}\leq t} \big|\mathcal{I}_{i,n}^{(l)}\big|^{p/2}, 
\\&G_{p_1,p_2}^n(t)=n^{(p_1+p_2)/2-1}\sum_{i,j:t_{i,n}^{(1)} \vee t_{j,n}^{(2)}\leq t} |\mathcal{I}_{i,n}^{(1)}|^{p_1/2}|\mathcal{I}_{j,n}^{(2)}|^{p_2/2}
\mathds{1}_{\{\mathcal{I}_{i,n}^{(1)}\cap \mathcal{I}_{j,n}^{(2)} \neq \emptyset\}},
\\&H_{k,m,p}^n(t)=n^{p/2-1}\sum_{i,j:t_{i,n}^{(1)} \vee t_{j,n}^{(2)}\leq t}|\mathcal{I}_{i,n}^{(1)} \setminus \mathcal{I}_{j,n}^{(2)}|^{k/2}|\mathcal{I}_{j,n}^{(2)} \setminus \mathcal{I}_{i,n}^{(1)}|^{m/2} 
\\ &~~~~~~~~~~~~~~~~~~~~~~~~~~~~~~~~~~~~~\times|\mathcal{I}_{i,n}^{(1)}\cap\mathcal{I}_{j,n}^{(2)}|^{(p-(k+m))/2}
\mathds{1}_{\{\mathcal{I}_{i,n}^{(1)}\cap \mathcal{I}_{j,n}^{(2)} \neq \emptyset\}},
\end{aligned}\end{equation}
whose limits, if they exist, will occur in the limits of $\overline{V}^{(l)}(p,g,\pi_n)_T$, $\overline{V}(p,f,\pi_n)_T$.

We start with a result for $\overline{V}(p,f,\pi_n)_T$ under the restriction that the observations are synchronous as in \eqref{synch_times}. If we apply a function $f:\mathbb{R}^2 \rightarrow \mathbb{R}$ which is positively homogeneous of degree $p$ to our toy example \eqref{example_2dim_BM} we get 
\begin{align*}
\mathbb{E}[\overline{V}(p,f,\pi_n)_T|\mathcal{S}]&= n^{p/2-1}\sum_{i: t_{i,n} \leq T} |\mathcal{I}_{i,n}|^{p/2} \mathbb{E}[f(|\mathcal{I}_{i,n}|^{-1/2}\Delta_{i,n}X^{toy})|\mathcal{S}]
\\&= m_{\sigma \sigma^*}(f)G_p^{(1),n}(T)
\end{align*}
where $\mathcal{S}$ denotes the $\sigma$-algebra generated by $\{\pi_n:n \in \mathbb{N}\}$. Therefore it appears to be a necessary condition that $G_p^{(1),n}(T)$ converges in order for $\overline{V}(p,f,\pi_n)_T$ to converge as well. This reasoning also carries over to the case of non-constant $\sigma_s$ via an approximation of $\sigma_s$ by piecewise constant stochastic processes. We then obtain the following result which covers the whole class of positively homogeneous functions $f: \mathbb{R}^2 \rightarrow \mathbb{R}$.

\begin{theorem}\label{theo_norm_2d_sync}
Let $p \geq 0$ and suppose that Condition \ref{cond_consistency} is fulfilled, and that the observation scheme is exogenous and synchronous. Further assume that
\begin{align}\label{theo_norm_2d_sync_cond}
G^{(1),n}_p(t)=G^{(2),n}_p(t) \overset{\mathbb{P}}{\longrightarrow} G_p(t),~~~ t \in [0,T],
\end{align}
for a (possibly random) continuous function $G_p:[0,T]\rightarrow \mathbb{R}_{\geq 0}$ and that we have one of the following two conditions:
\begin{itemize}
\item[a)] $p \in [0,2)$,
\item[b)] $p \geq 2$ and $X$ is continuous.
\end{itemize}
Then for all continuous positively homogeneous functions $f:\mathbb{R}^2 \rightarrow \mathbb{R}$ of degree $p$ it holds that
\begin{align*}
\overline{V}(p,f,\pi_n)_T \overset{\mathbb{P}}{\longrightarrow} \int_0^T m_{c_s}(f) dG_p(s)
\end{align*}
where $c_s=\sigma_s\sigma_s^*$.
\end{theorem}

As a corollary, we directly obtain the following convergence result for the functionals $V^{(l)}(p,g,\pi_n)_T$. 

\begin{corollary}\label{theo_norm_1d}
Let $l=1$ or $l=2$, $p \geq 0$, and suppose that Condition \ref{cond_consistency} is fulfilled and that the observation scheme is exogenous. Further assume that
\begin{align}\label{theo_norm_1d_cond}
G^{(l),n}_p(t) \overset{\mathbb{P}}{\longrightarrow} G_p^{(l)}(t),~~~ t \in [0,T],
\end{align}
for a (possibly random) continuous function $G_p^{(l)}:[0,T]\rightarrow \mathbb{R}_{\geq 0}$ and that we have one of the following two conditions:
\begin{itemize}
\item[a)] $p \in [0,2)$,
\item[b)] $p \geq 2$ and $X^{(l)}$ is continuous.
\end{itemize}
Then for all positively homogeneous functions $g:\mathbb{R} \rightarrow \mathbb{R}$ of degree $p$ it holds that
\begin{align*}
\overline{V}^{(l)}(p,g,\pi_n)_T \overset{\mathbb{P}}{\longrightarrow} m_1(g)\int_0^T (\sigma_s^{(l)})^p dG_p^{(l)}(s).
\end{align*}
\end{corollary}

\begin{remark}
In Chapter 14 of \cite{JacPro12} synchronous observation schemes of the form
\begin{align*}
t_{i,n}=t_{i-1,n}+\theta^n_{t_{i-1,n}} \varepsilon_{i,n}
\end{align*}
are investigated where $\theta^n=(\theta^n_t)_{t \geq 0}$ is a strictly positive process which is adapted to the filtration $(\mathcal{F}_t^n)_{t \geq 0}$, where $\mathcal{F}_t^{n}$ denotes the smallest $\sigma$-algebra containing the filtration $(\mathcal{F}_t)_{t \geq 0}$ with respect to which $X$ is defined and which has the property that all $t_{i,n}$ are $(\mathcal{F}_t^n)_{t \geq 0}$-stopping times. $\varepsilon_{i,n}$ is supposed to be an i.i.d.\ sequence of positive random variables in $i$ for fixed $n$. The $\varepsilon_{i,n}$ are independent of the process $X$ and its components. If $n \theta^n$ converges in u.c.p.\ to some $(\mathcal{F}_t)_{t\geq 0}$-adapted process $\theta$ and the moments $\mathbb{E}[(\varepsilon_{i,n})^{p/2}]=\kappa_{p/2}^n$ converge to some $\kappa_{p/2} < \infty$ then Lemma 14.1.5 in \cite{JacPro12} yields
\begin{align*}
G_p^n(t) \overset{\mathbb{P}}{\longrightarrow} \kappa_{p/2} \int_0^t (\theta_s)^{p/2-1}ds=:G_p(t) \quad \forall t \in [0,T],
\end{align*}
and using Theorem 14.2.1 in \cite{JacPro12} we conclude
\begin{align*}
\overline{V}(p,f,\pi_n)_T \overset{\mathbb{P}}{\longrightarrow} \kappa_{p/2}\int_0^T m_{c_s}(f)  (\theta_s)^{p/2-1} ds= \int_0^T m_{c_s}(f) dG_p(s)
\end{align*}
under the assumptions on $f$ and $X$ made in Theorem \ref{theo_norm_2d_sync}. The assumptions on the observation scheme in \cite{JacPro12} are weaker than the assumptions made in this paper in the sense that the observation scheme does not have to be exogenous, but there are still strong restrictions on the law of $t_{i,n}-t_{i-1,n}$. $\theta_{t_{i-1,n}}^n$ is known in advance at time $t_{i-1,n}$  and  $(t_{i,n}-t_{i-1,n})/\theta_{t_{i,n}}^n=\varepsilon_{i,n}$ is an exogenous random variable whose law is independent of the other observation times. On the other hand our assumptions allow for observation schemes which do not fulfill the assumptions made in Chapter 14 of \cite{JacPro12}. This is due to the fact that we need no analogon to the i.i.d.\ property of the $\varepsilon_{i,n}$. $G_p^n$ in general already converges to a linear function if the ratio $\mathbb{E}[n^{p/2}|\mathcal{I}_{i,n}|^{p/2}]/\mathbb{E}[n|\mathcal{I}_{i,n}|]$ remains constant; compare Lemma 6.4 in \cite{MarVet18b}. Asynchronous observation schemes are not considered in \cite{JacPro12}. \qed
\end{remark}

In the case of asynchronous observation times it is more difficult to derive results similar to Theorem \ref{theo_norm_2d_sync}. For functions $f$ which are like $f_{(p_1,p_2)}$, $\overline{f}_{(p_1,p_2)}$ positively homogeneous with degree $p_1$ in the first argument and with degree $p_2$ in the second argument it holds that
\begin{multline*}
\mathbb{E}[\overline{V}(p_1+p_2,f,\pi_n)_T|\mathcal{S}]
 =n^{(p_1+p_2)/2-1}\sum_{i,j: t_{i,n}^{(1)} \vee t_{j,n}^{(2)} \leq T} |\mathcal{I}_{i,n}^{(1)}|^{p_1/2}|\mathcal{I}_{j,n}^{(2)}|^{p_2/2} 
\\\times\mathbb{E}[f(|\mathcal{I}_{i,n}^{(1)}|^{-1/2}\Delta_{i,n}^{(1)}X^{(1)},|\mathcal{I}_{j,n}^{(2)}|^{-1/2}\Delta_{j,n}^{(2)}X^{(2)})|\mathcal{S}]
\mathds{1}_{\{\mathcal{I}_{i,n}^{(1)} \cap \mathcal{I}_{j,n}^{(2)} \neq \emptyset \}}.
\end{multline*}
However unlike in the case of synchronous observation times the law of $$(|\mathcal{I}_{i,n}^{(1)}|^{-1/2}\Delta_{i,n}^{(1)}X^{(1)},|\mathcal{I}_{j,n}^{(2)}|^{-1/2}\Delta_{j,n}^{(2)}X^{(2)})$$ is in general not independent of $\pi_n$ as explained in \eqref{overlap_depend}. The process $X^{toy}$ from \eqref{example_2dim_BM} has the simplest form of all processes for which the functionals discussed in this section yield a non-trivial limit and hence it makes sense to first investigate conditions which grant convergence of $\overline{V}(f,\pi_n)_T$ for $X^{toy}$. Also, the arguments used in the proof of Theorem \ref{theo_norm_2d_sync} rely on an approximation of $\sigma_s$ by piecewise constant processes in time, which then makes it possible to use results for processes like $X^{toy}$ with a constant $\sigma_s$. In particular, we need that $f(\Delta_{i,n}^{(1)}X^{toy,(1)},\Delta_{j,n}^{(2)}X^{toy,(2)})$ factorizes into a term that depends only on $\mathcal{S}$ and a term that is independent of $\mathcal{S}$. This technique of proof can be extended to the asynchronous setting whenever we can find a fixed natural number $N \in \mathbb{N}$ and functions $g_k, h_k$ for $k=1,\ldots,N$ such that we can write
\begin{align}\label{sum_factor}
\mathbb{E}[f(\Delta_{i,n}^{(1)}X^{toy,(1)},\Delta_{j,n}^{(2)}X^{toy,(2)})|\mathcal{S}]
=\sum_{k=1}^N g_k(|\mathcal{I}_{i,n}^{(1)}|,|\mathcal{I}_{j,n}^{(2)}|,|\mathcal{I}_{j,n}^{(2)}\cap \mathcal{I}_{j,n}^{(2)}|)h_k(\sigma^{(1)},\sigma^{(2)},\rho)~~~~~~
\end{align}
in our toy example, because in that case we have
\begin{multline*}
\mathbb{E}[\overline{V}(p_1+p_2,f,\pi_n)_T|\mathcal{S}]
 =\sum_{k=1}^K  h_k(\sigma^{(1)},\sigma^{(2)},\rho)\\ \times n^{(p_1+p_2)/2-1}\sum_{i,j: t_{i,n}^{(1)} \vee t_{j,n}^{(2)} \leq T} g_k(|\mathcal{I}_{i,n}^{(1)}|,|\mathcal{I}_{j,n}^{(2)}|,|\mathcal{I}_{j,n}^{(2)}\cap \mathcal{I}_{j,n}^{(2)}|)\mathds{1}_{\{\mathcal{I}_{i,n}^{(1)} \cap \mathcal{I}_{j,n}^{(2)} \neq \emptyset \}}
\end{multline*}
where the right hand side converges if we assume that the expression in the last line converges as $n \rightarrow \infty$ for each $k=1,\ldots,K$. 

It is in general not possible to find a representation of the form \eqref{sum_factor} for an arbitrary function $f$ which is positively homogeneous in both arguments. However, there are two interesting cases where such a representation is available. The first case is when $X^{toy,(1)}$ and $X^{toy,(2)}$ are uncorrelated, i.e.\ if $\rho \equiv 0$ on $[0,T]$, because then
\begin{align*}
\mathbb{E}[f(\Delta_{i,n}^{(1)}X^{toy,(1)},\Delta_{j,n}^{(2)}X^{toy,(2)})|\mathcal{S}]
=(\sigma^{(1)} )^{p_1}(\sigma^{(2)})^{p_2}\mathbb{E}[f(Z,Z')]|\mathcal{I}_{i,n}^{(1)}|^{p_1/2}|\mathcal{I}_{j,n}^{(2)}|^{p_2/2}
\end{align*}
holds for independent standard normal random variables $Z,Z'$. The result obtained in this case is stated in Theorem \ref{theo_norm_2d_uncor}. In the second case we consider the functions $f_{(p_1,p_2)}$ with $p_1,p_2 \in \mathbb{N}_0$. Indeed it holds  
\begin{multline*}
\hspace{-0.3cm} f(\Delta_{i,n}^{(1)}X^{{toy},(1)},\Delta_{j,n}^{(2)}X^{{toy},(2)}) \overset{\mathcal{L}_\mathcal{S}}{=} (\sigma^{(1)} )^{p_1}(\sigma^{(2)})^{p_2}(|\mathcal{I}_{i,n}^{(1)}\setminus \mathcal{I}_{j,n}^{(2)}|^{1/2} Z_1+|\mathcal{I}_{i,n}^{(1)}\cap \mathcal{I}_{j,n}^{(2)}|^{1/2} Z_2)^{p_1}
\\ \times (|\mathcal{I}_{i,n}^{(1)}\cap \mathcal{I}_{j,n}^{(2)}|^{1/2}( \rho Z_2+\sqrt{1-\rho^2} Z_3)+|\mathcal{I}_{j,n}^{(2)}\setminus \mathcal{I}_{i,n}^{(1)}|^{1/2} Z_4)^{p_2} 
\end{multline*}
where $\overset{\mathcal{L}_\mathcal{S}}{=}$ denotes equality of the $\mathcal{S}$-conditional law and $Z_1,\ldots,Z_4$ are i.i.d.\ standard normal random variables. Here the right hand side can be brought into the form \eqref{sum_factor} using the multinomial theorem. The result obtained in this case is stated in Theorem \ref{theo_norm_2d_integer}. Unfortunately, for the functions $\overline{f}_{(p_1,p_2)}$ there exists no similar representation: however, for even $p_1,p_2$ we have $f_{(p_1,p_2)}=\overline{f}_{(p_1,p_2)}$ and Theorem \ref{theo_norm_2d_integer} also applies there.

\begin{theorem}\label{theo_norm_2d_uncor}
Let $p_1,p_2\geq 0$ and suppose that Condition \ref{cond_consistency} is fulfilled, that the observation scheme is exogenous and that we have $\rho \equiv 0$ on $[0,T]$. Further assume that
\begin{align}\label{theo_norm_2d_uncor_cond}
G_{p_1,p_2}^n(t) \overset{\mathbb{P}}{\longrightarrow} G_{p_1,p_2}(t),~ t \in [0,T],
\end{align}
for a (possibly random) continuous function $G_{p_1,p_2}:[0,T]\rightarrow \mathbb{R}_{\geq 0}$ and that we have one of the following two conditions:
\begin{itemize}
\item[a)] $p_1+p_2 \in [0,2)$,
\item[b)] $p_1+p_2 \geq 2$ and $X$ is continuous.
\end{itemize}
Then for all continuous functions $f:\mathbb{R}^2 \rightarrow \mathbb{R}$ which are positively homogeneous of degree $p_1$ in the first argument and positively homogeneous of degree $p_2$ in the second argument it holds that
\begin{align}\label{theo_norm_2d_uncor_cons}
\overline{V}(p_1+p_2,f,\pi_n)_T \overset{\mathbb{P}}{\longrightarrow} m_{I_2}(f)\int_0^T (\sigma_s^{(1)})^{p_1} (\sigma_s^{(2)})^{p_2} d{G_{(p_1,p_2)}}(s).
\end{align}
Here $I_2$ denotes the two-dimensional identity matrix.
\end{theorem}

\begin{theorem}\label{theo_norm_2d_integer}
Let $p_1,p_2 \in \mathbb{N}_0$ and suppose that Condition \ref{cond_consistency} is fulfilled and that the observation scheme is exogenous. Define
\begin{align*}
L(p_1,p_2)= \{ (k,l,m) \in (2 \mathbb{N}_0)^3:k \leq p_1 , l+m \leq p_2 , p_1+p_2-(k+l+m) \in 2 \mathbb{N}_0\}.
\end{align*}
Assume that for all $k,m \in \mathbb{N}_0$ for which an $l\in \mathbb{N}_0$ exists with $(k,l,m) \in L(p_1,p_2)$ there exist (possibly random) continuous functions $H_{(k,m,p_1+p_2)}:[0,T]\rightarrow \mathbb{R}_{\geq 0}$ which fulfill
\begin{align}\label{theo_norm_2d_integer_cond}
H_{k,m,p_1+p_2}^n(t)
\overset{\mathbb{P}}{\longrightarrow} H_{k,m,p_1+p_2}(t),~~~t \in [0,T].
\end{align}
Further we assume that we have one of the following two conditions:
\begin{itemize}
\item[a)] $p_1+p_2 \in \{0,1\}$,
\item[b)] $p_1+p_2 \geq 2$ and $X$ is continuous.
\end{itemize}
 Then
\begin{multline}\label{theo_norm_2d_integer_cons}
\overline{V}(p_1+p_2,f_{(p_1,p_2)},\pi_n)_T 
\\ \overset{\mathbb{P}}{\longrightarrow} \int_0^T (\sigma_s^{(1)})^{p_1} (\sigma_s^{(2)})^{p_2}\Big(\sum_{(k,l,m) \in L(p_1,p_2)} \binom{p_1}{k}\binom{p_2}{l,m}
 m_1(x^k) m_1(x^l)m_1(x^m)\\ \times m_1(x^{p_1+p_2-(k+l+m)}) (1-\rho_s^2)^{l/2}(\rho_s)^{p_2-(l+m)} d{H_{k,m,p_1+p_2}}(s)\Big)
\end{multline}
holds. Here $\binom{p_2}{l,m}$ stands for the multinomial coefficient $p_2!/(l!m!(p_2-l-m)!)$.
\end{theorem}

\begin{example}\label{theo_norm_2d_integer_example} In this example we will use Theorem \ref{theo_norm_2d_integer} to find the limit of ${\overline{V}(p_1+p_2,f_{(p_1,p_2)},\pi_n)_T}$ for a few non-trivial cases with small $p_1,p_2$. For $p_1=p_2=1$ the set $L(1,1)$ contains only $(0,0,0)$ and we get 
\begin{align*}
\overline{V}(2,f_{(1,1)},\pi_n)\overset{\mathbb{P}}{\longrightarrow}\int_0^T \rho_s \sigma^{(1)}_s \sigma^{(2)}_s d{H_{0,0,2}}(s)=\int_0^T \rho_s \sigma^{(1)}_s \sigma^{(2)}_s ds
\end{align*}
as $H_{0,0,2}(t)=t$. Hence $\overline{V}(2,f_{(1,1)},\pi_n)$ converges to the covariation of $X^{(1)},X^{(2)}$ for continuous processes $X$ and we have retrieved \eqref{hayyos_est} for continuous semimartingales $X$. For $p_1=1$, $p_2=2$ we get $L(1,2)=\emptyset$ as $k,l,m$ and $3-(k+l+m)$ cannot be all divisible by $2$. Hence $\overline{V}(3,f_{(1,2)},\pi_n)\overset{\mathbb{P}}{\longrightarrow}0$. This holds for all $(p_1,p_2)$ where $p_1+p_2$ is odd.

Further we define
\begin{align*}
G_{k,m,p}^n(t)= n^{p/2-1}\sum_{i,j:t_{i,n}^{(1)} \vee t_{j,n}^{(2)}\leq t}|\mathcal{I}_{i,n}^{(1)} |^{k/2}|\mathcal{I}_{j,n}^{(2)} |^{m/2} 
|\mathcal{I}_{i,n}^{(1)}\cap\mathcal{I}_{j,n}^{(2)}|^{(p-k-m)/2}
\mathds{1}_{\{\mathcal{I}_{i,n}^{(1)}\cap \mathcal{I}_{j,n}^{(2)} \neq \emptyset\}}
\end{align*}
and define $G_{k,m,p}(t)$ as the limit of $G_{k,m,p}^n(t)$ in probability as $n \rightarrow \infty$ if it exists. For $p_1=p_2=2$ we have to consider the set $L(2,2)=\{0,2\}\times \{(0,0),(0,2),(2,0)\}$ and then obtain using the above notation 
\begin{align}\label{theo_norm_2d_integer_example_22}
\overline{V}(4,f_{(2,2)},\pi_n)&\overset{\mathbb{P}}{\longrightarrow}\int_0^T (\sigma^{(1)}_s \sigma^{(2)}_s)^2(3\rho_s^2d{H_{0,0,4}}(s)+d{H_{0,2,4}}(s)+(1-\rho_s^2)d{H_{0,0,4}}(s) \nonumber
\\&~~~~~~~~~~~~~~~~~~~~~~~~~~+\rho_s^2d{H_{2,0,4}}(s)+d{H_{2,2,4}}(s)+(1-\rho_s^2)d{H_{2,0,4}}(s)) \nonumber
\\&=\int_0^T (\sigma^{(1)}_s \sigma^{(2)}_s)^2(2\rho_s^2d{H_{0,0,4}}(s)+d{G_{2,2,4}}(s))
\end{align}
where we used $G_{2,2,p}(s)=H_{0,0,p}(s)+H_{0,2,p}(s)+H_{2,0,p}(s)+H_{2,2,p}(s)$, $p \geq 4$, which follows from the identity
\begin{align*}
|\mathcal{I}_{i,n}^{(1)} ||\mathcal{I}_{j,n}^{(2)} |
=(|\mathcal{I}_{i,n}^{(1)}\cap\mathcal{I}_{j,n}^{(2)}|+|\mathcal{I}_{i,n}^{(1)}\setminus \mathcal{I}_{j,n}^{(2)}|)(|\mathcal{I}_{i,n}^{(1)}\cap\mathcal{I}_{j,n}^{(2)}|+|\mathcal{I}_{j,n}^{(2)}\setminus \mathcal{I}_{i,n}^{(1)}|).
\end{align*}
The convergence \eqref{theo_norm_2d_integer_example_22} has already been shown in Proposition A.2 of \cite{MarVet18a}. 

Without presenting detailed computations we state two more results to demonstrate that the limit in Theorem \ref{theo_norm_2d_integer} after simplification sometimes has a much simpler representation compared to the general form in \eqref{theo_norm_2d_integer_cons}. For $p_1=p_2=3$ we have $L(3,3)=L(2,2)$ and we get after simplification
\begin{align*}
\overline{V}(6,f_{(3,3)},\pi_n)\overset{\mathbb{P}}{\longrightarrow}\int_0^T (\sigma^{(1)}_s \sigma^{(2)}_s)^3(6\rho_s^3d{H_{0,0,6}}(s)+9 \rho_sd{G_{2,2,6}}(s))
\end{align*}
and for $p_1=p_2=4$ we obtain
\begin{multline*}
\overline{V}(8,f_{(4,4)},\pi_n)
\overset{\mathbb{P}}{\longrightarrow}\int_0^T (\sigma^{(1)}_s \sigma^{(2)}_s)^4(24\rho_s^4d{H_{0,0,8}}(s)+72\rho_s^2d{G_{2,2,8}}(s)+9 d{G_{4,4,8}}
(s)).
\end{multline*} \qed
\end{example}

Two not very difficult generalizations can be made for the statements in Theorems \ref{theo_norm_2d_sync}, \ref{theo_norm_2d_uncor} and \ref{theo_norm_2d_integer}. The previous results were only stated in a more specific form to keep the notation and the proofs clearer and to direct the reader's focus to the key aspects. First throughout this section the rate $n$ was chosen rather arbitrarily as the appropriate scaling factor by which the average interval lengths decrease and such that we obtain convergence for the functions $G^{(l),n}_p(t)$, $G^{n}_{p_1,p_2}(t)$ and $H^{n}_{k,m,p}(t)$.

\begin{remark}\label{remark_general_rate}
Let $r: \mathbb{N} \rightarrow [0, \infty)$ be a function with $r(n) \rightarrow \infty$ for $n \rightarrow \infty$. Then we obtain the same result as in Corollary \ref{theo_norm_1d} if we set
\begin{align*}
&\overline{V}^{(l)}(p,g,\pi_n)_T=(r(n))^{p/2-1}\sum_{i:t_{i,n}^{(l)}\leq T} g(\Delta_{i,n}^{(l)}X^{(l)}),
\\&G^{(l),n}_p(t)=(r(n))^{p/2-1}\sum_{i:t_{i,n}^{(l)}\leq t} |\mathcal{I}_{i,n}^{(l)}X^{(l)}|^{p/2}.
\end{align*}
Similarly the results from Theorems \ref{theo_norm_2d_sync}, \ref{theo_norm_2d_uncor} and \ref{theo_norm_2d_integer} hold as well if we replace $n$ by $r(n)$ in the definition of the functional $\overline{V}(p,f,\pi_n)_T$ and the functions $G^{n}_{p_1,p_2}(t)$, $H^{n}_{k,m,p}(t)$. The proofs for these claims are identical to the proofs in the more specific case $r(n)=n$. Hence we only need that the observation scheme scales with a deterministic rate $r(n)$ to obtain the results in this section. \qed
\end{remark}

Further, only increments of the continuous martingale part of $X$ contribute to the limits in Theorems \ref{theo_norm_2d_sync}, \ref{theo_norm_2d_uncor} and \ref{theo_norm_2d_integer} and the increments of the continuous part of $X$ tend to get very small as the observation intervals become shorter. Hence only function evaluations $f(x)$ at very small $x$ and especially the behaviour of $f(x)$ for $x \rightarrow 0$ has an influence on the asymptotics. These arguments motivate that the convergences in the above theorems do not only hold for positively homogeneous functions but as the following corollary shows also for functions $f$ which are very close to being positively homogeneous for $x \rightarrow 0$; compare Corollary 3.4.3 in \cite{JacPro12}.

\begin{corollary}\label{cor_norm_almost_poshom}
Suppose that the convergence in one of the Theorems \ref{theo_norm_2d_sync}, \ref{theo_norm_2d_uncor} and \ref{theo_norm_2d_integer} holds for the function $f:\mathbb{R}^2 \rightarrow \mathbb{R}$. Then the corresponding convergence also holds for all functions $\tilde{f}:\R^2 \rightarrow \R$ which can be written as 
\begin{align*}
\tilde{f}(x)=L(x)f(x)
\end{align*}
for a locally bounded function $L(x)$ that fulfills $\lim_{x \rightarrow 0}L(x)=1$.
\end{corollary}

\begin{example}\label{example_norm_pois}
In the setting of Poisson sampling the assumptions of Theorems \ref{theo_norm_2d_sync}, \ref{theo_norm_2d_uncor} and \ref{theo_norm_2d_integer} are fulfilled. Indeed the functions $G_p^{(l),n},G_{p_1,p_2}^{n}, H_{k,m,p}^n$ converge using similar arguments as in the proof of Lemma 6.5 in \cite{MarVet18b} to deterministic linear and hence continuous functions. Although the functions $G_p^{(l)},G_{p_1,p_2}, H_{k,m,p}$ are in general unknown they can easily be estimated by $G_p^{(l),n},G_{p_1,p_2}^{n}, H_{k,m,p}^n$. \qed
\end{example}

\section{Outlook}\label{sec:outlook}
\def\theequation{5.\arabic{equation}}
\setcounter{equation}{0}

We have seen that it is possible to generalize classical results in the field of high-frequency statistics based on synchronous observation data to the setting of asynchronous observations, and the related papers \cite{BibVet15} and \cite{MarVet18a, MarVet18b} demonstrate that statistical inference can be based on functionals using asynchronously observed data as well. Although the results and statistical procedures become more complicated in the setting of irregular and asynchronous observations it is beneficial to develop such methods as they are more efficient in practical applications because no additional synchronization steps are needed. In this context we believe that it is worthwhile to further investigate statistical methods which directly work with irregular and asynchronous data, and many more generalizations of well-known results from the setting of equidistant and synchronous data are needed, as the functionals $V(f,\pi_n)_T$ and $\overline{V}(p,f,\pi_n)_T$ are among the simplest statistics discussed in the context of high-frequency observations.

Various generalizations and extensions of these functionals have been studied theoretically (compare \cite{JacPro12}) and used for numerous applications (compare \cite{AitJac14}). Examples include sums of functionals of truncated increments where $f(\Delta_{i,n} X)$ is only included in the sum if $\|\Delta_{i,n} X\|$ lies below or exceeds a certain threshold which may or may not depend on the length of the corresponding observation interval. Further, sums of functionals evaluated at multiple consecutive observation intervals $f(\Delta_{i-k,n} X,\Delta_{i-k+1,n} X,\ldots,\Delta_{i,n} X)$ have been studied and can e.g.\ be used for the estimation of integrated volatility; compare \cite{Barnetal06}. It would be important to investigate how analogous versions based on asynchronous observations should be defined, and whether these versions behave differently in the asymptotics when compared to the corresponding functionals based on synchronous and often equidistant observations.

Further, the laws of large numbers found here for the functionals $V(f,\pi_n)_T$ and $\overline{V}(p,f,\pi_n)_T$ should be accompanied by corresponding central limit theorems. Such central limit theorems are stated for general functions $f$ for the functionals $V(f,\pi_n)_T$ and $\overline{V}(p,f,\pi_n)_T$ in the setting of equidistant and synchronous observation times in Chapter 5 of \cite{JacPro12} and for the functional $\overline{V}(p,f,\pi_n)_T$ also in the setting of synchronous but irregular observation times in Chapter 14 of \cite{JacPro12}. For the function $f_{(1,1)}(x_1,x_2)=x_1 x_2$ a central limit theorem has been found for $V(f_{(1,1)},\pi_n)_T=\overline{V}(2,f_{(1,1)},\pi_n)$ in \cite{BibVet15}, and in \cite{MarVet18a} a central limit theorem for $f_{(2,2)}(x_1,x_2)=(x_1)^2(x_2)^2$ has been developed for $V(f_{(2,2)},\pi_n)_T$ under the assumption that no common jumps exist (then it holds $V(f^*,\pi_n)_T \rightarrow 0$). In general, central limit theorems in the setting of irregular and asynchronous observation times are much more complex compared to those in the equidistant and synchronous setting. Their asymptotic variances do not only have a more complicated structure depending on the observation scheme but might also contain additional terms which represent how fast the functions $G^{(l),n}_p(t)$, $G^{n}_{p_1,p_2}(t)$ and $H^{n}_{k,m,p}(t)$ converge; compare Theorem 14.3.2 in \cite{JacPro12}.

\section{Proofs}\label{sec:proofs}
\def\theequation{6.\arabic{equation}}
\setcounter{equation}{0}

\subsection{Preliminaries}

In the following we will assume without loss of generality that the processes $b_t,\sigma_t,\Gamma_t$ are bounded. They are locally bounded by Condition \ref{cond_consistency}. A localization
procedure then shows that the results for bounded processes can be carried over to the case of locally bounded processes; compare Section 4.4.1 in \cite{JacPro12}.

We introduce the decomposition $X_t=X_0+B(q)_t+C_t+M(q)_t+N(q)_t$ of the It\^o semimartingale \eqref{ItoSemimart} with
\begin{align}\label{decomposition_ito}
\begin{split}
B(q)_t&=\int_0^t \big( b_s -\int_{\mathbb{R}^2}(\delta(s,z)\mathds{1}_{\{\|\delta(s,z)\|\leq 1\}}-\delta(s,z)\mathds{1}_{\{\gamma(z)\leq 1/q\}})\lambda(dz)\big)ds,
\\C_t &= \int_0^t \sigma_s dW_s,
\\M(q)_t&=\int_0^t \int_{\mathbb{R}^2} \delta(s,z) \mathds{1}_{\{\gamma(z)\leq 1/q\}}(\mu-\nu)(ds,dz),
\\N(q)_t&=\int_0^t \int_{\mathbb{R}^2} \delta(s,z) \mathds{1}_{\{\gamma(z)>1/q\}}\mu(ds,dz).
\end{split}
\end{align}
Here $q$ is a parameter which controls whether jumps are classified as small jumps or big jumps. 

Throughout the upcoming proofs we will make repeated use of the estimates in the following lemma.

\begin{lemma}\label{elem_ineq}
If Condition \ref{cond_consistency} is fulfilled and the processes $b_t,\sigma_t,\Gamma_t$ are bounded there exist constants $K_p, K_{p'},K_{p,q},\widetilde{K}_{p,q},e_q \geq 0$ such that
\begin{align}
&\|B(q)_{s+t}-B(q)_s\|^p \leq K_{p,q} t^p, \label{elem_ineq_B}
\\ &\mathbb{E} \big[ \|C_{s+t}-C_{s}\|^p|\mathcal{F}_s\big] \leq K_p t^{\frac{p}{2}}, \label{elem_ineq_C}
\\ &\mathbb{E} \big[ \|M(q)_{s+t}-M(q)_{s}\|^{p}|\mathcal{F}_s\big] \leq K_{p} t^{\frac{p}{2} \wedge 1} (e_q)^{\frac{p}{2} \wedge 1}, \label{elem_ineq_M}
\\ &\mathbb{E} \big[ \|N(q)_{s+t}-N(q)_{s}\|^{p'}|\mathcal{F}_s\big] \leq \widetilde{K}_{p',q} t+K_{p',q}t^{p'},  \label{elem_ineq_N}
\\ &\mathbb{E} \big[ \|X_{s+t}-X_{s}\|^{p}|\mathcal{F}_s\big] \leq K_{p} t^{\frac{p}{2} \wedge 1},  \label{elem_ineq_X}
\end{align}
for all $s , t \geq 0$ with $s+t \leq T$ and all $q >0$, $p \geq 0$, $p' \geq 1$. Here, $e_q$ can be chosen such that $e_q \rightarrow 0$ for $q \rightarrow \infty$. For $p' \geq 2$ the constant $\widetilde{K}_{p',q}$ may be chosen independently of $q$.
\end{lemma}

\begin{proof}
The inequalities \eqref{elem_ineq_B}--\eqref{elem_ineq_X} follow from Condition \ref{cond_consistency}, inequalities (2.1.33), (2.1.34), (2.1.37), (2.1.41) in \cite{JacPro12} and Jensen's inequality.
\end{proof}

\subsection{Proofs for Section \ref{sec:nonnormFunc}}\label{subsection:lln_nonnorm_proof}

As a preparation for the proof of Theorem \ref{general_22} we prove \eqref{cons_22} for functions $f$ that vanish in a neighbourhood of the two axes $\{(x,y)\in \mathbb{R}^2|xy=0\}$.

\begin{lemma}\label{lemma_vanish} Under Condition \ref{cond_consistency} we have
\begin{align*}
V(f,\pi_n)_T \overset{\mathbb{P}}{\longrightarrow} B^*(f)_T
\end{align*}
for all continuous functions $f:\mathbb{R}^2 \rightarrow \mathbb{R}$ which vanish on the set \mbox{$\{(x,y) \in \mathbb{R}^2:|xy|< \rho\}$} for some $\rho > 0$.
\end{lemma}

\begin{proof}
The following arguments are similar to the proof of Lemma 3.3.7 in \cite{JacPro12}: Note that as $X$ is c\`adl\`ag there can only exist countably many jump times $s \geq 0$ with $|\Delta X^{(1)}_{s}\Delta X^{(2)}_{s}|\geq\rho/2$ and in each compact time interval there are only finitely many such jumps. Denote by $(S_p)_{p \in \mathbb{N}}$ an enumeration of those jump times and let 
\begin{align*}
\widetilde{X}_t=X_t-\int_0^t \int_{\mathbb{R}^2} \delta(s,z) \mathds{1}_{\{|\delta^{(1)}(s,z)\delta^{(2)}(s,z)|\geq \rho/2\}}\mu(ds,dz)
\end{align*}
denote the process $X$ without those jumps. This yields $|\Delta \widetilde{X}^{(1)}_s\Delta \widetilde{X}^{(2)}_s|< \rho/2$ for all $s \in [0,T]$. Hence
\begin{gather*}
~~\limsup_{\theta \rightarrow 0} \sup_{0 \leq s_l \leq t_l \leq T-|\pi_n|_T, t_l-s_l \leq \theta, (s_1,t_1]\cap (s_2,t_2] \neq \emptyset} |(\widetilde{X}^{(1)}_{t_1}(\omega)- \widetilde{X}^{(1)}_{s_1}(\omega))( \widetilde{X}^{(2)}_{t_2}(\omega)- \widetilde{X}^{(2)}_{s_2}(\omega))|  < \frac{\rho}{2}~~
\end{gather*}
for all $\omega \in \Omega$. Then there exists $\theta':\Omega \rightarrow (0, \infty)$ such that 
\[
\sup_{0 \leq s_l \leq t_l \leq T, t_l-s_l \leq \theta'(\omega), (s_1,t_1]\cap (s_2,t_2] \neq \emptyset} |(\widetilde{X}^{(1)}_{t_1}(\omega)- \widetilde{X}^{(1)}_{s_1}(\omega))( \widetilde{X}^{(2)}_{t_2}(\omega)- \widetilde{X}^{(2)}_{s_2}(\omega))|    < \rho. \] 
Denote by $\Omega(n)$ the subset of $\Omega$ which is defined as the intersetion of the set ${\{|\pi_n|_T \leq \theta'\}}$ and the set on which any two different jump times $S_p \neq S_{p'}$ with $S_{p},S_{p'} \leq T$ satisfy $|S_{p'}-S_{p}|>2 |\pi_n|_T$ and on which $|T-S_p|>|\pi_n|_T$ for any $S_p \leq T$. Then we have
\begin{align}\label{lemma_vanish_proof1}
V(f,\pi_n)_T \mathds{1}_{\Omega(n)}=\sum_{p:S_p \leq T-|\pi_n|_T} f\big(\Delta_{i_n^{(1)}(S_p),n}^{(1)}X^{(1)},\Delta_{i_n^{(2)}(S_p),n}^{(2)} X^{(2)} \big) \mathds{1}_{\Omega(n)}
\end{align}
where $i_n^{(l)}(s)$ denotes the index of the interval characterized by $s \in \mathcal{I}_{i_n^{(l)}(s),n}^{(l)}$. Further we get from Condition \ref{cond_consistency} 
\begin{align*}
f\big(\Delta_{i_n^{(1)}(S_p),n}^{(1)}X^{(1)},\Delta_{i_n^{(2)}(S_p),n}^{(2)} X^{(2)} \big)\mathds{1}_{\{S_p \leq T\}} \overset{\mathbb{P}}{\longrightarrow} f \big(\Delta X^{(1)}_{S_p},\Delta X^{(2)}_{S_p}\big)\mathds{1}_{\{S_p \leq T\}}
\end{align*}
for any $p \in \mathbb{N}$ because $X$ is c\`adl\`ag and $f$ is continuous. Using this convergence, the fact that there exist almost surely only finitely many $p \in \mathbb{N}$ with $S_p \leq T$ and $\mathbb{P}(\Delta X_T=0)=1$ we obtain 
\begin{multline*}
\sum_{p:S_p \leq T-|\pi_n|_T}f\big(\Delta_{i_n^{(1)}(S_p),n}^{(1)}X^{(1)},\Delta_{i_n^{(2)}(S_p),n}^{(2)} X^{(2)} \big) 
\\ \overset{\mathbb{P}}{\longrightarrow}
\sum_{s \leq T} f\big(\Delta X^{(1)}_s,\Delta X^{(2)}_s \big) \mathds{1}_{\{|\Delta X^{(1)}_s\Delta X^{(2)}_s| \geq \rho/2 \}}=B^*(f)_T, 
\end{multline*}
where the last equality holds because of $f(x,y)=0$ for $|xy|< \rho$. This yields the claim because of \eqref{lemma_vanish_proof1} and $\mathbb{P}(\Omega(n)) \rightarrow 1$ as $n \rightarrow \infty$. 
\end{proof}

\begin{proof}[Proof of Theorem \ref{general_22}]
Define $f_\rho(x,y)=f(x,y)\psi(|xy|/\rho)$ where $\psi: [0, \infty) \rightarrow [0, 1]$ is continuous with $\psi(u)=0$ for $u\leq 1/2$ and $\psi(u)=1$ for $u \geq 1$. By Lemma \ref{lemma_vanish} we have
\begin{align}\label{proof_general22_1}
V(f_\rho,\pi_n)_T \overset{\mathbb{P}}{\longrightarrow} B^*(f_\rho)_T
\end{align}
for all $\rho>0$. Because of $f(x,y)=O(x^2y^2)$ as $|xy|\rightarrow 0$ there exist constants $K_\rho>0$ with $|f_\rho(x,y)|\leq|f(x,y)| \leq K_\rho |x|^2|y|^2$ for all $(x,y)$ with $|xy|<\rho$ and $\rho \mapsto K_\rho$ is non-increasing as $\rho \rightarrow 0$. Hence it holds
\begin{align}\label{proof_general22_2}
|B^*(f)_T-B^*(f_\rho)_T|\leq 2 \sum_{s \leq T} K_\rho \big|\Delta X^{(1)}_s\big|^2\big|\Delta X^{(2)}_s\big|^2 \mathds{1}_{\{|\Delta X^{(1)}_s \Delta X^{(2)}_s|<\rho\}} \rightarrow 0
\end{align}
as $\rho \rightarrow 0$. To conclude \eqref{cons_22} we have to show
\begin{align}\label{proof_general22_3}
\lim_{\rho \rightarrow 0} \limsup_{n \rightarrow \infty}\mathbb{P}(|V(f,\pi_n)_T-V(f_\rho,\pi_n)_T|> \delta)=0
\end{align}
for any $\delta>0$ in addition to \eqref{proof_general22_1} and \eqref{proof_general22_2}. To this end we consider the following inequality which is obtained as in \eqref{proof_general22_2}
\begin{align}\label{proof_general22_3b}
|V(f,\pi_n)_T-V(f_\rho,\pi_n)_T| \nonumber
&\leq \sum_{i,j:t^{(1)}_{i,n} \vee t^{(2)}_{j,n} \leq T} 2K_\rho \big|\Delta_{i,n}^{(1)}X^{(1)} \big|^2 \big|\Delta_{j,n}^{(2)}X^{(2)} \big|^2 \\&~~~~~~~~~~~~~~~~\times \mathds{1}_{\{|\Delta_{i,n}^{(1)}X^{(1)}\Delta_{j,n}^{(2)}X^{(2)}|<\rho\}} 
\mathds{1}_{\{\mathcal{I}_{i,n}^{(1)} \cap \mathcal{I}_{j,n}^{(2)}\neq \emptyset \}}.~~~~~~~
\end{align}
Looking at the proof of (A.6) in \cite{MarVet18a} we conclude that \eqref{proof_general22_3b} vanishes in probability as first $n \rightarrow \infty$ and then $\rho \rightarrow 0$ if 
\begin{multline}\label{proof_general22_4}
\lim_{\rho \rightarrow 0} \lim_{q \rightarrow \infty} \limsup_{n \rightarrow \infty} \mathbb{P}\big(\sum_{i,j:t^{(1)}_{i,n} \vee t^{(2)}_{j,n} \leq T} 2K_\rho \big|\Delta_{i,n}^{(1)}N^{(1)}(q) \big|^2 \big|\Delta_{j,n}^{(2)}N^{(2)}(q) \big|^2 
\\ \times \mathds{1}_{\{|\Delta_{i,n}^{(1)}X^{(1)}\Delta_{j,n}^{(2)}X^{(2)}|<\rho\}} 
\mathds{1}_{\{\mathcal{I}_{i,n}^{(1)} \cap \mathcal{I}_{j,n}^{(2)}\neq \emptyset \}}   >\delta \big)=0
\end{multline}
holds for any $\delta>0$. Denote by $S_p$ the jump times of $N(q)$ and let $\Omega(n,q,\rho)$ be the set on which no two jump times $S_p,S_{p'} \leq T$ fulfill $|S_p-S_{p'}|\leq 2 |\pi_n|_T$, it holds $S_p \leq T-|\pi_n|_T$ and $$|\Delta N^{(1)}(q)_{{S}_p}\Delta N^{(2)}(q)_{{S}_p}-\Delta_{i_n^{(1)}({S}_p),n}X^{(1)}\Delta_{i_n^{(2)}({S}_p),n}X^{(2)}|\leq \rho$$ for any ${S_p \leq T}.$ Further we denote by $\widetilde{S}_p^{\rho}$ the jump times of $N(q)$ with $${|\Delta N^{(1)}(q)_{\widetilde{S}_p^{\rho}}\Delta N^{(2)}(q)_{\widetilde{S}_p^{\rho}}|<2\rho}.$$ Then it holds
\begin{align*}
 &\sum_{i,j:t^{(1)}_{i,n} \vee t^{(2)}_{j,n} \leq T} 2K_\rho \big|\Delta_{i,n}^{(1)}N^{(1)}(q) \big|^2 \big|\Delta_{j,n}^{(2)}N^{(2)}(q) \big|^2 
\\&~~~~~~~~~~~~~~~~~~~~~~~~~~~~~~~~~~~~~~~~\times \mathds{1}_{\{|\Delta_{i,n}^{(1)}X^{(1)}\Delta_{j,n}^{(2)}X^{(2)}|<\rho\}} 
\mathds{1}_{\{\mathcal{I}_{i,n}^{(1)} \cap \mathcal{I}_{j,n}^{(2)}\neq \emptyset \}} \mathds{1}_{\Omega(n,q,\rho)}
\\&~~~~\leq  2K_\rho\sum_{p:\widetilde{S}_p^{\rho}  \leq T}  \big|\Delta N^{(1)}(q)_{\widetilde{S}_p^{\rho}} \big|^2 \big|\Delta N^{(2)}(q)_{\widetilde{S}_p^{\rho}} \big|^2 
\\&~~~~~~~~~~~~~~~~~~~~~~~~~~~~~~~~~~~~~~~~\times \mathds{1}_{\{|\Delta_{i_n^{(1)}(\widetilde{S}_p^{\rho}),n}^{(1)}X^{(1)}\Delta_{i_n^{(2)}(\widetilde{S}_p^{\rho}),n}^{(2)}X^{(2)}|<\rho\}}  \mathds{1}_{\Omega(n,q,\rho)}
 \\&~~~~ \leq  2K_\rho\sum_{p:\widetilde{S}_p^{\rho}  \leq T}  \big|\Delta N^{(1)}(q)_{\widetilde{S}_p^{\rho}} \big|^2 \big|\Delta N^{(2)}(q)_{\widetilde{S}_p^{\rho}} \big|^2  \mathds{1}_{\{|\Delta N^{(1)}(q)_{\widetilde{S}_p^{\rho}} \Delta N^{(2)}(q)_{\widetilde{S}_p^{\rho}}|< 2\rho\}} \mathds{1}_{\Omega(n,q,\rho)}
  \\&~~~~ \leq  2K_\rho\sum_{s  \leq T}  \big|\Delta X^{(1)}_{s} \big|^2 \big|\Delta X^{(2)}_{s} \big|^2  \mathds{1}_{\{|\Delta X^{(1)}_{s} \Delta X^{(2)}_{s}|< 2\rho\}} 
\end{align*}
where the expression in the last line vanishes as $\rho \rightarrow 0$. Together with $\mathbb{P}(\Omega(n,q,\rho))\rightarrow 1$ as $n \rightarrow \infty $ for any $q, \rho >0$ this yields \eqref{proof_general22_4} and hence  \eqref{proof_general22_3}.
\end{proof}

\begin{proof}[Proof of Theorem \ref{general_1plusBeta}]
Set $f_\rho(x,y)=f(x,y)\psi(|xy|/\rho)$ like in the proof of Theorem \ref{general_22}. As in the proof of Theorem \ref{general_22} we obtain $V(f_\rho,\pi_n)_T \overset{\mathbb{P}}{\longrightarrow} B^*(f_\rho)_T$ and
\begin{align}\label{muirhead_trick}
&\lim_{\rho \rightarrow 0}|B^*(f)_T-B^*(f_\rho)_T|\leq \lim_{\rho \rightarrow 0}\sum_{s \leq T} 2 K_\rho |\Delta X^{(1)}_s|^{p_1}|\Delta X^{(2)}_s|^{p_2}\mathds{1}_{\{|\Delta X^{(1)}_s\Delta X^{(2)}_s|<\rho \}} \nonumber
\\&~~~~~~~~~~~~ \leq \lim_{\rho \rightarrow 0} K_\rho \sum_{s \leq T} \big( |\Delta X^{(1)}_s|^{p_1+p_2}+|\Delta X^{(2)}_s|^{p_1+p_2} \big)\mathds{1}_{\{|\Delta X^{(1)}_s\Delta X^{(2)}_s|<\rho \}} \nonumber
\\&~~~~~~~~~~~~=0
\end{align}
because of $p_1+p_2 \geq 2$ where we used Muirhead's inequality as in Theorem 45 of \cite{HaLiPo52}, which yields $2a^{p_1}b^{p_2} \leq a^{p_1+p_2}+b^{p_1+p_2}$ for any $a,b \geq0$. Hence it remains to show
\begin{align}\label{proof_general_1plusBeta_1}
\lim_{\rho \searrow 0} \limsup_{n \rightarrow \infty} \mathbb{P}(|V(f,\pi_n)_T-V(f_\rho,\pi_n)_T|> \varepsilon)=0~~~ \forall \varepsilon>0.
\end{align}
Because of $f(x,y)=o(|x|^{p_1}|y|^{p_2})$ as $|xy|\rightarrow 0$ we obtain as in \eqref{proof_general22_3b}
\begin{gather}\label{proof_general_1plusBeta_2}
|V(f,\pi_n)_T-V(f_\rho,\pi_n)_T|\leq 2K_\rho  \sum_{i,j:t_{i,n}^{(1)} \vee t_{j,n}^{(2)}\leq T} \big|\Delta_{i,n}^{(1)} X^{(1)}\big|^{p_1}\big|\Delta_{j,n}^{(2)} X^{(2)} \big|^{p_2}\mathds{1}_{\{\mathcal{I}_{i,n}^{(1)} \cap \mathcal{I}_{j,n}^{(2)}\neq \emptyset \}}
\end{gather}
with $K_\rho \rightarrow 0$ as $\rho \rightarrow 0$. Define stopping times $(T_k^n)_{k \in \mathbb{N}_0}$ via $T_0^n=0$ and $$T_k^n=\inf\{t_{i,n}^{(l)}>T_{k-1}^n|i \in \mathbb{N}_0, l=1,2\}, \quad k \geq 1.$$ Hence the $T_k^n$ mark the times, where at least one of the processes $X^{(1)}$ or $X^{(2)}$ is newly observed. Further we set 
\begin{align*}
&\tau_-^{(l)}(s)=\sup\{t_{i,n}^{(l)}\leq s| i \in \mathbb{N}_0\},
\\&\tau_+^{(l)}(s)=\inf\{t_{i,n}^{(l)}\geq s| i \in \mathbb{N}_0\}
\end{align*}
for the observation times of $X^{(l)}$, $l=1,2$, immediately before and after $s$. Then we denote
\begin{align*}
&\Delta_k^n X^{(l)}=X^{(l)}_{T_k}-X^{(l)}_{T_{k-1}},
\\&\Delta_k^{n,l,-}X^{(l)}=X^{(l)}_{T_{k-1}}-X^{(l)}_{\tau_-^{(l)}(T_{k-1}^n)},
\\&\Delta_k^{n,l,+}X^{(l)}=X^{(l)}_{\tau_+^{(l)}(T_{k}^n)}-X^{(l)}_{T_k}
\end{align*}
for $l=1,2$. Then
\begin{align*}
&\sum_{i,j:t_{i,n}^{(1)} \vee t_{j,n}^{(2)}\leq T} \big|\Delta_{i,n}^{(1)} X^{(1)}\big|^{p_1}\big|\Delta_{j,n}^{(2)} X^{(2)} \big|^{p_2}\mathds{1}_{\{\mathcal{I}_{i,n}^{(1)} \cap \mathcal{I}_{j,n}^{(2)}\neq \emptyset \}}
\\&~~~~~~ \leq \sum_{k:T_k^n \leq T} \prod_{l=1}^2 \big|\Delta_k^{n,l,-}X^{(l)}+\Delta_k^{n}X^{(l)}+\Delta_k^{n,l,+}X^{(l)} \big|^{p_l}
\\&~~~~~~ \leq \sum_{k:T_k^n \leq T}  \prod_{l=1}^2 K_{p_l}\big(\big|\Delta_k^{n,l,-}X^{(l)} \big|^{p_l}+\big|\Delta_k^{n}X^{(l)} \big|^{p_l}+\big|\Delta_k^{n,l,+}X^{(l)} \big|^{p_l}\big).
\end{align*}

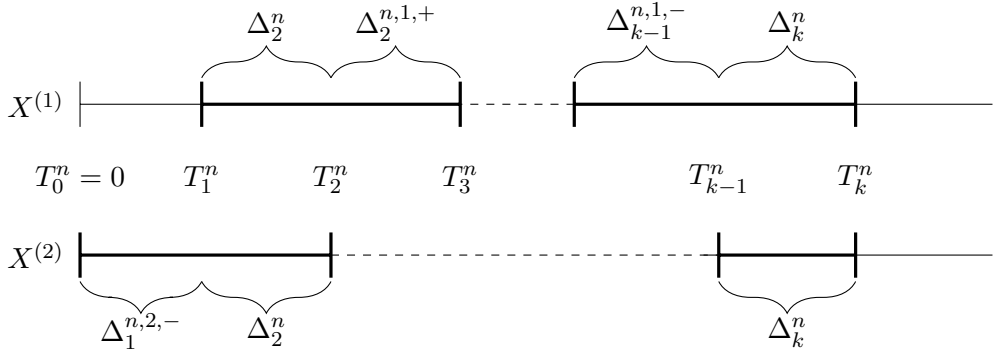
\begin{figure}[tb]
\centering
\hspace{-0.6cm}
\begin{tikzpicture}
\draw[dashed] (5,1) -- (8,1)
			(3.8,-1) -- (8.9,-1);
\draw[very thick] (2.1,1) -- (5.5,1)
		(7,1) -- (10.7,1)
		(2.1,0.7) -- (2.1,1.3)
		(5.5,0.7) -- (5.5,1.3)
		(7,0.7) -- (7,1.3)
		(10.7,0.7) -- (10.7,1.3)
		(0.5,-1) -- (3.8,-1)
		(8.9,-1) -- (10.7,-1)
		(0.5,-0.7) -- (0.5,-1.3)
		(3.8,-0.7) -- (3.8,-1.3)
		(8.9,-0.7) -- (8.9,-1.3)
		(10.7,-0.7) -- (10.7,-1.3);
\draw (0.5,1) -- (5,1)
	  (8,1) -- (12.5,1)
		(0.5,-1) -- (3.7,-1)
		(8.7,-1) -- (12.5,-1)
		(0.5,0.7) -- (0.5,1.3);
\draw (0.5,0) node{$T_0^n=0$}
		(2.1,0) node{$T_1^n$}
		(3.8,0) node{$T_2^n$}
		(5.5,0) node{$T_3^n$}
		(8.9,0) node{$T_{k-1}^n$}
		(10.7,0) node{$T_{k}^n$};	
\draw[decorate,decoration={brace,amplitude=12pt}]
	(2.1,1.35)--(3.8,1.35) node[midway, above,yshift=10pt,]{$\Delta_2^n$};
	\draw[decorate,decoration={brace,amplitude=12pt}]
	(3.8,1.35)--(5.5,1.35) node[midway, above,yshift=10pt,]{$\Delta_2^{n,1,+}$};
	\draw[decorate,decoration={brace,amplitude=12pt}]
	(7,1.35)--(8.9,1.35) node[midway, above,yshift=10pt,]{$\Delta_{k-1}^{n,1,-}$};
	\draw[decorate,decoration={brace,amplitude=12pt}]
	(8.9,1.35)--(10.7,1.35) node[midway, above,yshift=10pt,]{$\Delta_k^{n}$};
\draw[decorate,decoration={brace,amplitude=12pt}]
	(10.7,-1.35)--(8.9,-1.35) node[midway, below,yshift=-10pt,]{$\Delta_k^{n}$};
	\draw[decorate,decoration={brace,amplitude=12pt}]
	(3.8,-1.35)--(2.1,-1.35) node[midway, below,yshift=-10pt,]{$\Delta_2^{n}$};
	\draw[decorate,decoration={brace,amplitude=12pt}]
	(2.1,-1.35)--(0.5,-1.35) node[midway, below,yshift=-8pt,]{$\Delta_1^{n,2,-}$};
\draw (0.5,1) node[left,xshift=-2pt]{$X^{(1)}$}
(0.5,-1) node[left,xshift=-2pt]{$X^{(2)}$};
\end{tikzpicture}
\caption{Merged observation times and interval lengths to previous and upcoming observation times.}
\end{figure}

The $\mathcal{S}$-conditional expectation of this quantity is bounded by
\begin{align}\label{proof_general_1plusBeta_3}
&\sum_{k:T_k^n \leq T} \mathbb{E} [\prod_{l=1}^2 K_{p_l}\big(|\Delta_k^{n,l,-}X^{(l)} |^{p_l}+\big|\Delta_k^{n}X^{(l)} |^{p_l}+|\Delta_k^{n,l,+}X^{(l)} |^{p_l}\big)| \mathcal{S}] \nonumber
\\&~~ =K_{p_1}K_{p_2}\sum_{k:T_k^n \leq T} \mathbb{E} [|\Delta_k^{n}X^{(1)} |^{p_1}|\Delta_k^{n}X^{(2)} |^{p_2}| \mathcal{S}] \nonumber
\\&~~~~+K_{p_1}K_{p_2} \sum_{l=1,2} \sum_{k:T_k^n \leq T} \mathbb{E} \big[|\Delta_k^{n,l,-}X^{(l)} |^{p_l}  \nonumber
\\&~~~~~~~~~~~~~~~~~~~~\times \mathbb{E}[|\Delta_k^{n}X^{(3-l)} |^{p_{3-l}}+|\Delta_k^{n,3-l,+}X^{(3-l)} |^{p_{3-l}}|\sigma(\mathcal{F}_{T_{k-1}^n},\mathcal{S})]\big| \mathcal{S}\big]~~~~~ \nonumber
\\&~~~~+K_{p_1}K_{p_2} \sum_{l=1,2}\sum_{k:T_k^n \leq T} \mathbb{E} \big[|\Delta_k^{n}X^{(l)} |^{p_l} \mathbb{E}[|\Delta_k^{n,3-l,+}X^{(3-l)} |^{p_{3-l}}|\sigma(\mathcal{F}_{T_{k}^n},\mathcal{S})]\big| \mathcal{S}\big]
\end{align}
where we used $$|\Delta_k^{n,l,-}X^{(l)} |^{p_{l}}|\Delta_k^{n,3-l,-}X^{(3-l)} |^{p_{3-l}}=|\Delta_k^{n,l,+}X^{(l)} |^{p_{l}}|\Delta_k^{n,3-l,+}X^{(3-l)} |^{p_{3-l}}=0$$ which holds because one of each two increments is always zero. Further, using $$|\Delta_k^{n}X^{(1)}|^{p_1}|\Delta_k^{n}X^{(2)} |^{p_2}\leq 2\|\Delta_k^n X\|^{p_1+p_2},$$ which we obtain as in \eqref{muirhead_trick}, and inequality \eqref{elem_ineq_X}, \eqref{proof_general_1plusBeta_3} is bounded by
\begin{align*}
&K_{p_1}K_{p_2}\sum_{k:T_k^n \leq T} \mathbb{E} [\|\Delta_k^{n}X \|^{p_1+p_2}| \mathcal{S}] 
\\&~~~~~+K_{p_1}K_{p_2} \sum_{l=1,2} \sum_{k:T_k^n \leq T} \mathbb{E} [|\Delta_k^{n,l,-}X^{(l)} |^{p_l}   K  (\tau_+^{(3-l)}(T_{k}^n)- T_{k-1}^n)^{\frac{p_{3-l}}{2} \wedge 1}| \mathcal{S}]
\\&~~~~~+K_{p_1}K_{p_2} \sum_{l=1,2}\sum_{k:T_k^n \leq T} \mathbb{E} [|\Delta_k^{n}X^{(l)} |^{p_l} K  (\tau_+^{(3-l)}(T_{k}^n)- T_{k}^n)^{\frac{p_{3-l}}{2} \wedge 1}| \mathcal{S}]
\\&~\leq K_{p_1}K_{p_2}\sum_{k:T_k^n \leq T} K(T_{k}^n-T_{k-1}^n)
\\&~~~~~+4 K_{p_1}K_{p_2}\sum_{k:T_k^n \leq T}\prod_{l=1,2}
(\tau_+^{(l)}(T_{k}^n)-\tau_-^{(l)}(T_{k-1}^n))^{\frac{p_{l}}{2} \wedge 1}
\\&~ \leq K_{p_1,p_2}T+ K_{p_1,p_2} \sum_{i,j:t_{i,n}^{(1)} \vee t_{j,n}^{(2)}\leq T} K |\mathcal{I}_{i,n}^{(1)}|^{\frac{p_1}{2}\wedge 1}|\mathcal{I}_{j,n}^{(2)} |^{\frac{p_2}{2}\wedge 1}\mathds{1}_{\{\mathcal{I}_{i,n}^{(1)} \cap \mathcal{I}_{j,n}^{(2)}\neq \emptyset \}}.
\end{align*}
This expression is bounded in probability by condition \eqref{cons_cond_1plusBeta} and hence the right hand side of \eqref{proof_general_1plusBeta_2} vanishes for $\rho \rightarrow 0$ due to $K_\rho \rightarrow 0$ as $\rho \rightarrow 0$ which yields  \eqref{proof_general_1plusBeta_1}.
\end{proof}

\begin{proof}[Proof of Theorem \ref{general_1plusBeta_end1}]
Comparing the proof of Theorem \ref{general_1plusBeta} it is sufficient to show that 
\begin{align}\label{proof_general_1plusBeta_end1_1}
Y(n)=\sum_{i,j:t_{i,n}^{(1)} \vee t_{j,n}^{(2)}\leq T} \big|\Delta_{i,n}^{(1)} X^{(1)}\big|^{p_1}\big|\Delta_{j,n}^{(2)} X^{(2)} \big|^{p_2} \mathds{1}_{\{\mathcal{I}_{i,n}^{(1)} \cap \mathcal{I}_{j,n}^{(2)}\neq \emptyset \}}
\end{align}
is bounded in probability as $n \rightarrow \infty$. To this end fix $\delta>0$, choose $\varepsilon>0$ such that $3\varepsilon<\delta$ and define $$\Omega(n,\varepsilon)=\{\sup_{i:t_{i,n}^{(l)} \leq T} \sum_{j \in \mathbb{N}} \mathds{1}_{\{\mathcal{I}_{i,n}^{(l)} \cap \mathcal{I}_{j,n}^{(3-l)}\neq \emptyset \}}> N_\varepsilon,~l=1,2\}$$ with $N_\varepsilon$ as in \eqref{cons_cond_1plusBeta_end1}. We then obtain using H\"older's inequality 
\begin{align*}
Y(n) \mathds{1}_{\Omega(n,\varepsilon)^C} &\leq \sum_{i,j:t_{i,n}^{(1)} \vee t_{j,n}^{(2)}\leq T} \big(\big(\Delta_{i,n}^{(1)} X^{(1)}\big)^2+\big(\Delta_{j,n}^{(2)} X^{(2)} \big)^2\big)\mathds{1}_{\{\mathcal{I}_{i,n}^{(1)} \cap \mathcal{I}_{j,n}^{(2)}\neq \emptyset \}}\mathds{1}_{\Omega(n,\varepsilon)^C}
\\&\leq N_\varepsilon \sum_{l=1,2} \sum_{i:t_{i,n}^{(l)} \leq T} \big(\Delta_{i,n}^{(l)} X^{(l)}\big)^2
\end{align*}
where the sums in the last line converge to the quadratic variations $[X^{(l)},X^{(l)}]_T$, $l=1,2$. If we further choose $K_\varepsilon>0$ such that $\mathbb{P}([X^{(l)},X^{(l)}]_T> K_\varepsilon(1+\xi))<\varepsilon$, $l=1,2$, for some $\xi>0$ we obtain
\begin{align*}
&\limsup_{n \rightarrow \infty} \mathbb{P}(Y(n)>2N_\varepsilon K_\varepsilon)
\\&~~~~~\leq \limsup_{n \rightarrow \infty} \mathbb{P}(\Omega(n,\varepsilon)) +\limsup_{n \rightarrow \infty} \sum_{l=1,2}\mathbb{P}\big(\sum_{i:t_{i,n}^{(l)} \leq T} \big(\Delta_{i,n}^{(l)} X^{(l)}\big)^2> K_\varepsilon\big)<\delta.
\end{align*}
As $\delta>0$ can be chosen arbitrarily this yields the boundedness in probability of $Y(n)$, $n \in \mathbb{N}$.
\end{proof}

\subsection{Proofs for Section \ref{sec:normFunc}}\label{subsection:lln_norm_proof}

The following lemma contains estimates for positively homogeneous functions which will be used in the upcoming proofs. The proof is elementary and therefore skipped here.

\begin{lemma}\label{lemma_prop_pos_hom_2d} 
Let $f:\mathbb{R}^{d_1}\times \mathbb{R}^{d_2} \rightarrow \mathbb{R}$, $d_1,d_2 \in \mathbb{N}_0$, be a continuous function which is positively homogeneous with degree $p_1 \geq 0$ in the first argument and with degree $p_2 \geq 0$ in the second argument. Then there exists a constant $K$ with
\begin{align}\label{ineq_pos_hom_abs_2d}
|f(x_1,x_2)| \leq K\|x_1\|^{p_1}\|x_2\|^{p_2}~~\forall x_1 \in \mathbb{R}^{d_1}, x_2 \in \mathbb{R}^{d_2}.
\end{align}
Further there exists a function $\theta:[0,\infty)\rightarrow [0,\infty)$ depending on $f$ and $p$ with $\theta(\varepsilon)\rightarrow 0$ for $\varepsilon \rightarrow 0$ and a constant $K_{p,\varepsilon}$ which may depend on $f$ such that
\begin{multline}\label{ineq_pos_hom_inc_2d}
|f(x_1+y_1,x_2+y_2)-f(x_1,x_2)|\leq \theta(\varepsilon)\|x_1\|^{p_1}\|x_2\|^{p_2}\\+K_{p_1,p_2,\varepsilon}(\|y_1\|^{p_1}(\|x_2\|^{p_2}+\|y_2\|^{p_2})+\|y_2\|^{p_2}(\|x_1\|^{p_1}+\|y_1\|^{p_1}))
\end{multline}
holds for all $x_1,y_1 \in \mathbb{R}^{d_1}$ and $x_2,y_2 \in \mathbb{R}^{d_2}$. In the case $p_2=0$ the inequality \eqref{ineq_pos_hom_inc_2d} can be replaced by
\begin{align}\label{ineq_pos_hom_inc_1d}
|f(x_1+y_1,x_2+y_2)-f(x_1,x_2)|\leq \theta(\varepsilon)\|x_1\|^{p_1}+K_{p_1,p_2,\varepsilon}\|y_1\|^{p_1}
\end{align}
for all $x_1,y_1 \in \mathbb{R}^{d_1}$ and $x_2,y_2 \in \mathbb{R}^{d_2}$. The analogous result holds if $p_1=0$.
\end{lemma}

To discuss the synchronous setting and the asynchronous setting simultaneously we consider a $(d_1+d_2)$-dimensional It\^o semimartingale $\widetilde{X}$ of the form \eqref{ItoSemimart}. $\widetilde{X}^{(1)}$ denotes the vector-valued process containing the first $d_1$ components of $\widetilde{X}$ and is observed at the observation times $t_{i,n}^{(1)}$, $\widetilde{X}^{(2)}$ contains the remaining $d_2$ components of $\widetilde{X}$ and is observed at the observation times $t_{i,n}^{(2)}$. Then the synchronous setting and the asynchronous setting discussed in Section \ref{sec:normFunc} correspond to $d_1=2,~d_2=0$, $p_1=p,~p_2=0$, and $d_1=d_2=1$ respectively. Here the notion of a zero-dimensional semimartingale remains ambiguous. However, we will only plug in increments of this zero-dimensional process into the argument of $f$ in which $f$ is positively homogeneous of degree $p_2=0$ and then $f$ is by Definition \ref{def_pos_hom} constant in this argument. Hence it is not necessary to specify the notion of a zero-dimensional semimartingale as we are going to use it only to indicate the case where the function $f$ solely depends on the first argument. Further we define
\begin{align}\label{V_star}
\overline{V}^*(p,f,\pi_n)_T=n^{p/2-1}\sum_{i,j \geq 0 :t_{i,n}^{(1)}\vee t_{j,n}^{(2)} \leq T} f\big(\Delta_{i,n}^{(1)}\widetilde{X}^{(1)},\Delta_{j,n}^{(2)}\widetilde{X}^{(2)}\big) \mathds{1}^{*,n}_{d_1,d_2}(i,j)
\end{align}
for all functions $f:\mathbb{R}^{d_1}\times \mathbb{R}^{d_2} \rightarrow \mathbb{R}$ where we set $\mathcal{I}_{0,n}^{(l)}= \emptyset$, $\Delta_{0,n}^{(l)}X^{(l)}=0$, $l=1,2$, and
\begin{align}\label{star_indicator}
\mathds{1}^{*,n}_{d_1,d_2}(i,j)= \begin{cases} \mathds{1}_{\{\mathcal{I}_{i,n}^{(1)} \cap \mathcal{I}_{j,n}^{(2)} \neq \emptyset\}}, &d_1>0, d_2 >0,
\\ \mathds{1}_{\{i>0,j=0\}}, & d_1>0,d_2=0,
\\ \mathds{1}_{\{i=0,j>0\}}, & d_1=0,d_2>0. \end{cases}
\end{align}
That means, we start the sum in \eqref{V_star} at zero and define the indicator $\mathds{1}^{*,n}_{d_1,d_2}(i,j)$ in such a way that, whenever $d_{3-l}=0$, for any $i \in \mathbb{N}$ with $t_{i,n}^{(l)} \leq T$ exactly one summand depending on $\Delta_{i,n}^{(l)}\widetilde{X}^{(l)}$ occurs in the sum. Hence $\overline{V}^*(p,f,\pi_n)_T$ corresponds to $\overline{V}(p,f,\pi_n)_T$ in the asynchronous setting for $d_1=d_2=1$ and $\overline{V}^*(p,f,\pi_n)_T$ corresponds to $\overline{V}(p,f,\pi_n)_T$ in the synchronous setting for $d_l=2$, $d_{3-l}=0$.

\begin{definition} We denote by 
\begin{align*}
\overline{C}^*(p,f,\pi_n)_T=n^{p/2-1}\sum_{i,j:t_{i,n}^{(1)}\vee t_{j,n}^{(2)} \leq T} f\big(\Delta_{i,n}^{(1)}\widetilde{C}^{(1)},\Delta_{j,n}^{(2)}\widetilde{C}^{(2)}\big) \mathds{1}^{*,n}_{d_1,d_2}(i,j)
\end{align*}
the functional $\overline{V}^*(p,f,\pi_n)_T$ evaluated at the continuous martingale part $\widetilde{C}$ instead of at $\widetilde{X}$ itself. $\overline{C}(p,f,\pi_n)_T$ is defined as $\overline{C}^*(p,f,\pi_n)_T$ above only with $\widetilde{C}$ replaced by $C$. \qed
\end{definition}

As for the indicator in \eqref{star_indicator} we also define a unifying notation for the functions $G^n_{p_1,p_2}(t)$ and $G^{(l),n}_p$ via
\begin{align}\label{star_G}
\widetilde{G}^{(d_1,d_2),n}_{p_1,p_2}(t)= \begin{cases} G^n_{p_1,p_2}(t), &d_1>0, d_2 >0,
\\ G^{(1),n}_{p_1}(t), & d_1>0,d_2=0,
\\ G^{(2),n}_{p_2}(t), & d_1=0,d_2>0. \end{cases}
\end{align}

The following proposition yields that by specifying $d_1,d_2$ appropriately as discussed above, it suffices to prove the convergences in Theorems \ref{theo_norm_2d_sync}, \ref{theo_norm_2d_uncor} and \ref{theo_norm_2d_integer} for $\overline{C}(p,f,\pi_n)_T$, $\overline{C}(p_1+p_2,f,\pi_n)_T$ instead of $\overline{V}(p,f,\pi_n)_T$, $\overline{V}(p_1+p_2,f,\pi_n)_T$. 

\begin{proposition}\label{lemma_norm_sync}
Let $f:\mathbb{R}^{d_1} \times\mathbb{R}^{d_2} \rightarrow \mathbb{R}$ be a function as in Lemma \ref{lemma_prop_pos_hom_2d}. Suppose that ${\widetilde{G}^{(d_1,d_2),n}_{p_1,p_2}(T)=O_{\mathbb{P}}(1)}$ and let either $p_1+p_2 \in [0,2)$ or $p_1+p_2 \geq 2$ and assume that $\widetilde{X}$ is continuous. Further we assume $d_l=0 \Rightarrow p_l=0$, $l=1,2$. Then it holds that
\begin{align}\label{lemma_norm_sync_equation}
\overline{V}^*(p_1+p_2,f,\pi_n)_T-\overline{C}^*(p_1+p_2,f,\pi_n)_T \overset{\mathbb{P}}\longrightarrow 0
\end{align}
as $n \rightarrow \infty$.
\end{proposition}

\begin{proof} In the following, we denote by $\widetilde{B}(q),\widetilde{C},\widetilde{M}(q),\widetilde{N}(q)$ a decomposition of $\widetilde{X}$ similar to \eqref{decomposition_ito}.  Using \eqref{ineq_pos_hom_inc_2d} we obtain
\begin{align}
&\big|\overline{V}^*(p_1+p_2,f,\pi_n)_T-\overline{C}^*(p_1+p_2,f,\pi_n)_T \big| \nonumber
\\&~ \leq n^{(p_1+p_2)/2-1}\sum_{i,j:t_{i,n}^{(1)} \vee t_{j,n}^{(2)}\leq T} \mathds{1}^{*,n}_{d_1,d_2}(i,j) \big[ \theta(\varepsilon)\|\Delta_{i,n}^{(1)}\widetilde{C}^{(1)}\|^{p_1}\|\Delta_{j,n}^{(2)}\widetilde{C}^{(2)}\|^{p_2} \label{lemma_norm_uncor_proof1}
\\&~~~+K_{p_1,p_2,\varepsilon}\|\Delta_{i,n}^{(1)}(\widetilde{X}-\widetilde{C})^{(1)}\|^{p_1}(\|\Delta_{j,n}^{(2)}\widetilde{C}^{(2)}\|^{p_2}+\|\Delta_{j,n}^{(2)}(\widetilde{X}-\widetilde{C})^{(2)}\|^{p_2}) \label{lemma_norm_uncor_proof2}
\\&~~~+K_{p_1,p_2,\varepsilon}\|\Delta_{j,n}^{(2)}(\widetilde{X}-\widetilde{C})^{(2)}\|^{p_2}(\|\Delta_{i,n}^{(1)}\widetilde{C}^{(1)}\|^{p_1}+\|\Delta_{i,n}^{(1)}(\widetilde{X}-\widetilde{C})^{(1)}\|^{p_1}) \big]. \label{lemma_norm_uncor_proof3}
\end{align}
For \eqref{lemma_norm_uncor_proof1} we get using the Cauchy-Schwarz inequality and inequality \eqref{elem_ineq_C}, as Lemma \ref{elem_ineq} holds for It\^o semimartingales of arbitrary dimension,
\begin{align*}
&\mathbb{E} \big[n^{(p_1+p_2)/2-1}\sum_{i,j:t_{i,n}^{(1)} \vee t_{j,n}^{(2)}\leq T}  \theta(\varepsilon)\|\Delta_{i,n}^{(1)}\widetilde{C}^{(1)}\|^{p_1}\|\Delta_{j,n}^{(2)}\widetilde{C}^{(2)}\|^{p_2} \mathds{1}^{*,n}_{d_1,d_2}(i,j) \big| \mathcal{S} \big] 
\\&~~\leq
\theta(\varepsilon) K \widetilde{G}^{(d_1,d_2),n}_{p_1,p_2}(T).
\end{align*}
Hence by Lemma 6.2 from \cite{MarVet18b} we obtain
\begin{gather*}
~\lim_{\varepsilon \rightarrow 0} \limsup_{n \rightarrow 0} \mathbb{P}\big(\theta(\varepsilon)n^{(p_1+p_2)/2-1}\sum_{i,j:t_{i,n}^{(1)} \vee t_{j,n}^{(2)}\leq T}\|\Delta_{i,n}^{(1)}\widetilde{C}^{(1)}\|^{p_1}\|\Delta_{j,n}^{(2)}\widetilde{C}^{(2)}\|^{p_2}  \mathds{1}^{*,n}_{d_1,d_2}(i,j)> \delta\big)=0~
\end{gather*}
for any $\delta>0$. To prove \eqref{lemma_norm_sync_equation} it now remains to show that \eqref{lemma_norm_uncor_proof2} vanishes as $n \rightarrow \infty$ for any $\varepsilon>0$ because then \eqref{lemma_norm_uncor_proof3} can be dealt with analogously by symmetry. We will separately discuss the cases $p_1>0$ and $p_1=0$.

\textit{Case 1.} We first consider the case where $p_1>0$. In the situation $p_1+p_2 \geq 2$ we have $\widetilde{X}=\widetilde{B}+\widetilde{C}$ with $\widetilde{B}_t=\int_0^t \tilde{b}_sds$ for some bounded process $\tilde{b}$. Hence we get that the $\mathcal{S}$-conditional expectation of \eqref{lemma_norm_uncor_proof2} is bounded by
\begin{align*}
 &K_{p_1,p_2,\varepsilon} n^{(p_1+p_2)/2-1}\sum_{i,j:t_{i,n}^{(1)} \vee t_{j,n}^{(2)}\leq T} |\mathcal{I}_{i,n}^{(1)}|^{p_1} (|\mathcal{I}_{j,n}^{(2)}|^{p_2}+|\mathcal{I}_{j,n}^{(2)}|^{p_2/2})\mathds{1}^{*,n}_{d_1,d_2}(i,j)
\\&~~ \leq K_{p_1,p_2,\varepsilon} (|\pi_n|_T)^{p_1/2} \widetilde{G}^{(d_1,d_2),n}_{p_1,p_2}(T)
\end{align*}
which vanishes as $n \rightarrow \infty$ for $p_1>0$. 

Next we consider \eqref{lemma_norm_uncor_proof2} in the case $p_1+p_2 < 2$. Using $$\|\Delta_{i,n}^{(1)}(\widetilde{X}-\widetilde{C})^{(1)}\|^{p_1} \leq K_{p_1} (\|\Delta_{i,n}^{(1)}(\widetilde{B}(q)+\widetilde{M}(q))^{(1)}\|^{p_1}+\|\Delta_{i,n}^{(1)}\widetilde{N}^{(1)}(q)\|^{p_1})$$ allows to treat the different components of $\Delta_{i,n}^{(1)}(\widetilde{X}-\widetilde{C})$ separately. Applying H\"older's inequality for $p'=2/(2-p_2)$, $q'=2/p_2$ and using the inequalities from Lemma \ref{elem_ineq} (note that $p_1<2-p_2$) yields
\begin{align}
&\mathbb{E}\big[ n^{(p_1+p_2)/2-1}\sum_{i,j:t_{i,n}^{(1)} \vee t_{j,n}^{(2)}\leq T}K_{p_1,p_2,\varepsilon}\|\Delta_{i,n}^{(1)}(\widetilde{B}(q)+\widetilde{M}(q))^{(1)}\|^{p_1} \nonumber
\\&~~~~~~\times (\|\Delta_{j,n}^{(2)}\widetilde{C}^{(2)}\|^{p_2}+\|\Delta_{j,n}^{(2)}(\widetilde{X}-\widetilde{C})^{(2)}\|^{p_2})\mathds{1}^{*,n}_{d_1,d_2}(i,j)
\big| \mathcal{S} \big]\nonumber
\\&~~ \leq K_{p_1,p_2,\varepsilon} n^{(p_1+p_2)/2-1}\sum_{i,j:t_{i,n}^{(1)} \vee t_{j,n}^{(2)}\leq T}\mathbb{E}[\|\Delta_{i,n}^{(1)}(\widetilde{B}(q)+\widetilde{M}(q))^{(1)}\|^{\frac{2p_1}{2-p_2}}|\mathcal{S}]^{\frac{2-p_2}{2}} \nonumber
\\&~~~~~~\times (\mathbb{E}[\|\Delta_{j,n}^{(2)}\widetilde{C}^{(2)}\|^2|\mathcal{S}]^{p_2/2}+\mathbb{E}[\|\Delta_{j,n}^{(2)}(\widetilde{X}-\widetilde{C})^{(2)}\|^2|\mathcal{S}]^{p_2/2}) \mathds{1}^{*,n}_{d_1,d_2}(i,j)\label{hoelder_trick}
\\& ~~\leq K_{p_1,p_2,\varepsilon} (K_q (|\pi_n|_T)^{p_1/2}+(e_q)^{p_1/2}) \widetilde{G}^{(d_1,d_2),n}_{p_1,p_2}(T)\nonumber
\end{align}
which vanishes as $n,q \rightarrow \infty$ for any $\varepsilon >0$ if $p_1>0$.

Finally consider
\begin{align}
&n^{(p_1+p_2)/2-1}K_{p_1,p_2,\varepsilon}\sum_{i,j:t_{i,n}^{(1)} \vee t_{j,n}^{(2)}\leq T}\|\Delta_{i,n}^{(1)}\widetilde{N}^{(1)}(q)\|^{p_1}  \nonumber
\\&~~~~~~~~~~~~~\times(\|\Delta_{j,n}^{(2)}\widetilde{C}^{(2)}\|^{p_2}+\|\Delta_{j,n}^{(2)}(\widetilde{X}-\widetilde{C})^{(2)}\|^{p_2})\mathds{1}^{*,n}_{d_1,d_2}(i,j) \nonumber
\\&~~ \leq n^{(p_1+p_2)/2-1}K_{p_1,p_2,\varepsilon}\sum_{i,j:t_{i,n}^{(1)} \vee t_{j,n}^{(2)}\leq T}\|\Delta_{i,n}^{(1)}\widetilde{N}^{(1)}(q)\|^{p_1}  \nonumber
\\&~~~~~~~~~~~~~\times(\|\Delta_{j,n}^{(2)}\widetilde{B}^{(2)}(q)\|^{p_2}+\|\Delta_{j,n}^{(2)}\widetilde{C}^{(2)}\|^{p_2}+\|\Delta_{j,n}^{(2)}\widetilde{M}^{(2)}(q)\|^{p_2})\mathds{1}^{*,n}_{d_1,d_2}(i,j)\label{lemma_norm_uncor_proof4}
\\&~~~~+n^{(p_1+p_2)/2-1}K_{p_1,p_2,\varepsilon}\sum_{i,j:t_{i,n}^{(1)} \vee t_{j,n}^{(2)}\leq T}\|\Delta_{i,n}^{(1)}\widetilde{N}^{(1)}(q)\|^{p_1}  
\|\Delta_{j,n}^{(2)}\widetilde{N}^{(2)}(q)\|^{p_2}\mathds{1}^{*,n}_{d_1,d_2}(i,j).\label{lemma_norm_uncor_proof4b}
\end{align}
\eqref{lemma_norm_uncor_proof4b} vanishes as $n \rightarrow \infty$ due to $p_1+p_2<2$ and because the finitely many jumps of $N(q)$ are asymptotically separated by the observation scheme. Further choose $\delta>0$ such that $p_1\vee 1 < 2-\delta$, $2(p_1+p_2)+(2-p_2)\delta<4$. Then the $\mathcal{S}$-conditional expectation of \eqref{lemma_norm_uncor_proof4} is by H\"older's inequality for $p'=(2-\delta)/p_1$ and $q'=p'/(p'-1)$  and inequalities \eqref{elem_ineq_B}--\eqref{elem_ineq_N} bounded by
\begin{align*}
&n^{(p_1+p_2)/2-1}K_{p_1,p_2,\varepsilon}\sum_{i,j:t_{i,n}^{(1)} \vee t_{j,n}^{(2)}\leq T} \big(\mathbb{E}[\|\Delta_{i,n}^{(1)}\widetilde{N}^{(1)}(q)\|^{2-\delta}|\mathcal{S}] \big)^{\frac{p_1}{2-\delta}}\mathds{1}^{*,n}_{d_1,d_2}(i,j)
\\&\times \big(\mathbb{E}[\|\Delta_{j,n}^{(2)}\widetilde{B}^{(2)}(q)\|^{p_2\frac{2-\delta}{2-\delta-p_1}}+\|\Delta_{j,n}^{(2)}\widetilde{C}^{(2)}\|^{p_2\frac{2-\delta}{2-\delta-p_1}}+\|\Delta_{j,n}^{(2)}\widetilde{M}^{(2)}(q)\|^{p_2\frac{2-\delta}{2-\delta-p_1}}|\mathcal{S}] \big)^{\frac{2-\delta-p_1}{2-\delta}}
\\&~\leq n^{(p_1+p_2)/2-1}K_{p_1,p_2,\varepsilon}\sum_{i,j:t_{i,n}^{(1)} \vee t_{j,n}^{(2)}\leq T} \big(|\mathcal{I}_{i,n}^{(1)}|^{\frac{p_1}{2-\delta}}+K_q |\mathcal{I}_{i,n}^{(1)}|^{p_1} \big)K_q |\mathcal{I}_{j,n}^{(2)}|^{\frac{p_2}{2}} \mathds{1}^{*,n}_{d_1,d_2}(i,j)
\\&~\leq K_{p_1,p_2,\varepsilon,q}(|\pi_n|_T)^{\frac{p_1}{2-\delta}-\frac{p_1}{2}} \widetilde{G}^{(d_1,d_2),n}_{p_1,p_2}(T)
\end{align*}
where we used
\begin{align*}
p_2q'=p_2\frac{p'}{p'-1}=p_2 \frac{2-\delta}{2-\delta-p_1}<2 \Leftrightarrow 2(p_1+p_2)+(2-p_2)\delta<4.
\end{align*}
Hence \eqref{lemma_norm_uncor_proof4} and then also \eqref{lemma_norm_uncor_proof2} vanish by Lemma 6.2 from \cite{MarVet18b} as $n\rightarrow \infty$ for any $\varepsilon$ and any $q>0$ if $p_1>0$.

\textit{Case 2.} Now we consider the case $p_1=0$. As \eqref{lemma_norm_sync_equation} is trivial for $p_1=p_2=0$ it remains to discuss $p_1=0$, $p_2>0$. In that case $${|\overline{V}^*(p_1+p_2,f,\pi_n)_T-\overline{C}^*(p_1+p_2,f,\pi_n)_T |}$$ is by \eqref{ineq_pos_hom_inc_1d} bounded by
\begin{gather*}
~n^{(p_1+p_2)/2-1}\sum_{i,j:t_{i,n}^{(1)} \vee t_{j,n}^{(2)}\leq T} \big[ \theta(\varepsilon)\|\Delta_{j,n}^{(2)}\widetilde{C}^{(2)}\|^{p_2}+K_{p_1,p_2,\varepsilon} \|\Delta_{i,n}^{(2)}(\widetilde{X}-\widetilde{C})^{(2)}\|^{p_2} \big]\mathds{1}^{*,n}_{d_1,d_2}(i,j).~
\end{gather*}
Here the first term in the sum corresponds to \eqref{lemma_norm_uncor_proof1} and the second term to \eqref{lemma_norm_uncor_proof3}. Hence the term \eqref{lemma_norm_uncor_proof2} does not have to be dealt with if $p_1=0$ because such a term simply does not occur in the upper bound in this situation.

By symmetry \eqref{lemma_norm_uncor_proof3} can be discussed in the same way as \eqref{lemma_norm_uncor_proof2} with the difference that for \eqref{lemma_norm_uncor_proof3} we have to discuss the cases $p_2>0$ and $p_2=0$ separately. Hence \eqref{lemma_norm_sync_equation} follows because we have shown that \eqref{lemma_norm_uncor_proof1}--\eqref{lemma_norm_uncor_proof3} vanish.
\end{proof}

Further we define discretizations of $\sigma$ and $C$ by 
\begin{align}\label{discret_r}
\begin{split}
&\sigma(r)_s=\sigma_{(k-1)T/2^r}, \quad s \in [(k-1)T/2^r,kT/2^r),
\\& C(r)_t=\int_0^t \sigma_s(r) ds.
\end{split}
\end{align}
Similarly we define discretizations of $\tilde{\sigma}$ and $\widetilde{C}$ for the $(d_1+d_2)$-dimensional process $\widetilde{X}$ and denote
\begin{align*}
\overline{C}^{*,r}(p,f,\pi_n)_T=n^{p/2-1}\sum_{i,j:t_{i,n}^{(1)}\vee t_{j,n}^{(2)} \leq T} f\big(\Delta_{i,n}^{(1)}\widetilde{C}^{(1)}(r),\Delta_{j,n}^{(2)}\widetilde{C}^{(2)}(r)\big) \mathds{1}^{*,n}_{d_1,d_2}(i,j).
\end{align*}

\begin{proposition}\label{lemma_norm_sync_r}
Let $f:\mathbb{R}^{d_1} \times\mathbb{R}^{d_2} \rightarrow \mathbb{R}$ be a function as in Lemma \ref{lemma_prop_pos_hom_2d} and assume that ${\widetilde{G}^{(d_1,d_2),n}_{p_1,p_2}(T)=O_{\mathbb{P}}(1)}$. Further we assume $d_l=0 \Rightarrow p_l=0$, $l=1,2$. Then
\begin{align*}
\lim_{r \rightarrow \infty} \limsup_{n \rightarrow \infty} \mathbb{P}(|\overline{C}^*(p_1+p_2,f,\pi_n)_T-\overline{C}^{*,r}(p_1+p_2,f,\pi_n)_T|>\delta )= 0
\end{align*}
holds for any $\delta >0$.
\end{proposition}

\begin{proof}
We obtain using inequality \eqref{ineq_pos_hom_inc_2d}
\begin{align}\label{theo_norm_2d_uncor_proof1}
&\mathbb{E}\big[|\overline{C}^*(p_1+p_2,f,\pi_n)_T-\overline{C}^{*,r}(p_1+p_2,f,\pi_n)_T|\big| \mathcal{S} \big] \nonumber
\\ &~~\leq n^{(p_1+p_2)/2-1} \sum_{i,j:t_{i,n}^{(1)} \vee t_{j,n}^{(2)} \leq T}
 \mathbb{E}\big[|f(\Delta_{i,n}^{(1)}\widetilde{C}^{(1)},\Delta_{j,n}^{(2)}\widetilde{C}^{(2)}) \nonumber
 \\&~~~~~~~~~~~~~~~~~~~~~~~~~~~~~~~~~~ -f(\Delta_{i,n}^{(1)}\widetilde{C}^{(1)}(r),\Delta_{j,n}^{(2)}\widetilde{C}^{(2)}(r)) |\big|\mathcal{S} \big] \mathds{1}^{*,n}_{d_1,d_2}(i,j)\nonumber
 \\ &~~\leq  n^{(p_1+p_2)/2-1} \sum_{i,j:t_{i,n}^{(1)} \vee t_{j,n}^{(2)} \leq T} \mathds{1}^{*,n}_{d_1,d_2}(i,j)\Big(
 \mathbb{E}\big[\theta(\varepsilon)\|\Delta_{i,n}^{(1)}\widetilde{C}^{(1)}\|^{p_1}|\Delta_{j,n}^{(2)}\widetilde{C}^{(2)}\|^{p_2}\big|\mathcal{S} \big] \nonumber
 \\&~~~~~~~~~~~~+K_{\varepsilon}\mathbb{E}\big[\|\Delta_{i,n}^{(1)}(\widetilde{C}-\widetilde{C}(r))^{(1)}\|^{p_1}(\|\Delta_{j,n}^{(2)}\widetilde{C}^{(2)}\|^{p_2}+\|\Delta_{j,n}^{(2)}\widetilde{C}^{(2)}(r)\|^{p_2}) \big|\mathcal{S} \big]\nonumber
 \\&~~~~~~~~~~~~+K_{\varepsilon}\mathbb{E}\big[\|\Delta_{j,n}^{(2)}(\widetilde{C}-\widetilde{C}(r))^{(2)}\|^{p_2}(\|\Delta_{i,n}^{(1)}\widetilde{C}^{(1)}\|^{p_1}+\|\Delta_{i,n}^{(1)}\widetilde{C}^{(1)}(r)\|^{p_1}) \big|\mathcal{S} \big] \Big)\nonumber
 \\&~~ \leq \theta(\varepsilon)\widetilde{G}^{(d_1,d_2),n}_{p_1,p_2}(T) \nonumber+K_{\varepsilon}n^{(p_1+p_2)/2-1}\sum_{l=1,2} \sum_{i,j:t_{i,n}^{(l)} \vee t_{j,n}^{(3-l)} \leq T}
 \\&~~~~~~ \mathbb{E}\big[\big(\int_{t_{i-1,n}^{(l)}}^{t_{i,n}^{(l)}} \|\tilde{\sigma}_s-\tilde{\sigma}_s(r)\|^2ds \big)^{p_l \vee \frac{1}{2}}\big|\mathcal{S} \big]^{p_l \wedge \frac{1}{2}} K |\mathcal{I}_{j,n}^{(3-l)}|^{\frac{p_{3-l}}{2}} \mathds{1}^{*,n}_{d_1,d_2}(i,j)
\end{align}
where we applied the Cauchy-Schwarz inequality, inequality \eqref{elem_ineq_B} and (2.1.34) from \cite{JacPro12} together with Jensen's inequality for $p_l <1/2$. Using the trivial inequality $a^{x} \leq \tilde{\eta}^{x}+ \tilde{\eta}^{{x}-1}a$ which holds for any $\tilde{\eta},a>0$ and ${x} \in [0,1]$ for $x=p_l \wedge (1/2) \leq 1$, 
\begin{align*}
\tilde{\eta}=(\eta |\mathcal{I}_{i,n}^{(l)}|)^{p_l/(2 p_l \wedge 1)}, \quad a=\mathbb{E} \big[\big(\int_{t_{i-1,n}^{(l)}}^{t_{i,n}^{(l)}} \|\tilde{\sigma}_s-\tilde{\sigma}_s(r)\|^2ds \big)^{p_l \vee \frac{1}{2}} \big| \mathcal{S} \big]
\end{align*}
yields that \eqref{theo_norm_2d_uncor_proof1} is bounded by (note $\frac{(p_l \wedge (1/2) -1)p_l}{2 p_l \wedge 1}+(p_l \vee \frac{1}{2})=\frac{p_l}{2}$)
\begin{align}\label{theo_norm_2d_uncor_proof2}
&(\theta(\varepsilon)+K_{\varepsilon} \eta^{p_l/2} )\widetilde{G}^{(d_1,d_2),n}_{p_1,p_2}(T)
+K_{\varepsilon,\eta}n^{(p_1+p_2)/2-1}\sum_{l=1,2} \sum_{i,j:t_{i,n}^{(l)} \vee t_{j,n}^{(3-l)} \leq T}|\mathcal{I}_{i,n}^{(l)}|^{\frac{(p_l \wedge (1/2) -1)p_l}{2 p_l \wedge 1}} \nonumber
 \\&~~~~~~ 
\times
 \mathbb{E}\big[\big(\int_{t_{i-1,n}^{(l)}}^{t_{i,n}^{(l)}} \|\tilde{\sigma}_s-\tilde{\sigma}_s(r)\|^2ds \big)^{p_l \vee \frac{1}{2}}\big|\mathcal{S} \big] |\mathcal{I}_{j,n}^{(3-l)}|^{\frac{p_{3-l}}{2}} \mathds{1}^{*,n}_{d_1,d_2}(i,j)\nonumber
 \\&~ \leq (\theta(\varepsilon)+K_{\varepsilon} \eta^{p_l/2}+K_{\varepsilon,\eta}(\delta^{2p_1 \vee 1}+\delta^{2p_2 \vee 1}) )\widetilde{G}^{(d_1,d_2),n}_{p_1,p_2}(T)+K_{\varepsilon,\eta}n^{(p_1+p_2)/2-1}\nonumber
 \\&~~~~~\times \sum_{l=1,2} \mathbb{E}\Big[\sum_{i,j:t_{i,n}^{(l)} \vee t_{j,n}^{(3-l)} \leq T}|\mathcal{I}_{i,n}^{(l)}|^{\frac{(p_l \wedge (1/2) -1)p_l}{2 p_l \wedge 1}} |\mathcal{I}_{j,n}^{(3-l)}|^{\frac{p_{3-l}}{2}} \mathds{1}^{*,n}_{d_1,d_2}(i,j) \nonumber
 \\&~~~~~~~~~~~ 
\times
\big(\int_{t_{i-1,n}^{(l)}}^{t_{i,n}^{(l)}} \|\tilde{\sigma}_s-\tilde{\sigma}_s(r)\|^2ds \big)^{p_l \vee \frac{1}{2}}\mathds{1}_{\{\sup_{s \in (t_{i-1,n}^{(l)},t_{i,n}^{(l)}]}\|\tilde{\sigma}_s-\tilde{\sigma}_s(r)\|>\delta \}}\Big|\mathcal{S} \Big]. 
\end{align}
Denote by $\Omega(N,n,r,\delta)$ the set where $\tilde{\sigma}$ has at most $N$ jump times $S_p$ in $[0,T]$ with $\|\Delta \tilde{\sigma}_{S_p}\|>\delta/2$, two different such jumps are further apart than $|\pi_n|_T$ and ${\|\tilde{\sigma}_t-\tilde{\sigma}_s\|\leq \delta}$ for all $s,t \in [0,T]$ with $s<t$, $|t-s|<2^{-r}+|\pi_n|_T$ and $\nexists p:S_p \in [t,s]$. Then $\mathbb{P}(\Omega(N,n,r,\delta))\rightarrow 1$ as $N,n,r \rightarrow \infty$ for any $\delta>0$ because $\sigma$ is c\`adl\`ag. Using the assumption that $\tilde{\sigma}$ is bounded we get that \eqref{theo_norm_2d_uncor_proof2} is less or equal than
\begin{align*}
&\big(\theta(\varepsilon)+K_{\varepsilon} \eta^{p_l/2}+K_{\varepsilon,\eta}[\delta^{2p_1 \vee 1}+\delta^{2p_2 \vee 1}+\mathbb{P}((\Omega(N,n,r,\delta))^c|\mathcal{S})] \big)\widetilde{G}^{(d_1,d_2),n}_{p_1,p_2}(T)
\\&~~~
+K_{\varepsilon,\eta} N \sup_{0\leq s < t,|t-s|\leq 2^{-r}+|\pi_n|_T} \big(\widetilde{G}^{(d_1,d_2),n}_{p_1,p_2}(t)-\widetilde{G}^{(d_1,d_2),n}_{p_1,p_2}(s)  \big)
\end{align*}
which yields 
\begin{multline*}
\lim_{\varepsilon \rightarrow 0} \limsup_{\eta \rightarrow 0}\limsup_{\delta \rightarrow 0}\limsup_{N \rightarrow \infty} \\ \limsup_{r \rightarrow \infty} \limsup_{n \rightarrow \infty} \mathbb{P}(|\overline{C}^{*}(p_1+p_2,f,\pi_n)_T-\overline{C}^{*,r}(p_1+p_2,f,\pi_n)_T|>\varepsilon')=0 
\end{multline*}
for all $\varepsilon'>0$.
\end{proof}

\begin{proof}[Proof of Theorem \ref{theo_norm_2d_sync}]
By Proposition \ref{lemma_norm_sync} it suffices to prove Theorem \ref{theo_norm_2d_sync} only in the case $X_t=C_t$. 

We consider the discretization \eqref{discret_r} and denote $c_s(r)=\sigma_s(r)\sigma_s(r)^*$. Setting
\begin{align*}
&R_n=\overline{C}(p,f,\pi_n)_T, 
\\&R=\int_0^T m_{c_s}(f) d{G_p}(s),
\\&R_n(r)=n^{p/2-1}\sum_{i:t_{i,n}\leq T} f\big(\Delta_{i,n}C(r)\big), 
\\&R(r)=\int_0^T m_{c_s(r)}(f) d{G_p}(s),
\end{align*}
we will prove
\begin{align}\label{theo_norm_2d_sync_proof1}
\lim_{r \rightarrow \infty} \limsup_{n \rightarrow \infty} \mathbb{P} \big( \big|R-R(r) \big|+\big| R(r)-R_n(r)\big|+\big|R_n(r)-R_n \big|>\delta  \big) = 0~~ \forall \delta>0.
\end{align}
\textit{Step 1.} As $c_s$ is c\`adl\`ag and $G_p(s)$ is continuous it holds
\begin{align*}
R=\int_0^T m_{c_{s-}}(f) d{G_p}(s).
\end{align*}
Deonte by $\Phi_{0,I_2}$ the distribution function of a two-dimensional standard normal random variable. Further consider a function $\psi:[0,\infty) \rightarrow[0,1]$ as in the proof of Theorem \ref{general_22} with $\mathds{1}_{[1,\infty)}(x)\leq\psi(x) \leq \mathds{1}_{[1/2,\infty)}(x)$ and define $\psi_A(x)=\psi(x/A)$ and $\psi'_A=1-\psi_A$ for $A>0$. Note that
\begin{align}\label{theo_norm_2d_sync_proof2}
|R-R(r)|& =\Big|  \int_0^T \int_{\mathbb{R}^2} (f(\sigma_{s-} x)-f(\sigma_s(r)x))\Phi_{0,I_2}(dx) d{G_p}(s)\Big| \nonumber
\\ & \leq   \int_0^T \int_{\mathbb{R}^2}|(f  \psi_A)(\sigma_{s-} x)-(f  \psi_A)(\sigma_s(r)x)|\Phi_{0,I_2}(dx)d{G_p}(s)   \nonumber
\\&~~~+\int_0^T \int_{\mathbb{R}^2}|(f  \psi_A')(\sigma_{s-} x)-(f  \psi_A')(\sigma_s(r)x)|\Phi_{0,I_2}(dx)d{G_p}(s).
\end{align}
By \eqref{ineq_pos_hom_abs_2d} for $p_1=p$, $p_2=0$ we obtain $|(f \psi_A')(x)| \leq K A^p $ and hence $(f \psi_A')$ is bounded. Then the fact that $(f \psi_A')$ is continuous together with the pointwise convergence $\sigma_s(r) \rightarrow \sigma_{s-}$ yields by dominated convergence that the second summand in \eqref{theo_norm_2d_sync_proof2} vanishes for any $A>0$ as $r \rightarrow \infty$.

The first summand in \eqref{theo_norm_2d_sync_proof2} is bounded by
\begin{align}\label{theo_norm_2d_sync_proof3}
K\int_0^T \int_{\mathbb{R}^2}\big(\|\sigma_{s-}x\|^p \mathds{1}_{\{\|\sigma_{s-} x\| \geq A/2\}}+\|\sigma_{s}(r)x\|^p \mathds{1}_{\{\|\sigma_{s}(r) x\| \geq A/2\}}\big)\Phi_{0,I_2}(dx)d{G_p}(s)
\end{align}
where the inner integral is increasing in $\sigma_s$. As we assume that $\sigma$ is bounded on $[0,T]$ this yields that there exists a constant $$K'=\esssup_{s \in [0,T],\omega \in \Omega} (|\sigma_s^{(1)}(\omega)|+|\sigma_s^{(2)}(\omega)|)$$ such that \eqref{theo_norm_2d_sync_proof3} is bounded by
\begin{align*}
K\int_{\mathbb{R}^2}\big(\|K'x\|^p \mathds{1}_{\{\|K' x\| \geq A/2\}}+\|K'x\|^p \mathds{1}_{\{\|K' x\| \geq A/2\}}\big)\Phi_{0,I_2}(dx) \int_0^T d{G_p}(s)
\end{align*}
which vanishes as $A \rightarrow \infty$. Hence we have shown $|R-R(r)|\rightarrow 0$ almost surely as $r \rightarrow \infty$. 

\textit{Step 2.} In order to prove $|R(r)-R_n(r)| \overset{\mathbb{P}}{\longrightarrow} 0$ as $n \rightarrow \infty$ we apply Lemma 2.2.12 in \cite{JacPro12} with 
\begin{align*}
\xi_k^n=n^{p/2-1}\sum_{i \in L(n,k,T) } f(\Delta_{i,n}C(r) )
, 
\end{align*}
$L(n,k,T)=\{i:t_{i-1,n}   \in  [(k-1)T/2^{r_n},kT/2^{r_n})\}$, $k=1,2,\ldots,2^{r_n}$, and $\mathcal{G}_k^n=\sigma\big(\mathcal{F}_{(k-1)T/2^{r_n}} \cup \mathcal{S}\big)$. Here, $r_n$ is a sequence of real numbers with $r_n \geq r$, $r_n \rightarrow \infty$ and
\begin{align}\label{r_n}
\begin{split}
&2^{r_n} \sup_{s \in [0,T]} |G_p(s)-G_p^{(l),n}(s)|=o_{\mathbb{P}}(1), 
\\&2^{r_n}\sup_{s,t \in [0,T],|t-s|\leq |\pi_n|_T} |G_p^{(l),n}(t)-G_p^{(l),n}(s)|=o_{\mathbb{P}}(1). 
\end{split}
\end{align}
Such a sequence exists, because $G_p^{(l),n}$ and hence $G_p$ are nondecreasing functions such that pointwise convergence implies uniform convergence on $[0,T]$ to the continuous function $G_p$. We then get
\begin{align}\label{O_p_ntimespi_n}
&\mathbb{E}\big[\xi_k^n \big| \mathcal{G}_{k-1}^n  \big]=
n^{p/2-1}\sum_{i \in L(n,k,T) } \mathbb{E}\big[f(\Delta_{i,n} C(r))\big| \mathcal{G}_{k-1}^n  \big] \nonumber
\\&=n^{p/2-1} \sum_{i \in L(n,k,T) } |\mathcal{I}_{i,n}|^{p/2} \mathbb{E}\big[f(|\mathcal{I}_{i,n}|^{-1/2}\Delta_{i,n} C(r) )\big| \mathcal{G}_{k-1}^n  \big]\nonumber
\\&=n^{p/2-1} m_{c_{(k-1)T/2^{r_n}}(r)}(f) \sum_{i :\mathcal{I}_{i,n} \subset ((k-1)T/2^{r_n},kT/2^{r_n}] } |\mathcal{I}_{i,n}|^{p/2} \nonumber
\\&~~~~+K\sup_{s,t \in [0,T],|t-s|\leq |\pi_n|_T} |G_p^{(l),n}(t)-G_p^{(l),n}(s)|\nonumber
\\&=m_{c_{(k-1)T/2^{r_n}}(r)}(f)\big( G_p^{(l),n}(kT/2^{r_n})-G_p^{(l),n}((k-1)T/2^{r_n}) \big) \nonumber
\\&~~~~+K\sup_{s,t \in [0,T],|t-s|\leq |\pi_n|_T} |G_p^{(l),n}(t)-G_p^{(l),n}(s)|
\end{align}
because $|\mathcal{I}_{i,n}|^{-1/2}\Delta_{i,n} C(r) $ is conditional on $\mathcal{G}_{k-1}^n$ centered normal distributed with covariance matrix $c_{(k-1)T/2^{r_n}}(r)$. The term $$K\sup_{s,t \in [0,T],|t-s|\leq |\pi_n|_T} |G_p^{(l),n}(t)-G_p^{(l),n}(s)|$$ is due to the summand with ${kT/2^{r_n}\in \mathcal{I}_{i,n}}$ which has to be treated separately as in the corresponding interval the process $\sigma(r)$ might jump. Further as in Step 1 the boundedness of $\sigma_s$ together with $|f(x)|\leq K \|x \|^p$ yields that $m_{c_s}(f)$ is also bounded which together with the previous computations yields
\begin{multline*}
\big|R(r)-\sum_{k=1}^{2^{r_n}} \mathbb{E}[\xi_k^n | \mathcal{G}_{k-1}^n  ]\big| \\
\leq K2^{r_n}\sup_{s \in [0,T]} |G_p(s)-G_p^n(s)|+K2^{r_n}\sup_{s,t \in [0,T],|t-s|\leq |\pi_n|_T} |G_p^{(l),n}(t)-G_p^{(l),n}(s)|
\end{multline*}
where the right hand side is $o_{\mathbb{P}}(1)$ by \eqref{r_n}. Hence the sum over the $\mathbb{E}[\xi_k^n \big| \mathcal{G}_{k-1}^n  ]$ converges in probability to $R(r)$.

Using the Cauchy-Schwarz inequality, inequality \eqref{elem_ineq_C}, the definition of $G_p^n$ and telescoping sums we also get
\begin{align*}
&\sum_{k=1}^{2^{r_n}}\mathbb{E}[|\xi_k^n|^2 | \mathcal{G}_{k-1}^n  ] \leq  \sum_{k=1}^{2^{r_n}} \mathbb{E}\big[\big( n^{p/2-1}\sum_{i \in L(n,k,T) }
K \|\Delta_{i,n} C(r) \|^p  \big)^2\big| \mathcal{G}_{k-1}^n  \big]
\\&~~=K \sum_{k=1}^{2^{r_n}}n^{p-2}\sum_{i \in L(n,k,T) }\sum_{j \in L(n,k,T) }\mathbb{E}\big[ \|\Delta_{i,n} C(r) \|^p \|\Delta_{j,n} C(r) \|^p \big| \mathcal{G}_{k-1}^n  \big]
\\&~~\leq K \sum_{k=1}^{2^{r_n}}n^{p-2}\sum_{i \in L(n,k,T) }\sum_{j \in L(n,k,T) }\big(\prod_{m=i,j} \mathbb{E}\big[ \|\Delta_{m,n} C(r) \|^{2p}\big| \mathcal{G}_{k-1}^n  \big]\big)^{1/2}
\\&~~\leq K \sum_{k=1}^{2^{r_n}}n^{p-2}\sum_{i \in L(n,k,T) }\sum_{j \in L(n,k,T) }|\mathcal{I}_{i,n}|^{p/2}|\mathcal{I}_{j,n}|^{p/2}
\\&~~\leq K \sum_{k=1}^{2^{r_n}} \big(n^{p/2-1}\sum_{i \in L(n,k,T) }|\mathcal{I}_{i,n}|^{p/2}\big)^2
\\ &~~\leq K G_p^n(T) \sup_{u,s \in [0,T], |u-s|\leq T 2^{-r_n}+|\pi_n|_T}\big|G_p^{(l),n}(u)-G_p^{(l),n}(s)\big|
\end{align*}
where the right hand side converges to zero in probability, since $G_p^n$ converges uniformly to a continuous function $G_p$. Hence we have shown 
$$
\sum_{k=1}^{2^{r_n}} \mathbb{E}[\xi_k^n | \mathcal{G}_{k-1}^n  ] \overset{\mathbb{P}}{\longrightarrow} R(r),~~~\sum_{k=1}^{2^{r_n}}\mathbb{E}[|\xi_k^n|^2 | \mathcal{G}_{k-1}^n  ] \overset{\mathbb{P}}{\longrightarrow} 0.
$$
Further the $\xi_k^n$ are $\mathcal{G}_k^n$-measurable and hence $\xi_k^n-\mathbb{E}[\xi_k^n|G_{k-1}^n]$ are martingale differences. Lemma 2.2.12 from \cite{JacPro12} then yields
$$
R_n(r)=\sum_{k=1}^{2^{r_n}} \xi_k^n  \overset{\mathbb{P}}{\longrightarrow} R(r)
$$
for any $r \in \mathbb{N}$.

\textit{Step 3.}
Finally we obtain
\begin{align*}
\lim_{r \rightarrow \infty} \limsup_{n \rightarrow \infty} \mathbb{P}(|R_n-R_n(r)|>\delta)=0
\end{align*}
for any $\delta >0$ from Proposition \ref{lemma_norm_sync_r} with $d_1=2$, $d_2=0$, $p_1=p$ and $p_2=0$.
\end{proof}

\begin{proof}[Proof of Corollary \ref{theo_norm_1d}]
This is Theorem \ref{theo_norm_2d_sync} with the function $$f(x_1,x_2)=\text{sgn}(g(x_1))|g(x_1)g(x_2)|^{1/2}$$ applied to the process $\widetilde{X}_t=(X^{(l)}_t,X^{(l)}_t)^*$. Note that any positively homogeneous function $g$ in dimension $1$ is continuous because it holds $$g(x)=|x|g(1)\mathds{1}_{\{x>0\}}+|x|g(-1)\mathds{1}_{\{x<0\}}.$$ 
\end{proof}

\begin{proof}[Proof of Theorem \ref{theo_norm_2d_uncor}] By Proposition \ref{lemma_norm_sync} it suffices to prove Theorem \ref{theo_norm_2d_uncor} if ${X_t=C_t}$. We will proceed as in the proof of Theorem \ref{theo_norm_2d_sync} and define
\begin{align*}
&R_n=\overline{C}(p_1+p_2,f,\pi_n)_T, 
\\&R=m_{I_2}(f)\int_0^T (\sigma_s^{(1)})^{p_1} (\sigma_s^{(2)})^{p_2} d{G_{(p_1,p_2)}}(s),
\\&R_n(r)=n^{(p_1+p_2)/2-1}\sum_{i,j:t_{i,n}^{(1)}\vee t_{j,n}^{(2)} \leq T} f\big(\Delta_{i,n}^{(1)}C^{(1)}(r),\Delta_{j,n}^{(2)}C^{(2)}(r)\big) \mathds{1}_{\{\mathcal{I}_{j,n}^{(2)} \cap \mathcal{I}_{j,n}^{(2)} \neq \emptyset\}}, 
\\&R(r)=m_{I_2}(f)\int_0^T (\sigma_s^{(1)}(r))^{p_1} (\sigma_s^{(2)}(r))^{p_2} d{G_{(p_1,p_2)}}(s).
\end{align*}

\textit{Step 1.} As $\sigma$ is c\`adl\`ag we obtain $\sigma_s(r) \rightarrow \sigma_{s-}$ pointwise for $r \rightarrow \infty$. As $\sigma$ is further bounded we derive by dominated convergence
\begin{align*}
R(r) \rightarrow m_{I_2}(f)\int_0^T (\sigma_{s-}^{(1)})^{p_1} (\sigma_{s-}^{(2)})^{p_2} d{G_{(p_1,p_2)}}(s)
\end{align*}
where the right hand side equals $R$ because $G_{p_1,p_2}$ is continuous.

\textit{Step 2.} Comparing Step 2 in the proof of Theorem \ref{theo_norm_2d_sync} we define
\begin{align*}
\xi_k^n=n^{(p_1+p_2)/2-1}\sum_{(i,j) \in L(n,k,T) } f(\Delta_{i,n}^{(1)} C^{(1)}(r) ,\Delta_{j,n}^{(2)} C^{(2)}(r) )
\mathds{1}_{\{\mathcal{I}_{i,n}^{(1)} \cap\mathcal{I}_{j,n}^{(2)} \neq \emptyset\}}
\end{align*}
where the sequence $r_n$ fulfills similar properties as in the proof of Theorem \ref{theo_norm_2d_uncor} and ${L(n,k,T)=\{(i,j):t_{i-1,n}^{(1)} \wedge t_{j-1,n}^{(2)}  \in  [(k-1)T/2^{r_n},kT/2^{r_n})\}}$, $k=1,\ldots, 2^{r_n}$. Then we obtain the identity
\begin{align*}
\mathbb{E}[\xi_k^n|\mathcal{G}_{k-1}^n]
&=n^{(p_1+p_2)/2-1}\sum_{(i,j) \in L(n,k,T) } (\sigma_{(k-1)T/2^{r_n}}^{(1)}(r))^{p_1} (\sigma_{(k-1)T/2^{r_n}}^{(2)}(r))^{p_2}
\\&~~~~~~~~~~~~~~~~~~~~~~~~~~~~~~\times|\mathcal{I}_{i,n}^{(1)}|^{p_1/2} |\mathcal{I}_{j,n}^{(2)}|^{p_2/2} m_{I_2}(f)
\mathds{1}_{\{\mathcal{I}_{i,n}^{(1)} \cap\mathcal{I}_{j,n}^{(2)} \neq \emptyset\}}
\\&~~~~+ K\sup_{s,t \in [0,T],|t-s|\leq 3|\pi_n|_T} |G_{p_1,p_2}^{n}(t)-G_{p_1,p_2}^{n}(s)|
\end{align*}
were we used that $f$ is positively homogeneous in both arguments and that
\begin{align*}
\big(\Delta_{i,n}^{(1)} C^{(1)}(r)/(\sigma_{(k-1)T/2^{r_n}}^{(1)}(r)|\mathcal{I}_{i,n}^{(1)}|^{1/2}),\Delta_{j,n}^{(2)} C^{(2)}(r)/(\sigma_{(k-1)T/2^{r_n}}^{(2)}(r)|\mathcal{I}_{j,n}^{(2)}|^{1/2})\big)
\end{align*}
is for $\mathcal{I}^{(1)}_{i,n} \cup \mathcal{I}_{j,n}^{(2)}\subset [(k-1)T/2^{r_n},kT/2^{r_n})$ conditionally on $\mathcal{G}_{k-1}^n$ standard normally distributed due to $\rho \equiv 0$. The term $$K\sup_{s,t \in [0,T],|t-s|\leq 3|\pi_n|_T} |G_{p_1,p_2}^{n}(t)-G_{p_1,p_2}^{n}(s)|$$ originates similarly as before from summands with $kT/2^{r_n}\in \mathcal{I}^{(1)}_{i,n} \cup \mathcal{I}_{j,n}^{(2)}$ which have to be treated separately as in the corresponding intervals the process $\sigma(r)$ might jump.

This yields
\begin{multline*}
\big|R(r)-\sum_{k=1}^{2^{r_n}} \mathbb{E}[\xi_k^n | \mathcal{G}_{k-1}^n  ]\big|
 \leq K2^{r_n}\sup_{s \in [0,T]} |G_{p_1,p_2}(s)-G_{p_1,p_2}^n(s)|\\+K2^{r_n}\sup_{s,t \in [0,T],|t-s|\leq 3|\pi_n|_T} |G_{p_1,p_2}^{n}(t)-G_{p_1,p_2}^{n}(s)|
\end{multline*}
and we get as in Step 2 in the proof of Theorem \ref{theo_norm_2d_sync}
\begin{align*}
&\sum_{k=1}^{2^{r_n}}\mathbb{E}[|\xi_k^n|^2 | \mathcal{G}_{k-1}^n  ] 
\\&~\leq \sum_{k=1}^{2^{r_n}} \mathbb{E}\big[  \big(n^{(p_1+p_2)/2-1}
\sum_{(i,j) \in L(n,k,T)} K |\Delta_{i,n}^{(1)} C^{(1)}(r)|^{p_1}|\Delta_{j,n}^{(2)} C^{(2)}(r)|^{p_2} \mathds{1}_{\{\mathcal{I}_{i,n}^{(1)} \cap \mathcal{I}_{j,n}^{(2)} \neq \emptyset\}}\big)^2
\big| \mathcal{G}_{k-1}^n  \big] 
\\&~\leq 
 K \sum_{k=1}^{2^{r_n}} n^{p_1+p_2-2}\sum_{(i,j),(i',j') \in L(n,k,T)}\big(\mathbb{E}[|\Delta_{i,n}^{(1)}C^{(1)}(r)|^{4p_1}|\mathcal{G}_{k-1}^n ]\mathbb{E}[|\Delta_{j,n}^{(2)}C^{(2)}(r)|^{4p_1}|\mathcal{G}_{k-1}^n ]\big)^{1/4}
\\&~~~~\times \big(\mathbb{E}[|\Delta_{i',n}^{(1)}C^{(1)}(r)|^{4p_1}|\mathcal{G}_{k-1}^n ]\mathbb{E}[|\Delta_{j',n}^{(2)}C^{(2)}(r)|^{4p_1}|\mathcal{G}_{k-1}^n ]\big)^{1/4}
 \mathds{1}_{\{\mathcal{I}_{i,n}^{(1)} \cap \mathcal{I}_{j,n}^{(2)} \neq \emptyset\}}\mathds{1}_{\{\mathcal{I}_{i',n}^{(1)} \cap \mathcal{I}_{j',n}^{(2)} \neq \emptyset\}}
 \\&~\leq 
 K \sum_{k=1}^{2^{r_n}} n^{p_1+p_2-2}\sum_{(i,j),(i',j') \in L(n,k,T)}
|\mathcal{I}_{i,n}^{(1)}|^{p_1/2} |\mathcal{I}_{j,n}^{(2)}|^{p_2/2}
 |\mathcal{I}_{i',n}^{(1)}|^{p_1/2} |\mathcal{I}_{j',n}^{(2)}|^{p_2/2}
 \\&~~~~~~~~~~~~~~~~~~~~~~~~~~~~~~~~~~~~~~~~~~~~~~~~~~~~~~~~~~~~\times \mathds{1}_{\{\mathcal{I}_{i,n}^{(1)} \cap \mathcal{I}_{j,n}^{(2)} \neq \emptyset\}}\mathds{1}_{\{\mathcal{I}_{i',n}^{(1)} \cap \mathcal{I}_{j',n}^{(2)} \neq \emptyset\}}
 \\&~=K \sum_{k=1}^{2^{r_n}} \big(n^{(p_1+p_2)/2-1}\sum_{(i,j) \in L(n,k,T)}
|\mathcal{I}_{i,n}^{(1)}|^{p_1/2} |\mathcal{I}_{j,n}^{(2)}|^{p_2/2}
 \mathds{1}_{\{\mathcal{I}_{i,n}^{(1)} \cap \mathcal{I}_{j,n}^{(2)} \neq \emptyset\}}\big)^2
\\&~ \leq K G_{p_1,p_2}^n(T) \sup_{u,s \in [0,T], |u-s|\leq T 2^{-r_n}+|\pi_n|_T}\big|G_{p_1,p_2}^n(u)-G_{p_1,p_2}^n(s)\big|
\end{align*}
which then yields $R_n(r) \overset{\mathbb{P}}{\longrightarrow} R(r)$ as $n \rightarrow \infty$ by Lemma 2.2.12 in \cite{JacPro12}.

\textit{Step 3.} Finally we obtain
\begin{align*}
\lim_{r \rightarrow \infty} \limsup_{n \rightarrow \infty} \mathbb{P}(|R_n-R_n(r)|>\delta)=0
\end{align*}
for any $\delta >0$ from Proposition \ref{lemma_norm_sync_r} with $d_1=d_2=1$.
\end{proof}

\begin{proof}[Proof of Theorem \ref{theo_norm_2d_integer}]
Proposition \ref{lemma_norm_sync} yields that it suffices to prove Theorem \ref{theo_norm_2d_integer} in the case $X_t=C_t$. Further, following the proof of Theorem \ref{theo_norm_2d_uncor} we observe that Step 1 and Step 3 do not make use of the assumption $\rho=0$.  Hence the arguments therein also apply here. The only difference occurs if we want to adapt Step 2 in the proof of Theorem \ref{theo_norm_2d_uncor} for the proof of Theorem \ref{theo_norm_2d_integer}. In fact, if we look at 
\begin{align*}
\xi_k^n=n^{(p_1+p_2)/2-1}\sum_{(i,j) \in L(n,k,T)} (\Delta_{i,n}^{(1)}C^{(1)}(r))^{p_1}(\Delta_{j,n}^{(2)}C^{(2)}(r))^{p_2} \mathds{1}_{\{\mathcal{I}_{i,n}^{(1)}\cap \mathcal{I}_{j,n}^{(2)} \neq \emptyset\}}
\end{align*}
and denote by $\Phi_{0,I_4}$ the distribution function of a four-dimensional standard normal random variable we get 
\begin{align*}
&\mathbb{E}[\xi_k^n|\mathcal{G}_{k-1}^n]=n^{(p_1+p_2)/2-1} \sum_{(i,j) \in L(n,k,T)}  \mathbb{E}[(\Delta_{i,n}^{(1)}C^{(1)}(r))^{p_1}(\Delta_{j,n}^{(2)}C^{(2)}(r))^{p_2}|\mathcal{G}_k^n] \mathds{1}_{\{\mathcal{I}_{i,n}^{(1)}\cap \mathcal{I}_{j,n}^{(2)} \neq \emptyset\}}
\\&~~=\prod_{l=1,2}(\sigma_{(k-1)T/2^{r_n}}^{(l)})^{p_l}\sum_{(i,j) \in L(n,k,T)} 
\int_{\mathbb{R}^4}\Big(|\mathcal{I}_{i,n}^{(1)}\setminus \mathcal{I}_{j,n}^{(2)}|^{1/2}x_1+|\mathcal{I}_{i,n}^{(1)}\cap \mathcal{I}_{j,n}^{(2)}|^{1/2}x_2 \Big)^{p_1}
\\&~~~~~~~~\times \Big(\rho_{(k-1)T/2^{r_n}}|\mathcal{I}_{i,n}^{(1)}\cap \mathcal{I}_{j,n}^{(2)}|^{1/2} x_2+(1-(\rho_{(k-1)T/2^{r_n}})^2)^{1/2}|\mathcal{I}_{i,n}^{(1)}\cap \mathcal{I}_{j,n}^{(2)}|^{1/2}x_3
\\& ~~~~~~~~~~~~+|\mathcal{I}_{j,n}^{(2)} \setminus \mathcal{I}_{i,n}^{(1)}|^{1/2} x_4\Big)^{p_2}\Phi_{0,I_4}(dx) \mathds{1}_{\{\mathcal{I}_{i,n}^{(1)}\cap \mathcal{I}_{j,n}^{(2)} \neq \emptyset\}}
\\&~~~~~~~~~~+  K\sup_{s,t \in [0,T],|t-s|\leq 3|\pi_n|_T} |G_{p_1,p_2}^{n}(t)-G_{p_1,p_2}^{n}(s)|
\\&~~=(\sigma_{(k-1)T/2^{r_n}}^{(1)})^{p_1}(\sigma_{(k-1)T/2^{r_n}}^{(2)})^{p_2}\sum_{(i,j) \in L(n,k,T)} \sum_{k=0}^{p_1} \sum_{m,l=0}^{p_2} \binom{p_1}{k}\binom{p_2}{l,m}
\\&~~~~~~ \times \int_{\mathbb{R}^4}|\mathcal{I}_{i,n}^{(1)}\setminus \mathcal{I}_{j,n}^{(2)}|^{k/2} (x_1)^k (\rho_{(k-1)T/2^{r_n}})^{p_2-(l+m)}|\mathcal{I}_{i,n}^{(1)}\cap \mathcal{I}_{j,n}^{(2)}|^{(p_1+p_2-(k+l+m))/2}
\\&~~~~~~~~~~~ \times (x_2)^{p_1+p_2-(k+l+m)}(1-(\rho_{(k-1)T/2^{r_n}})^2)^{l/2} |\mathcal{I}_{i,n}^{(1)}\cap \mathcal{I}_{j,n}^{(2)}|^{l/2} (x_3)^{l}
\\&~~~~~~~~~~~ \times  |\mathcal{I}_{j,n}^{(2)} \setminus \mathcal{I}_{i,n}^{(1)}|^{m/2} (x_4)^{m} \Phi_{0,I_4}(dx) \mathds{1}_{\{\mathcal{I}_{i,n}^{(1)}\cap \mathcal{I}_{j,n}^{(2)} \neq \emptyset\}}
\\&~~~~~~~~~~+  K\sup_{s,t \in [0,T],|t-s|\leq 3|\pi_n|_T} |G_{p_1,p_2}^{n}(t)-G_{p_1,p_2}^{n}(s)|
\\&~~=(\sigma_{(k-1)T/2^{r_n}}^{(1)})^{p_1}(\sigma_{(k-1)T/2^{r_n}}^{(2)})^{p_2}\Big[\sum_{(k,l,m) \in L(p_1,p_2)} \binom{p_1}{k}\binom{p_2}{l,m}
 m_1(x^k) m_1(x^l)
 \\&~~~~~~~~~~~ \times m_1(x^m)m_1(x^{p_1+p_2-(k+l+m)}) (\rho_{(k-1)T/2^{r_n}})^{p_2-(l+m)}(1-\rho_{(k-1)T/2^{r_n}}^2)^{l/2} 
 \\&~~~~~~~~~~~ \times\Big(H^n_{k,m,p_1+p_2}(kT/2^{r_n})
 -H^n_{k,m,p_1+p_2}((k-1)T/2^{r_n}   \Big)\Big] 
\\&~~~~~~~~~~+  K\sup_{s,t \in [0,T],|t-s|\leq 3|\pi_n|_T} |G_{p_1,p_2}^{n}(t)-G_{p_1,p_2}^{n}(s)|
\end{align*}
where we used the multinomial theorem, $\mathbb{E}[X^k]=0$ for $X \sim \mathcal{N}(0,1)$, $k$ odd, and
\begin{small}
\begin{align*}
&\begin{pmatrix}
\Delta_{i,n}^{(1)}C^{(1)}(r)
\\ \Delta_{j,n}^{(2)}C^{(2)}(r)
\end{pmatrix}
\\&~\overset{\mathcal{L}_{\mathcal{G}_{k-1}^n}}{=}  
\begin{pmatrix}
\sigma_{(k-1)T/2^{r_n}}^{(1)}|\mathcal{I}_{i,n}^{(1)}\setminus \mathcal{I}_{j,n}^{(2)}|^{1/2} & 0
\\ \sigma_{(k-1)T/2^{r_n}}^{(1)}|\mathcal{I}_{i,n}^{(1)}\cap \mathcal{I}_{j,n}^{(2)}|^{1/2} & \rho_{(k-1)T/2^{r_n}} \sigma_{(k-1)T/2^{r_n}}^{(2)}|\mathcal{I}_{i,n}^{(1)}\cap \mathcal{I}_{j,n}^{(2)}|^{1/2}
\\ 0 &
 (1-(\rho_{(k-1)T/2^{r_n}})^2)^{1/2}\sigma_{(k-1)T/2^{r_n}}^{(2)}|\mathcal{I}_{i,n}^{(1)}\cap \mathcal{I}_{j,n}^{(2)}|^{1/2}
\\ 0 &  (\sigma_{(k-1)T/2^{r_n}}^{(2)})|\mathcal{I}_{j,n}^{(2)} \setminus \mathcal{I}_{i,n}^{(1)}|^{1/2}
\end{pmatrix}^*
U.
\end{align*}
\end{small}for $\mathcal{I}^{(1)}_{i,n} \cup \mathcal{I}_{j,n}^{(2)}\subset [(k-1)T/2^{r_n},kT/2^{r_n})$ where $U=(U_1,U_2,U_3,U_4)^*\sim \mathcal{N}(0,I_4)$-distributed and independent of $\mathcal{F}$ where $\mathcal{L}_{\mathcal{G}_{k-1}^n}$ denotes identity of the $\mathcal{G}_{k-1}^n$-conditional distributions. The rest of Step 2 from the proof of Theorem \ref{theo_norm_2d_uncor} also applies here without modification.
\end{proof}

\begin{proof}[Proof of Corollary \ref{cor_norm_almost_poshom}]
It suffices to prove
\begin{align*}
\overline{V}^*(p_1+p_2,f,\pi_n)_T-\overline{V}^*(p_1+p_2,\tilde{f},\pi_n)_T \overset{\mathbb{P}}{\longrightarrow} 0.
\end{align*}
For $\varepsilon>0$ pick $\delta>0$ such that $|L(x)-1|<\varepsilon$ for all $x$ with $\|x\| \in [0,\delta]$. Then it holds
\begin{align}\label{cor_norm_almost_poshom_proof1}
&\big| \overline{V}^*(p_1+p_2,f,\pi_n)_T-\overline{V}^*(p_1+p_2,\tilde{f},\pi_n)_T \big| \nonumber
\\&~=\big|n^{(p_1+p_2)/2-1}\sum_{i,j:t_{i,n}^{(1)} \vee t_{j,n}^{(2)} \leq T} (1-L(\Delta_{i,n}^{(1)}\widetilde{X}^{(1)},\Delta_{j,n}^{(2)}\widetilde{X}^{(2)} )) \nonumber
\\&~~~~~~~~~~~~~~~~~~~~~~~~~~~~~~~~~~~~~~~~~~~~~~~~~~\times f(\Delta_{i,n}^{(1)}\widetilde{X}^{(1)},\Delta_{j,n}^{(2)}\widetilde{X}^{(2)} )\mathds{1}^{*,n}_{d_1,d_2}(i,j) \big| \nonumber
\\&~ \leq \varepsilon \overline{V}^*(p_1+p_2,|f|,\pi_n)_T
+n^{(p_1+p_2)/2-1}\sum_{i,j:t_{i,n}^{(1)} \vee t_{j,n}^{(2)} \leq T} |1-L(\Delta_{i,n}^{(1)}\widetilde{X}^{(1)},\Delta_{j,n}^{(2)}\widetilde{X}^{(2)} )| \nonumber
\\&~~~~~~~~~~~~~\times|f(\Delta_{i,n}^{(1)}\widetilde{X}^{(1)},\Delta_{j,n}^{(2)}\widetilde{X}^{(2)} )|\mathds{1}_{\{\|(\Delta_{i,n}^{(1)}\widetilde{X}^{(1)},\Delta_{j,n}^{(2)}\widetilde{X}^{(2)} )\|>\delta\}}\mathds{1}^{*,n}_{d_1,d_2}(i,j).
\end{align}
The function $|f|:x \mapsto |f(x)|$ is positively homogeneous of degree $p_1$ in the first argument and positively homogeneous of degree $p_2$ in the second argument. Hence using Proposition \ref{lemma_norm_sync} and
\begin{align*}
&\mathbb{E}[\overline{C}^*(p_1+p_2,|f|,\pi_n)_T|\mathcal{S}]
\\&~~\leq  \sum_{i,j:t_{i,n}^{(1)} \vee t_{j,n}^{(2)} \leq T} K\mathbb{E}[|\Delta_{i,n}^{(1)}\widetilde{C}^{(1)}|^{p_1}|\Delta_{j,n}^{(2)}\widetilde{C}^{(2)} |^{p_2}|\mathcal{S}]\mathds{1}^{*,n}_{d_1,d_2}(i,j)
\\&~~\leq K \widetilde{G}^{(d_1,d_2),n}_{p_1,p_2}(T),
\end{align*}
which follows by \eqref{ineq_pos_hom_abs_2d}, we obtain $\overline{V}^*(p_1+p_2,|f|,\pi_n)_T=O_\mathbb{P}(1)$ as $n \rightarrow \infty$ and hence the first term in \eqref{cor_norm_almost_poshom_proof1} vanishes as $n \rightarrow \infty$ and then $\varepsilon \rightarrow 0$. If $\widetilde{X}$ is continuous, then the second term in \eqref{cor_norm_almost_poshom_proof1} converges almost surely to $0$ as $n \rightarrow \infty$. If $\widetilde{X}$ may be discontinuous we have $p_1+p_2<2$. We then denote by $\Omega(n,q,N,\delta)$ the set where $\|(\Delta_{i,n}^{(1)}\widetilde{X}^{(1)},\Delta_{j,n}^{(2)}\widetilde{X}^{(2)} )\|>\delta$ implies $\Delta_{i,n}^{(1)}N^{(1)}(q)\neq 0$ or $\Delta_{j,n}^{(2)}N^{(2)}(q)\neq 0$, it holds $\|\Delta \widetilde{X}_s \|\leq N$ for all $s \in [0,T]$ and $\|(\Delta_{i,n}^{(1)}\widetilde{X}^{(1)},\Delta_{j,n}^{(2)}\widetilde{X}^{(2)} )\|\leq 2N$ for all $i,j$ with $t_{i,n}^{(1)} \vee t_{j,n}^{(2)} \leq T$. On this set the second term in \eqref{cor_norm_almost_poshom_proof1} is by the local boundedness of $L$ bounded by
\begin{multline*}
n^{(p_1+p_2)/2-1}\sum_{i,j:t_{i,n}^{(1)} \vee t_{j,n}^{(2)} \leq T} K |f(\Delta_{i,n}^{(1)}\widetilde{X}^{(1)},\Delta_{j,n}^{(2)}\widetilde{X}^{(2)} )|
\\ \times\mathds{1}_{\{\|(\Delta_{i,n}^{(1)}\widetilde{N}^{(1)}(q),\Delta_{j,n}^{(2)}\widetilde{N}^{(2)}(q) )\| \neq \emptyset\}}\mathds{1}^{*,n}_{d_1,d_2}(i,j)
\end{multline*}
which vanishes due to $p_1+p_2<2$ using \eqref{ineq_pos_hom_abs_2d} and the arguments used for the discussion of \eqref{hoelder_trick} in the proof of Proposition \ref{lemma_norm_sync}. The proof is then finished by observing $\mathbb{P}(\Omega(n,q,N,\delta))\rightarrow 1$ as $n,q,N \rightarrow \infty$ for any $\delta >0$. 
\end{proof}

\bibliographystyle{chicago}
\bibliography{bibliography}

\end{document}